\documentclass[11pt]{amsart}
\usepackage[
paper=a4paper ,
  hmarginratio={1:1},     % equal left and right margins
  vmarginratio={1:1},     % equal top and bottom margins
]{geometry}

\usepackage{graphicx}
\usepackage[hidelinks]{hyperref}
\usepackage{amsrefs}
\usepackage{amsmath,amssymb,mathtools,mathrsfs,amsthm}
\usepackage{tikz}
\usepackage{enumerate}

\usepackage[margin=10px]{caption}

\newtheorem{theorem}{Theorem}[section]
\newtheorem*{theorem*}{Theorem}
\usepackage[section]{placeins}
\newtheorem{proposition}[theorem]{Proposition}
\newtheorem{corollary}[theorem]{Corollary}
\newtheorem{lemma}[theorem]{Lemma}
\newtheorem{conjecture}[theorem]{Conjecture}

\theoremstyle{definition}
\newtheorem{definition}[theorem]{Definition}

\newtheorem{remark}[theorem]{Remark}

\numberwithin{equation}{section}
\definecolor{myo}{RGB}{255,196,0}
\definecolor{mpurp}{RGB}{255,0,255}
\DeclareMathOperator{\sing}{sing}
\DeclareMathOperator{\st}{st}
\DeclareMathOperator{\topp}{top}
\DeclareMathOperator{\im}{im}
\title[{Distances between surfaces in $4$-manifolds}]{Distances between surfaces in $4$-manifolds}
\author{Oliver Singh}
\address{Department of Mathematical Sciences, Durham University, United Kingdom.}
\email{\href{mailto:oliver.singh@durham.ac.uk} {oliver.singh@durham.ac.uk}}

\begin{document}
\begin{abstract}
If $\Sigma$ and $\Sigma'$ are homotopic embedded surfaces in a $4$-manifold then they may be related by a regular homotopy (at the expense of 
introducing double points) or by a sequence of stabilisations and destabilisations (at the expense of adding genus).  This naturally gives 
rise to two integer-valued notions of distance between the embeddings: the singularity distance $d_{\sing}(\Sigma,\Sigma')$ and the stabilisation 
distance $d_{\st}(\Sigma,\Sigma')$. Using techniques similar to those used by Gabai in his proof of the 4-dimensional light-bulb theorem, we prove that $d_{\st}(\Sigma,\Sigma')\leq d_{\sing}(\Sigma,\Sigma')+1$.
\end{abstract}
\maketitle
\section{Introduction}\label{sec:intro}
%There are two notions of distances for knotted surfaces of genus $g$ in a 4-manifold defined by Juh{\'{a}}sz and Zemke \cite{Juhasz2018}. We will recall their definitions, but for our result, it suffices to only allow unknotted stabilisations.
Let $X$ be a smooth, compact, orientable 4-manifold, possibly with boundary. Let~$\Sigma$,~$\Sigma'$ be connected, oriented, compact, smooth, properly embedded surfaces in $X$. We say that $\Sigma'$ is a \emph{stabilisation} of $\Sigma$ if there is an embedded solid tube $D^1\times D^2\subset X$ such that~$\Sigma\cap (D^1\times D^2)=\{0,1\}\times D^2$, and $\Sigma'$ is obtained from $\Sigma$ by removing these two discs and replacing them with $D^1\times S^1$, as in Figure~\ref{fig:astab}, and then smoothing corners. In this situation we say that $\Sigma$ is a \emph{destabilisation} of $\Sigma'$.

\begin{figure}[htb]

\begin{tikzpicture}
\node[anchor=south west,inner sep=0] (image) at (0,0) {\includegraphics[width=0.7\textwidth]{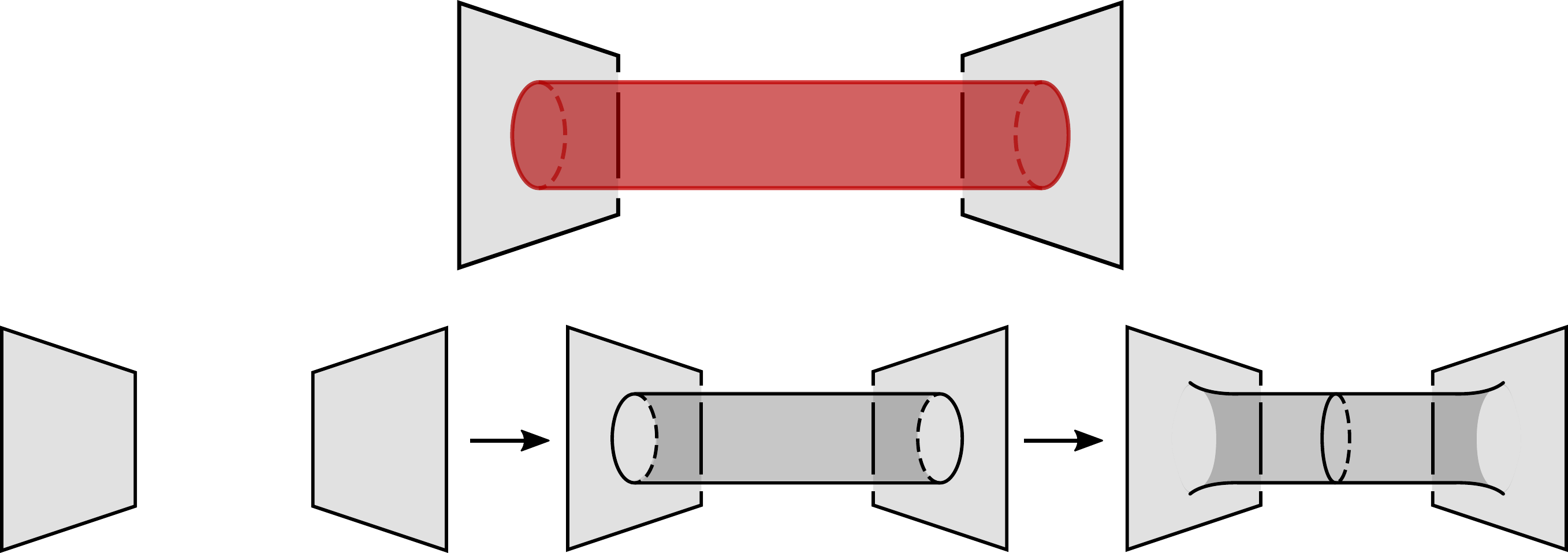}};
  \begin{scope}[x={(image.south east)},y={(image.north west)}]
    \node  at (0.5,0.75)  {$D^1\times D^2$};
    \end{scope}
\end{tikzpicture}
\caption{\label{fig:astab}A stabilisation. Given $D^1\times D^2\subset X$ which intersects~$\Sigma$ on~$S^0\times D^2$, we remove the two discs $S^0\times D^2$, add the tube $D^1\times S^1$, then smooth corners.}
\end{figure}

\begin{definition} Given $\Sigma$, $\Sigma'$ as above, both of genus $g$, define the \emph{stabilisation distance} between $\Sigma$ and $\Sigma'$ to be
$$d_{\st}(\Sigma,\Sigma')=\min_{\mathbb{S}}\left|\max\{g(P_1),\ldots ,g(P_k)\}-g\right|,$$ 
where $\mathbb{S}$ is the set of sequences $P_1,\ldots,P_k$ of connected, oriented, embedded surfaces where $\Sigma=P_1$, $\Sigma'=P_k$ and $P_{i+1}$ differs from $P_{i}$ by one of, i) stabilisation, ii) destabilisation, or iii) ambient isotopy. If no such sequence exists we declare $d_{\st}(\Sigma,\Sigma')=\infty$.
\end{definition}
\begin{definition}Given an immersion $f\colon S\rightarrow X$, define $\sing(f)\subset X$ to be the set of double points. Define the \emph{singularity distance} to be 
$$d_{\sing}(\Sigma,\Sigma')=\tfrac{1}{2} \min_{H} \max_{t\in [0,1]} \left|\sing(H_t)\right|,$$
where the minimum is taken over all generic regular homotopies $H$ from $\Sigma$ to $\Sigma'$. If no such regular homotopy exists we declare $d_{\sing}(\Sigma,\Sigma')=\infty$.
\end{definition}

\begin{remark}

It is possible that $d_{\st}(\Sigma,\Sigma')<\infty$ and $d_{\sing}(\Sigma,\Sigma')=\infty$. For example let~$U\subset S^3$ be the unknot, let~$\Sigma=U\times S^1\subset S^3\times S^1$, and let $\Sigma'$ be a trivially embedded torus in $S^3\times \{\text{pt}\}\subset S^3\times S^1$. Since $\Sigma$ and $\Sigma'$ are not homotopic $d_{\sing}(\Sigma,\Sigma')=\infty$.\\ Further~$\Sigma$ and $\Sigma'$ are related by a destabilisation then a stabilisation, so $d_{\st}(\Sigma,\Sigma')=1$.

\end{remark}
The main theorem of this paper is the following:
\begin{theorem}\label{ineq}
If $\Sigma,\Sigma'\subset X$ are connected, smooth, properly embedded, oriented surfaces of the same genus then
$$d_{\st}(\Sigma,\Sigma')\leq d_{\sing}(\Sigma,\Sigma')+1.$$
\end{theorem}
The proof is constructive, in that given a regular homotopy with at most $2n$ double points at each time, we construct a sequence of stabilisations and destabilisations such that the maximum genus occurring is at most $g+n+1$. Combining this with the fact that embedded surfaces are regularly homotopic if and only if they are homotopic, see Remark \ref{htpyrmk}, we give a new proof of a certain case of Baykur and Sunukjian's result \cite[Theorem~1]{InancBaykur2015}.

\begin{corollary}\label{bscor}
If $\Sigma,\Sigma'\subset X$ are homotopic, connected, smooth, properly embedded, oriented surfaces, then they differ by a sequence of stabilisations, destabilisations, and ambient isotopy.
\end{corollary}

We do not know whether the $+1$ is really essential. As we discuss in Remark \ref{improvremark}, it does not seem obvious how to adapt the proof to prove the inequality without the +1. It would, however, be interesting to find an example where the sequence of stabilisations constructed in the proof is minimal.

\begin{conjecture}\label{plus1conjecture}
There exists a smooth, compact, orientable 4-manifold $X$ and smooth, properly embedded, orientable, regularly homotopic surfaces $\Sigma$ and $\Sigma'$ with $$d_{\st}\left(\Sigma,\Sigma'\right)=d_{\sing}\left(\Sigma,\Sigma'\right)+1.$$
\end{conjecture}

Juh{\'{a}}sz and Zemke define invariants in \cite{Juhasz2018} which bound below the minimum of two similarly defined distances, $\mu_{\st}$ and $\mu_{\sing}$ using their notation. Their singularity distance $\mu_{\sing}$ is the same as $d_{\sing}$ defined here, however $\mu_{\st}\leq d_{st}$, as they allow more general stabilisations and destabilisations. We currently do not know of any invariants that can bound $d_{\st}$ below without also bounding~$d_{\sing}$ below.

The techniques we use to turn regular homotopies into sequences of stabilisations are inspired by those used by Gabai \cite{Gabai2017} to prove that if two homotopic embedded surfaces have a common transverse sphere in $X$, where $\pi_1(X)$ has no 2-torsion, then the surfaces are ambiently isotopic (and so both the distances defined here are $0$). A transverse sphere for a given embedded sphere $S\subset X$ is another embedded sphere $P\subset X$ such that $P$ intersects $S$ transversely at a single point.

Miller \cite{Miller2019} also recently proved, under the assumption that one of the surfaces has a transverse sphere, and that $\pi_1(X)$ has no 2-torsion, that regularly homotopic surfaces are concordant. There are also modified statements of these theorems when $X$ has 2-torsion.

Schwartz found infinitely many examples \cite{Schwartz2018} which demonstrate that the assumption on $\pi_1(X)$ was essential. She produced 4-manifolds $X$ with $\pi_1(X)\cong \mathbb{Z}_2$, and pairs of homotopic embedded 2-spheres in $X$ which share a transverse sphere but are not concordant. We would be interested to know the stabilisation and double point distances of these examples.

Schneiderman and Teichner \cite{Schneiderman2019} recently reproved the 4-dimensional light bulb theorem and characterised the situation for general $\pi_1(X)$ using an obstruction of Freedman and Quinn.

\subsection{The topologically flat case}
The condition of smoothness above was required to define regular homotopy. We define regular homotopy in a topological 4-manifold $X$ by saying two locally flat surfaces $\Sigma$, $\Sigma'\subset X$ are \emph{regularly homotopic} if they differ by a sequence of finger moves, Whitney moves (where the Whitney discs are locally flatly embedded), and ambient isotopy.

We also define stabilisation in the same way, dropping the smoothness condition on the embedding of $D^1\times D^2$. Using the topological definitions, we define~$d_{\st}^{\topp}$ and $d_{\sing}^{\topp}$ analogously. Repeating the proof of Theorem~\ref{ineq}, without the initial smoothness assumption yields:
\begin{theorem}
If $\Sigma$ and $\Sigma'$ are orientable, compact, connected, locally flat, properly embedded surfaces in $X$ of the same genus, then,
$$d_{\st}^{\topp}(\Sigma,\Sigma')\leq d_{\sing}^{\topp}(\Sigma,\Sigma')+1.$$
\end{theorem}

As a consequence, we prove an analogy of Corollary~\ref{bscor} for locally flat surfaces.
\begin{corollary}
Let $\Sigma$ and $\Sigma'$ be orientable, compact, connected, locally flat, properly embedded surfaces in $X$. If $\Sigma$ and $\Sigma'$ are regularly homotopic (topologically), then they differ by a sequence of topological stabilisations, destabilisations, and ambient isotopy.
\end{corollary}
\subsection{Plan of the paper.} The proof of Theorem~\ref{ineq} will be set out as follows
\begin{itemize}
\item \textbf{Section \ref{sec:reghtpy}:} We recall standard definitions and results for regular homotopy of surfaces in 4-manifolds. We also define surface tubing diagrams, which are similar to those defined by Gabai \cite[Definition 5.5]{Gabai2017}. These contain the data to replace an immersed surface with an embedded one by removing discs around pairs of intersection points and adding in tubes which run along the surface and join up the resulting boundary circles.
\item \textbf{Sections \ref{sec:tubeswap} and \ref{sec:tubemove}:} We prove the \emph{tube swap lemma}, Lemma~\ref{tubeswap}, and the \emph{tube move lemma}, Lemma~\ref{tubemove}. These prove that we can change the way in which the surface is tubed (i.e.\ change the surface tubing diagram) by performing a single stabilisation followed by a single destabilisation.
\item \textbf{Section \ref{sec:mainproof}:} We associate to a regular homotopy a sequence of stabilisations and destabilisations by \emph{shadowing}, again using the terminology of Gabai \cite{Gabai2017}, the homotopy by a tubed surface as follows.
\begin{itemize}
\item At each stage in the regular homotopy where a finger move is performed, we remove the pair of double points by performing a stabilisation which adds a tube running along the surface.
\item At each stage in the regular homotopy where a Whitney move is performed, we wish to perform a destabilisation which removes a tube. However, this may not always immediately be possible for several reasons. Firstly, pairs of double points removed by the Whitney move may not have been tubed to each other, but rather to other double points. Secondly, the tubes may not run over one of the arcs of the Whitney circle. Thirdly, the pairing may be `crossed' in the terminology of Gabai \cite[$\mathsection 3.3$]{Gabai2017}. We use the results proved in Sections~\ref{sec:tubeswap} and \ref{sec:tubemove} to change the way that the surface is tubed to eliminate the first two difficulties. We reduce the third difficulty to the fact that two surfaces $K, K'\subset B^4$ which are slice surfaces for the unknot, both have the property that they destabilise to the standard disc bounded by the unknot. 

\end{itemize}
\item \textbf{Section \ref{sec:destabilising}:} We prove that the pair of slice surfaces $K$ and $K'$ found Section~4 both destabilise to the standard disc bounded by the unknot.
\end{itemize}

\subsection*{Acknowledgements}

Thanks to Andras Juh{\'{a}}sz for proposing this question, and for the interesting and informative discussion that led me to think about it. Thanks to Stefan Friedl, Maggie Miller, and Matthias Nagel for their interest and helpful comments. Thanks also to Andrew Lobb and Mark Powell for their guidance and for many useful conversations. This research was supported by the EPSRC Grant no.\ EP/N509462/1 project no.\ 1918079, awarded through Durham University.

\section{Regular homotopy and surface tubing diagrams}\label{sec:reghtpy}
We recall some standard definitions, and define surface tubing diagrams, which describe how to turn immersed surfaces of genus $g$ with $2n$ double points, into embedded surfaces of genus $g+n$.  
\subsection{Regular homotopy}
Recall that embedded surfaces $\Sigma,\Sigma'\subset X$ are \emph{regularly homotopic} if there exists a smooth map $H \colon S\times [0,1]\rightarrow X$ (where $S$ is an abstract surface), where $H(-,0)$ and $H(-,1)$ are embeddings with $H(S,0)=\Sigma$, $H(S,1)=\Sigma'$, and $H(-,t)$ is an immersion for all $t\in [0,1]$. In the case $X$ and $S$ have boundary, we require the embeddings be proper embeddings, and that $H(t,\partial S)\subset\partial X$ for all $t$.

\begin{remark}
By theorems of Smale \cite{Smale1957} and Hirsch \cite{Hirsch1959} two embeddings of an orientable surface in an orientable 4-manifold are regularly homotopic if and only if they are homotopic. Note that Smale-Hirsch theory gives much more general results about homotopy classes of embedded manifolds. We refer the reader for example to Miller's treatment in this specific case \cite[Theorem 3.3]{Miller2019}.\label{htpyrmk} %and the Euler class of their normal bundles agree. 
%\itemAs a consequence if $\Sigma,\Sigma'\subset X$ are embedded surfaces, they are homotopic if and only if they are regularly homotopic. Note that if they are homotopic they live in the same homology class and so have the same self-intersection number. The Euler number of the normal bundle is equal to the self-intersection number for an embedding, and so it follows the that surfaces are regularly homotopic.
\end{remark}

We recall the following definitions from \cite{Freedman1990}.
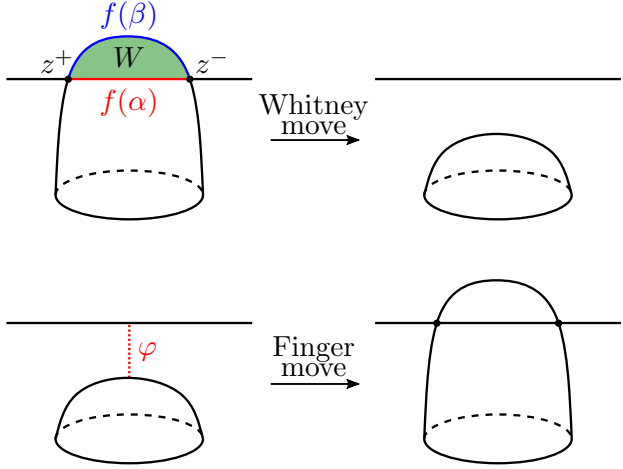
\begin{figure}[!ht]
   \centering
\begin{minipage}[c]{0.55\textwidth}
\begin{tikzpicture}
\node[anchor=south west,inner sep=0] (image) at (0,0) {\includegraphics[width=\textwidth]{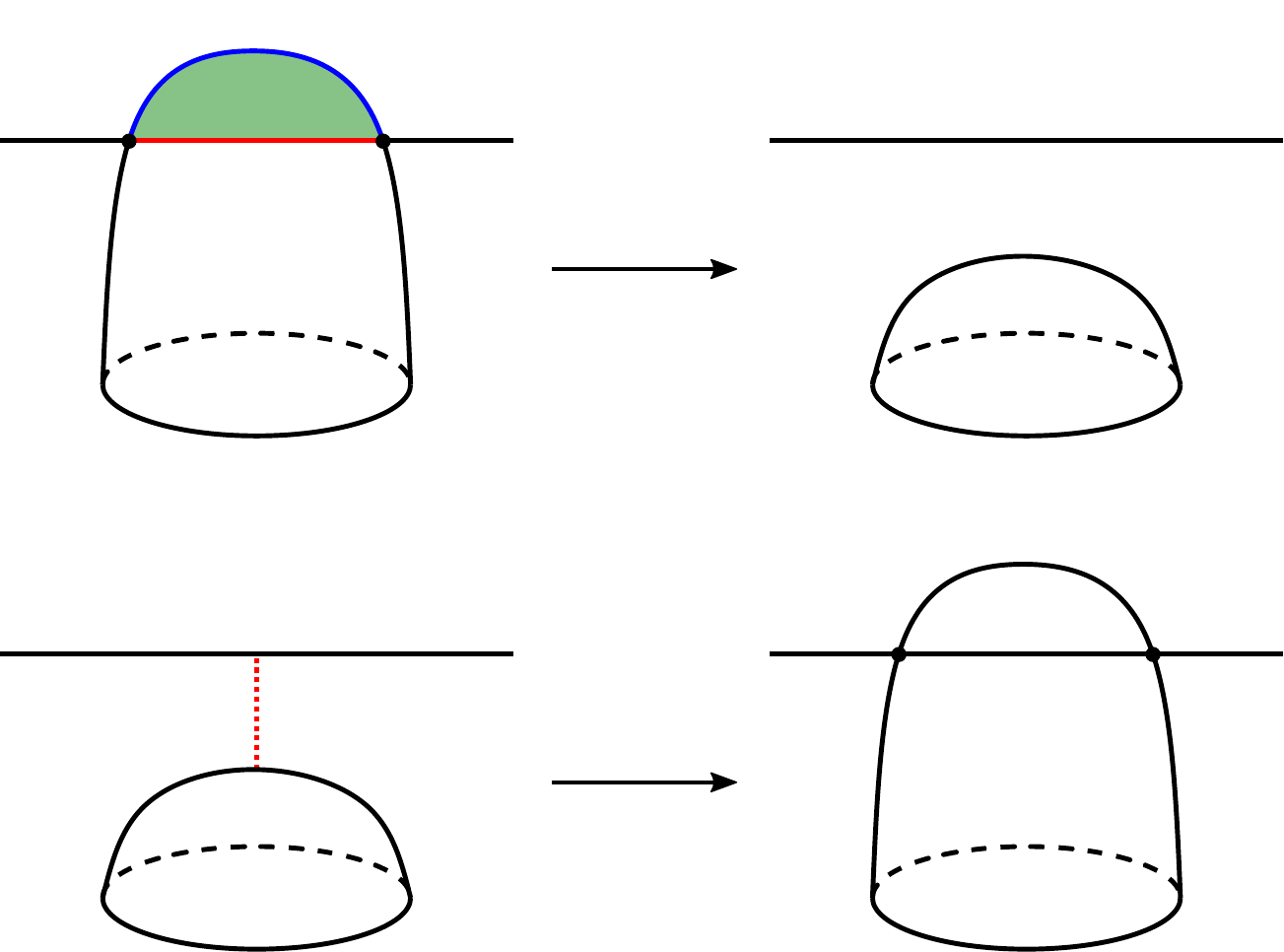}};
  \begin{scope}[x={(image.south east)},y={(image.north west)}]
%  \draw[step=0.1,gray,very thin] (0,0) grid (1,1);
    \node[text=red] at (0.2,0.8) {$f(\alpha)$};
    \node[text=blue] at (0.2,0.99) {$f(\beta)$};
    \node at (0.083,0.89) {$z^+$};
    \node at (0.335,0.89) {$z^-$};
    \node at (0.2,0.9) {$W$};
    \node[red] at (0.23,0.25) {$\varphi$};
    \node at (0.5,0.79) {Whitney};
    \node at (0.5,0.744) {move};
    \node at (0.5,0.29-0.034) {Finger};
    \node at (0.5,0.244-0.034) {move};
    \end{scope}
\end{tikzpicture}\end{minipage}%
\begin{minipage}[c]{0.45\textwidth}
\caption{\label{fig:whitneydef} The time slice where the finger moves and Whitney moves occur. Note that the horizontal line continues into the past and future, and is unchanged by the homotopies.}\end{minipage}
\end{figure}
\begin{definition}Let $S$ be a compact, oriented surface, and $f\colon S\rightarrow X$ a generic self transverse immersion. Suppose $z^+, z^-\in X$ are a pair of double points of $f$ where the sign of the self intersection at $z^\pm$ is $\pm 1$. Suppose $W\subset X$ is an embedded disc such that~$W\cap f(S)=\partial W$, and the preimage $f^{-1}(\partial W)$ is two disjoint embedded arcs~$\alpha,\beta\subset S$. Note that~$\mathring{\alpha}$ and~$\mathring{\beta}$, the interiors of the arcs, are disjoint from the preimages of  double points of $S$. We call $W$ a \emph{Whitney disc}, and $\alpha$ and $\beta$ the \emph{Whitney arcs}. We require that the canonical framing of $\nu(W)$ restricts on $\partial W$ to the Whitney framing of $\nu (W)|_{\partial W}$. Here the \emph{Whitney framing} is determined by a section of $\nu (W)|_{\partial W}$ which is tangent to $S$ along one Whitney arc, and normal to $S$ along the other Whitney arc. The \emph{Whitney move} along $W$ is the regular homotopy which moves the sheet containing $\beta$ along $W$ to remove the two self intersections; see Figure~\ref{fig:whitneydef}. See \cite{Freedman 1990} for further details.
\end{definition}
\begin{definition}Let $S$ be a compact, oriented surface, and $f:S\rightarrow X$ a self transverse immersion. Let $\varphi\subset X$ be an embedded arc such that~$\varphi\cap f(S)=\partial\varphi$, with $\partial\varphi$ disjoint from double points of $f(S$). The \emph{finger move} along $\varphi$ is the regular homotopy which drags one disc of $f(\nu( f^{-1}(\partial\varphi)))$ along $\varphi$ to create a pair of intersections in the surface with opposite signs; again see Figure~\ref{fig:whitneydef}.
\end{definition}
\begin{remark}\leavevmode
\begin{enumerate}\item Two surfaces are regularly homotopic if and only if they differ by a sequence of finger moves and Whitney moves. Moreover, a generic regular homotopy is a sequence of finger moves and Whitney moves. See \cite{Freedman1990} for a detailed treatment. 
\item Finger moves and Whitney moves are inverse operations; a finger move gives rise to a Whitney disc that undoes the finger move, and a Whitney move gives rise to an arc that undoes the Whitney move via a finger move; again see \cite{Freedman1990}.
\end{enumerate}
\end{remark}

\begin{definition} Given an arc $\Gamma$ in the plane $\mathbb{R}^2\times{0}\subset\mathbb{R}^2\times\mathbb{R}^2$, the \emph{linking annulus} of~$\Gamma$ is the annulus $\Gamma\times S^1\subset \mathbb{R}^2\times \mathbb{R}^2$. 
\end{definition}
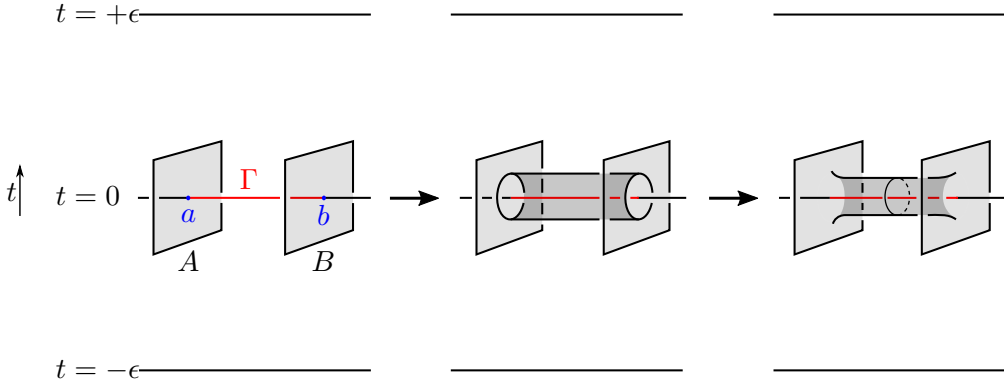
\begin{figure}[!ht]
\centering
\begin{tikzpicture}
\node[anchor=south west,inner sep=0] (image) at (0,0) {\includegraphics[width=0.9\textwidth]{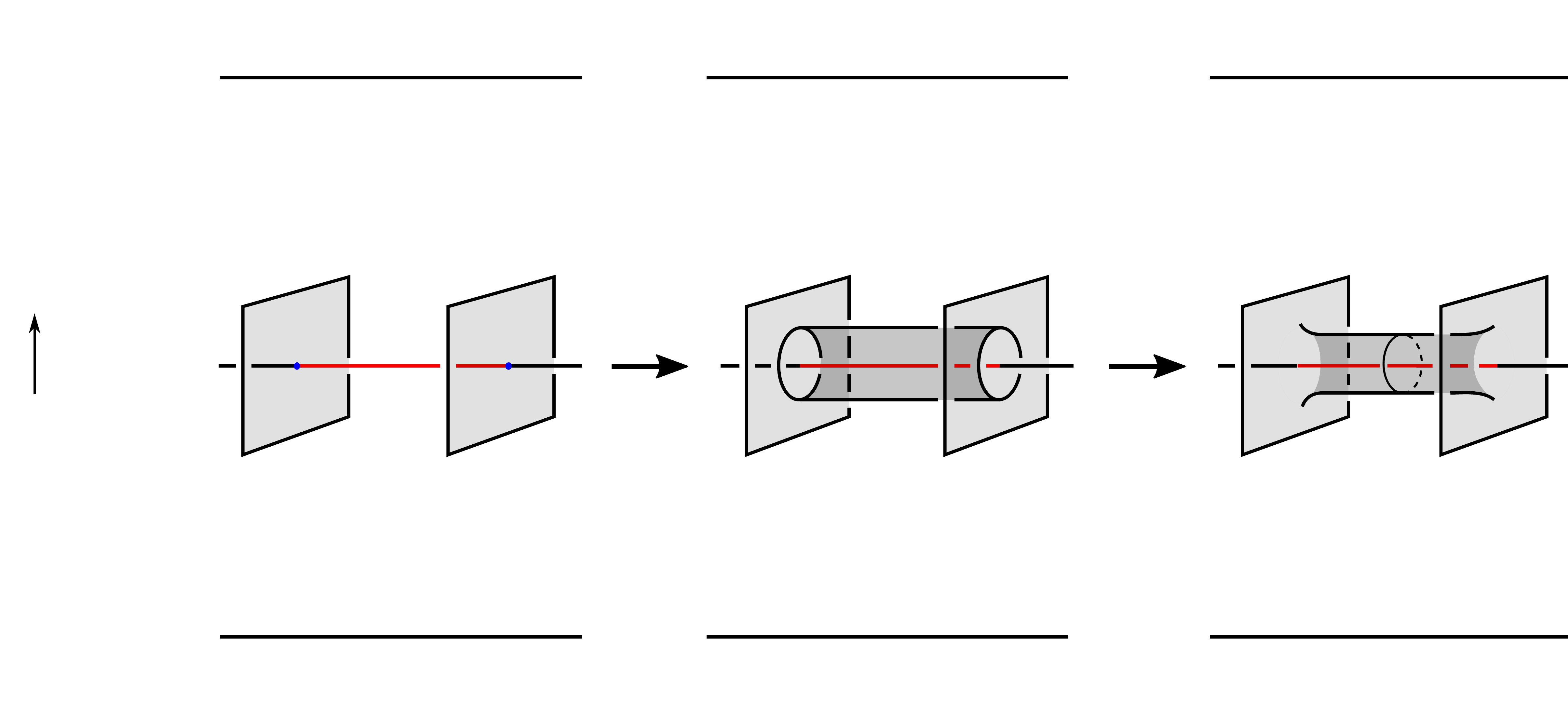}};
  \begin{scope}[x={(image.south east)},y={(image.north west)}]
      \node[anchor=west] at (0.048,0.89) { $t=+\epsilon$};
  \node[anchor=west]  at (0.048,0.493) { $t=0$};
    \node[anchor=west]  at (0.048,0.11) { $t=-\epsilon$};
    \node[text=blue] at (0.19,0.45) { $a$};
    \node[text=blue] at (0.324,0.45) { $b$};
    \node[anchor=west] at (-0.001,0.5) { $t$};
    \node at (0.19,0.35) {$A$};
    \node at (0.324,0.35) {$B$};
    \node[text=red] at (0.25,0.53) {$\Gamma$};
    \end{scope}
\end{tikzpicture}
\caption{\label{fig:surfacesums}A neighbourhood $\nu(\Gamma)$ in which the tubing operation takes place. The time direction is depicted upwards and we see the sheet of the surface containing~$\Gamma$ continues into the past and future. Tubing along $\Gamma$ removes two intersections by removing discs from $A$ and $B$, and adding the linking annulus of the arc $\Gamma$.}
\end{figure}
\begin{definition}\label{tubing} Suppose $A$, $B$, and $C$ are immersed surfaces in $X$, which intersect each other and themselves transversely only in double points. Suppose that $a\in A\cap C$ and $b\in B\cap C$. Suppose that $\Gamma$ is an embedded arc in $C$ (whose interior is disjoint from $A$ and $B$, and from other double points of $C$). Then \emph{tubing} along $\Gamma$ is the result of removing a disc around $a$ from $A$, a disc around $b$ from $B$, and adding the linking annulus of $\Gamma$ (parametrising $\nu(\Gamma)\cong B^4\subset X$ as $\mathbb{R}^2\times\mathbb{R}^2$ so that $C\cap\nu(\Gamma)$ is the plane~$\mathbb{R}^2\times{0}$), then smoothing corners; see Figure~\ref{fig:surfacesums}.
\end{definition}

\begin{remark}
Usually when we tube, $A$, $B$, and $C$ will, in fact, be subsurfaces of the same connected surface in $X$, obtained by taking the intersection of the surface with some ball $B^4\subset X$. Note that the resulting surface, in this case, is oriented if and only if the self-intersections $a$ and $b$ have opposite sign.
\end{remark}

\begin{remark} By convention, the middle time picture in movies such as Figure~\ref{fig:surfacesums} will be referred to as the $t=0$ slice or \emph{the present}.
\end{remark}

\subsection{Tubed Surfaces}

We shall describe how to replace self transverse immersed surfaces of genus $g$ in $X$, with $2n$ double points, with embedded surfaces of genus $g+n$ by pairing up double points with opposite sign, and tubing along an arc between them. We make the following definition, analogous to that made by Gabai \cite[Definition~5.5]{Gabai2017}.

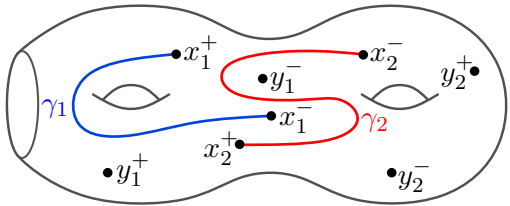
\begin{figure}[ht!] 
\begin{minipage}[c]{0.45\textwidth}
   \begin{tikzpicture}
\node[anchor=south west,inner sep=0] (image) at (0,0) {\includegraphics[width=\textwidth]{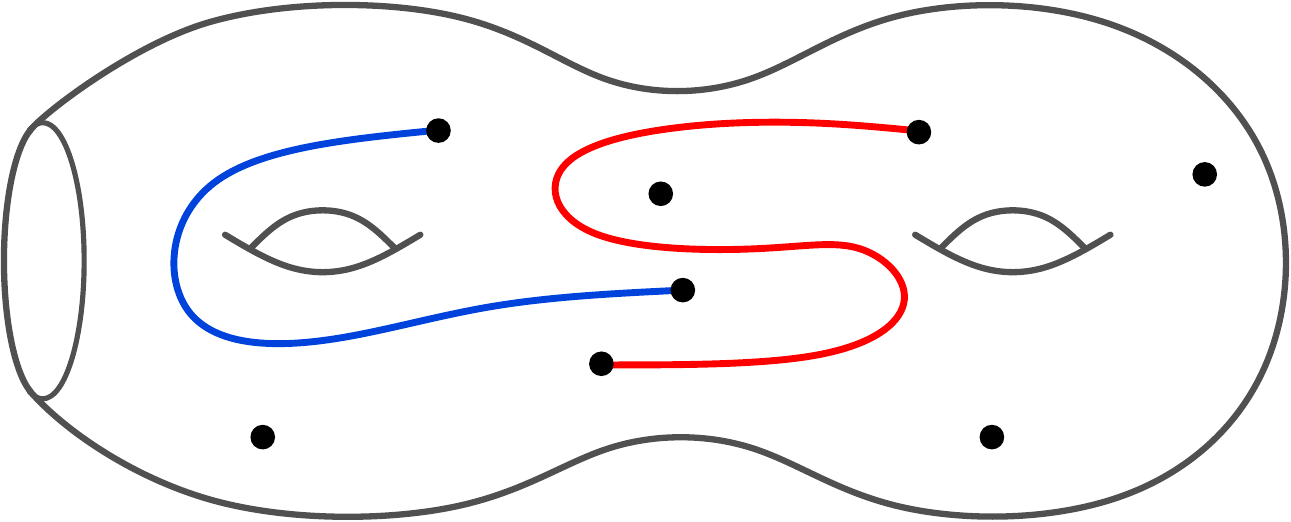}};
  \begin{scope}[x={(image.south east)},y={(image.north west)}]
    \node[text=blue] at (0.10,0.48) {$ \gamma_1$};
    \node[text=red] at (0.735,0.41) {$ \gamma_2$};
    \node at (0.39,0.76) {$x_1^+$};
    \node at (0.58,0.44) {$x_1^-$};
    \node at (0.255,0.18) {$y_1^+$};
    \node at (0.56,0.64) {$y_1^-$};
    
    \node at (0.76,0.755) {$x_2^-$};
    \node at (0.43,0.285) {$x_2^+$};
    \node at (0.899,0.66) {$y_2^+$};
    \node at (0.814,0.15) {$y_2^-$};
    \end{scope}
\end{tikzpicture}
\end{minipage}\hfill
\begin{minipage}[c]{0.55\textwidth}
      \caption{A surface tubing diagram. We depict the preimage $S$ of the embedding, and the two preimages of each double point, for example, $x_1^+$ and $y_1^+$ are the preimages of the double point $z_1^+\in X$. We also show the arcs $\gamma_1$ and $\gamma_2$, along which we tube to construct the associated tubed surface $\widetilde{S}$. }\label{fig:surfacediagram}\end{minipage}
\end{figure}

\begin{definition} Let $X$ be a smooth 4-manifold. A \emph{surface tubing diagram} $\mathcal{S}$, consists of the following data:
\begin{enumerate}[(i)]
\item A compact, oriented, connected surface $S$ possibly with boundary.
\item A generic (only isolated double points), self transverse, immersion ${f\colon S\rightarrow X}$\linebreak with $\partial X\cap f(S)=f(\partial S)$ and with the same number of positive self-intersections as negative. We take care not to confuse the abstract surface $S$ with the immersed surface which is its image, $f(S)$, denoted $S^{\im}$.

\item A partition of the set of double point images, $\{z\in X\:\colon \:|f^{-1}(z)|=2\}$ into pairs~$\{z_i^+,z_i^-\}\subset X$,~$i=1 \ldots n$~with  the sign of the intersection $z^\pm_i$ being $\pm 1$. We denote the preimage of $z_i^+$ by the pair $(x_i^+,y_i^+)\in S\times S$, which comes with a choice of ordering. We denote the preimage of $z_i^-$ by the pair $(x_i^-,y_i^-)$. We refer to any pre-images of double points as \emph{marked points} of the diagram.

\item A set of disjoint embedded arcs $\gamma_1,\ldots ,\gamma_n\colon [0,1]\rightarrow S$ with endpoints $\gamma_i(0)=x_i^+$ and ${\gamma_i(1)=x_i^-}$ and which are disjoint from $\{x_j^\pm,y_j^\pm\}_{j\neq i}$ and from $y_i^\pm$ (note that their images $\Gamma_i=f( \gamma_i([0,1]))$ are also disjoint embedded arcs in $S^{\im}$); see Figure~\ref{fig:surfacediagram}.
\end{enumerate}
\end{definition}
\begin{remark}
In the topological case we make the same definition for~$f\colon S\rightarrow X$ which is an immersion obtained from some locally flat embedding $g\colon S\rightarrow X$ by applying some sequence of finger moves, Whitney moves, and ambient isotopy.
\end{remark}

\begin{definition}
Given a surface tubing diagram we construct the \emph{associated tubed surface} by tubing $S^{\im}$ to itself along each arc $\Gamma_i$ using tubes in a small neighbourhood of each $\Gamma_i$, as in Definition \ref{tubing} with $\Gamma =\Gamma_i$, $A=f(\nu (y_i^+))$, $B=f(\nu (y_i^-))$, and $C=f(\nu (\gamma_i))$; again see Figure \ref{fig:surfacesums}. Since $z_i^+$ and $z_i^-$ have opposite signs the result is an oriented embedded surface which we call $\widetilde{S}$.
%\begin{enumerate}[(1)]
%\item Consider a small tubular neighbourhood $\Gamma_i$ in $X$. Possibly after a small isotopy we may and shall assume that $\nu (\Gamma_i)$ intersects $S^{\im}$ on $f(\nu \left(\gamma_i)\right)\cup f\left(\nu ( y^+_i)\right)\cup f\left(\nu ( y^-_i)\right)$.
%\item Identify $\nu (\Gamma_i)$ with $[-1,1]^4$. Pick coordinates $(x,y,z,t)$ so that $f(\nu \left(\gamma_i)\right)$ is the plane $\{y=z=0\}$, $f\left(\nu (y^+_i)\right)$ is the plane $\{x=-1,\: t=0\}$ and $f\left(\nu (y^-_i)\right)$ is the plane $\{x=1,\:t=0\}$ (see Figure~\ref{fig:howtotube}).
%\item Form $\widetilde{S}$ by removing the subsets $\{t=0,\; x=\pm 1,\; 0<\sqrt{y^2+z^2}<1/2\}$ from $f\left(\nu (y^\pm_i)\right)\subset S^{\im}$ (i.e.\ we remove two discs from the sheets, but leave the points $z_i^\pm=(\pm 1,0,0,0)$ in the other sheet), then adding the annulus $\{t=0,\:\sqrt{y^2+z^2}=1/2\}$.
%\item Repeat the above for each $i$.
%\item After smoothing corners, $\widetilde{S}$ is a smooth 2-submanifold of genus $g(S)+n$.
%\end{enumerate}
\end{definition}

%\begin{figure}[h!]
%\begin{tikzpicture}
%\node[anchor=south west,inner sep=0] (image) at %(0,0) {\includegraphics[width=0.65\textwidth]{howtotube.pdf}};
%  \begin{scope}[x={(image.south east)},y={(image.north west)}]
%    \node[text=blue] at (0.092+0.04,0.42) { $z_i^+$};
%    \node[text=blue] at (0.334+0.03,0.42) { $z_i^-$};
%    \node[anchor=west] at (-0.01,0.5) { $t$};
%    \node at (0.092+0.04,0.225) {$f(\nu(y_i^+))$};
%    \node at (0.334+0.03,0.225) {$f(\nu(y_i^-))$};
%    \node[text=red] at (0.21+0.032,0.53) {$\Gamma_i$};
%    \end{scope}
%\end{tikzpicture}
%      \caption{\label{fig:howtotube}Creating the associated tubed surface $\widetilde{S}\subset X$ (right) by tubing along each arc $\Gamma_i$. In the terminology of Definition \ref{tubing}, the $B^4\subset X$ is a neighbourhood of $\Gamma_i$ in $X$, $A=f(\nu (y_i^+))$, $B=f(\nu (y_i^+))$, $\Gamma=\Gamma_i$ and $C=f(\nu (\gamma_i))$.}
%\end{figure}

\begin{remark}\label{isoremark}\leavevmode
\begin{enumerate}[(1)]
\item Ambient isotopy of the immersed surface $S^{\im}$ (which gives rise to an isotopy of the immersion $f$) describes an ambient isotopy of the arcs $\Gamma_i$, and so by extension an isotopy of the associated tubed surface $\widetilde{S}$, where we make sure to keep the tubes close to the surface. At the end of the isotopy, we still have a tubed surface and we still have a surface tubing diagram with the same arcs (but with immersion data~$H(-,1)\circ f$, where~$H\colon S\times [0,1]\rightarrow X$ is the ambient isotopy).\label{isoremark1}
\item An isotopy of the set of arcs $\gamma_i$ in $S$ (which keeps the arcs disjoint from the preimages of self-intersection points throughout the isotopy) gives rise to an isotopy of the tubes and hence of the associated tubed surface.\label{isoremark2}
\end{enumerate}
\end{remark}

\section{The Tube Swap Lemma}\label{sec:tubeswap}
We prove that if we change our surface tubing diagram as depicted in Figure~\ref{fig:tubeswap}, then there is a single stabilisation followed by a single destabilisation taking one associated surface to the other.

The proof is mainly by pictures. We shall draw surfaces in a 4-ball as movies with a time direction, drawing slices of~$3$-dimensional space. For the sake of readability, we draw our pictures as piecewise-smooth and so they have corners, however, they should each be understood to describe a smooth surface. The corners arise in two ways, firstly from stabilisations and destabilisations, and secondly when part of the surface `jumps' into the time direction. The reader should mentally smooth these pictures. For stabilisations and destabilisations, the smoothing is as in Figure~\ref{fig:astab}. The corners arising from jumps into the time direction are locally modelled as a product of two arcs which are properly embedded in two discs, one arc having a corner. Smoothing the corner of the arc gives a smoothing of the surface.
\begin{lemma}[Tube Swap Lemma] Given a surface tubing diagram $\mathcal{S}$, let $\beta$ be any arc in $S$ from $y_i^+$ to $y_i^-$ which is disjoint from all points and curves in the surface diagram $($including $\gamma_i)$. Form $\mathcal{S}'$ by removing the arc $\gamma_i$ and replacing it with $\beta$ $($and changing the order of $(x_i^+,y_i^+)$ to $(y_i^+,x_i^+))$ and  $(x_i^-,y_i^-)$ to $(y_i^-,x_i^-))$; see Figure~\ref{fig:tubeswap}.

Then the associated tubed surface $\widetilde{S}$ can be transformed into $\widetilde{S}'$ by performing a single stabilisation, followed by a destabilisation $($and ambient isotopy$)$.\label{tubeswap}
\end{lemma}
\vspace{-3.4mm}
\begin{figure}[!h]
  \begin{minipage}[c]{0.55\textwidth}
\begin{tikzpicture}
\node[anchor=south west,inner sep=0] (image) at (0,0) {\includegraphics[width=\textwidth]{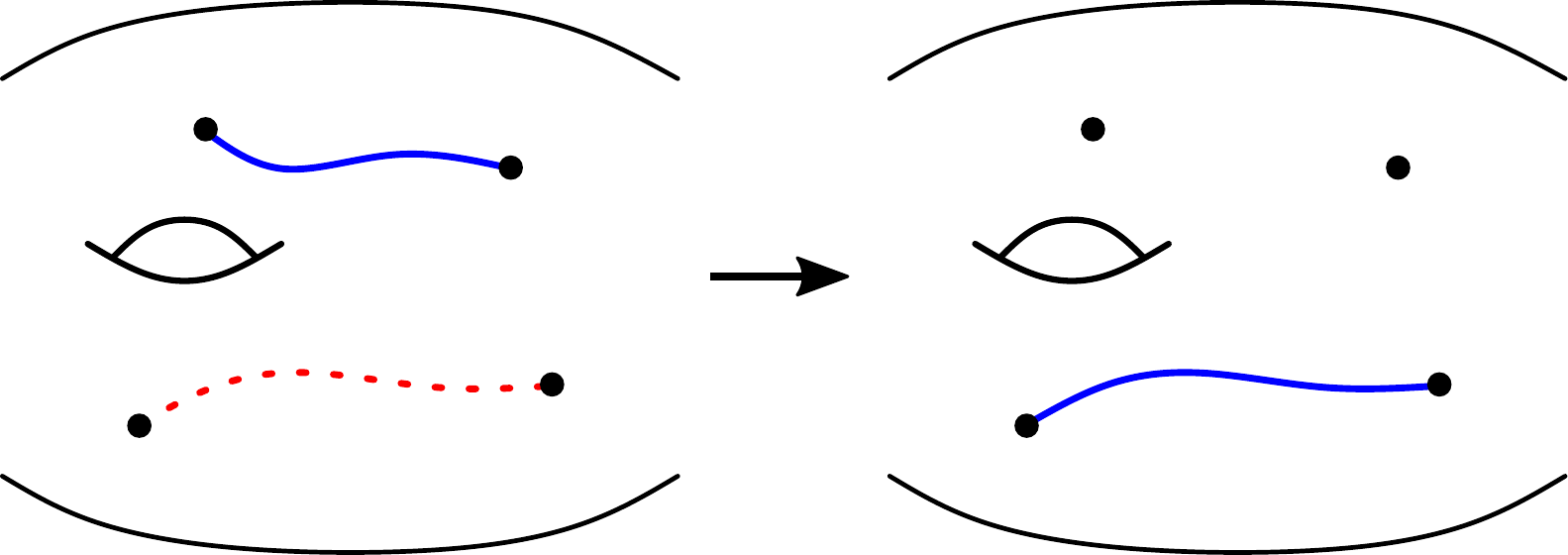}};
  \begin{scope}[x={(image.south east)},y={(image.north west)}]
  \node[text=blue] at (0.78,0.23) {$\beta$};
  \node[text=red] at (0.78-0.564,0.23) {$\beta$};
  \node[text=blue] at (0.23,0.625) {$\gamma_i$};
  \node at (0.1,0.77) {$x_i^+$};
  \node at (0.37,0.72) {$x_i^-$};
  \node at (0.06,0.23) {$y_i^+$};
  \node at (0.39,0.31) {$y_i^-$};
  \node at (0.1+0.564,0.77) {$x_i^+$};
  \node at (0.37+0.564,0.72) {$x_i^-$};
  \node at (0.06+0.564,0.23) {$y_i^+$};
  \node at (0.39+0.564,0.31) {$y_i^-$};
    \end{scope}
\end{tikzpicture}   
  \end{minipage}\hfill
  \begin{minipage}[c]{0.45\textwidth}
    \caption{\label{fig:tubeswap}The diagram for a tube swap: we swap the arc $\gamma_i$ for any arc $\beta$ from $y_i^+$ to $y_i^-$, disjoint from all arcs and points. The tube swap lemma says there is a stabilisation and destabilisation, taking the associated tubed surface of the left diagram, to the associated tubed surface of the right.  } 
  \end{minipage}
\end{figure}
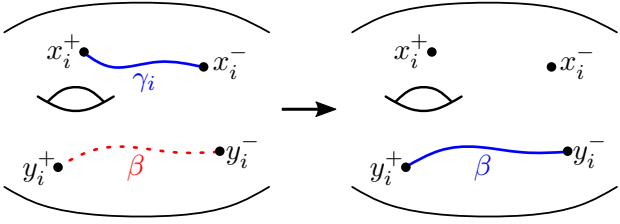
\vspace{-1.4mm}
\begin{proof}
We direct the reader to Figure~\ref{fig:tubeswap0}. We first consider a small tubular neighbourhood~$\nu (\Gamma_i\cup f(\beta))$~ in $X$, which is diffeomorphic to $S^1\times B^3$. We pick a diffeomorphism of~$\nu ( \Gamma_i)$ to~$B^4$, parametrising so that at~$t=0$ we see~$f(\nu ( \gamma_i))$ (the sheet of the immersed surface containing~$\Gamma_i$) and the ends of the arc~$f(\beta)$). The rest of the sheet containing the ends of $f(\beta)$ extends into the past and future. We extend our parametrisation to one of $S^1\times B^3$ so that at each $t$ we see a copy of $S^1\times B^2$, and so $f(\beta)$ is in the $t=0$ frame and the sheet containing it extends into the past and future; see Figure~\ref{fig:tubeswap0}.

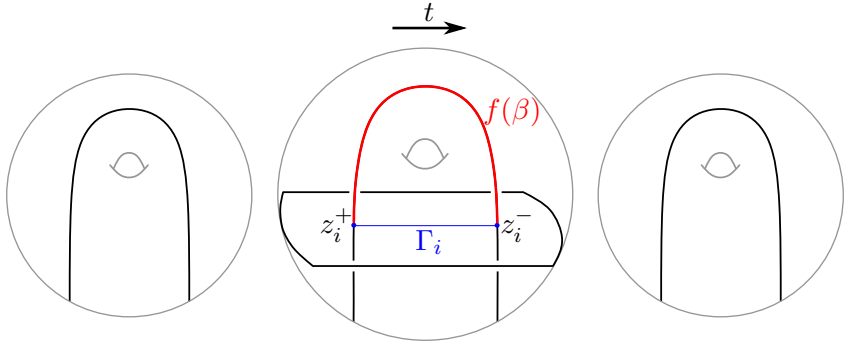
\begin{figure}[!h]
   \centering 
\begin{tikzpicture}
\node[anchor=south west,inner sep=0] (image) at (0,0) {\includegraphics[width=0.75\textwidth]{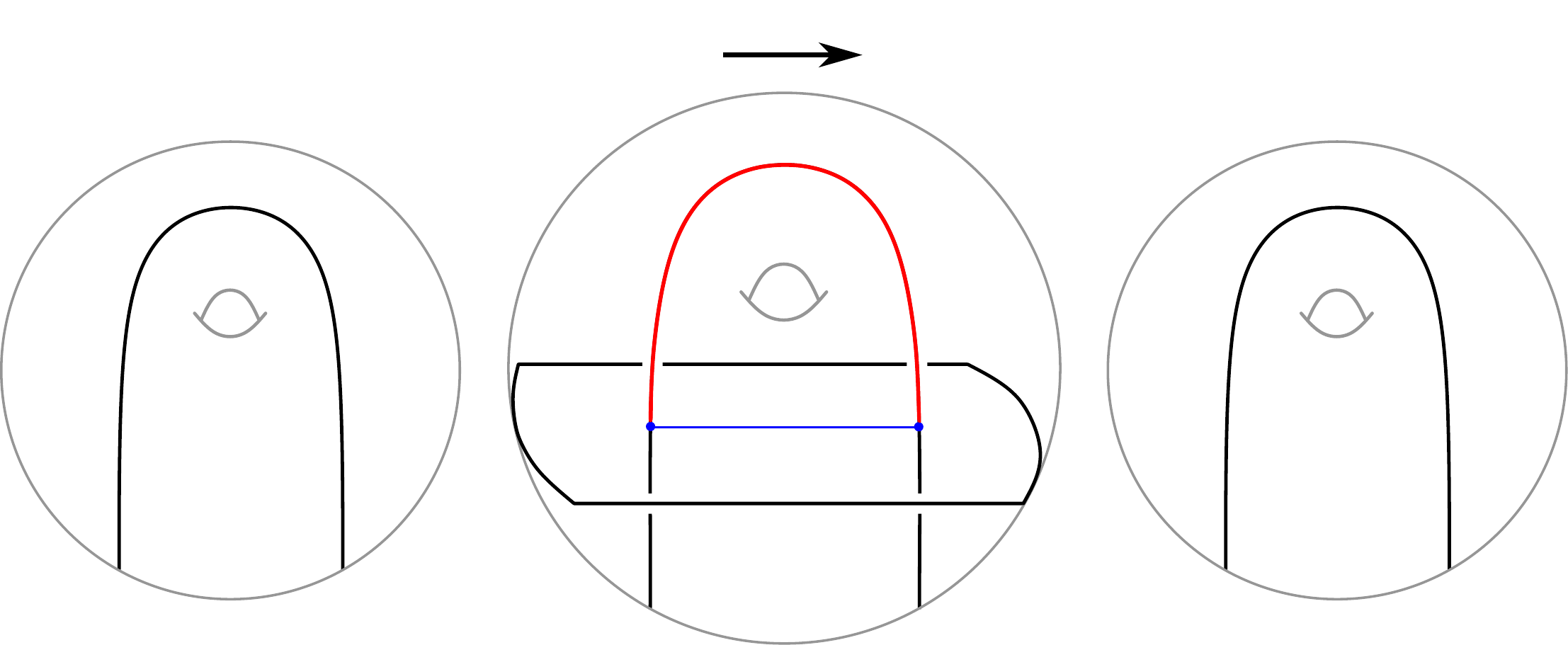}};
  \begin{scope}[x={(image.south east)},y={(image.north west)}]
  %\draw[step=0.1,gray,very thin] (0,0) grid (1,1);

    \node at (0.505,0.96) { $t$};
    \node[blue] at (0.505,0.29) { $\Gamma_i$};
    \node[red] at (0.604,0.67) { $f(\beta)$};
    \node at (0.395,0.34) { $z_i^+$};
    \node at (0.61,0.34) { $z_i^-$};
    \end{scope}
\end{tikzpicture}
      \caption{\label{fig:tubeswap0}The surface $S^{\im}$ in the neighbourhood $\nu ( \Gamma\cup f(\beta))\cong S^1\times B^3$. The middle picture is $t=0$. In each time frame we see a copy of $S^1\times D^2$.}
\end{figure}
\begin{figure}[p]
   \centering 
   \begin{tikzpicture}
\node[anchor=south west,inner sep=0] (image) at (0,0) {\includegraphics[width=0.9\textwidth]{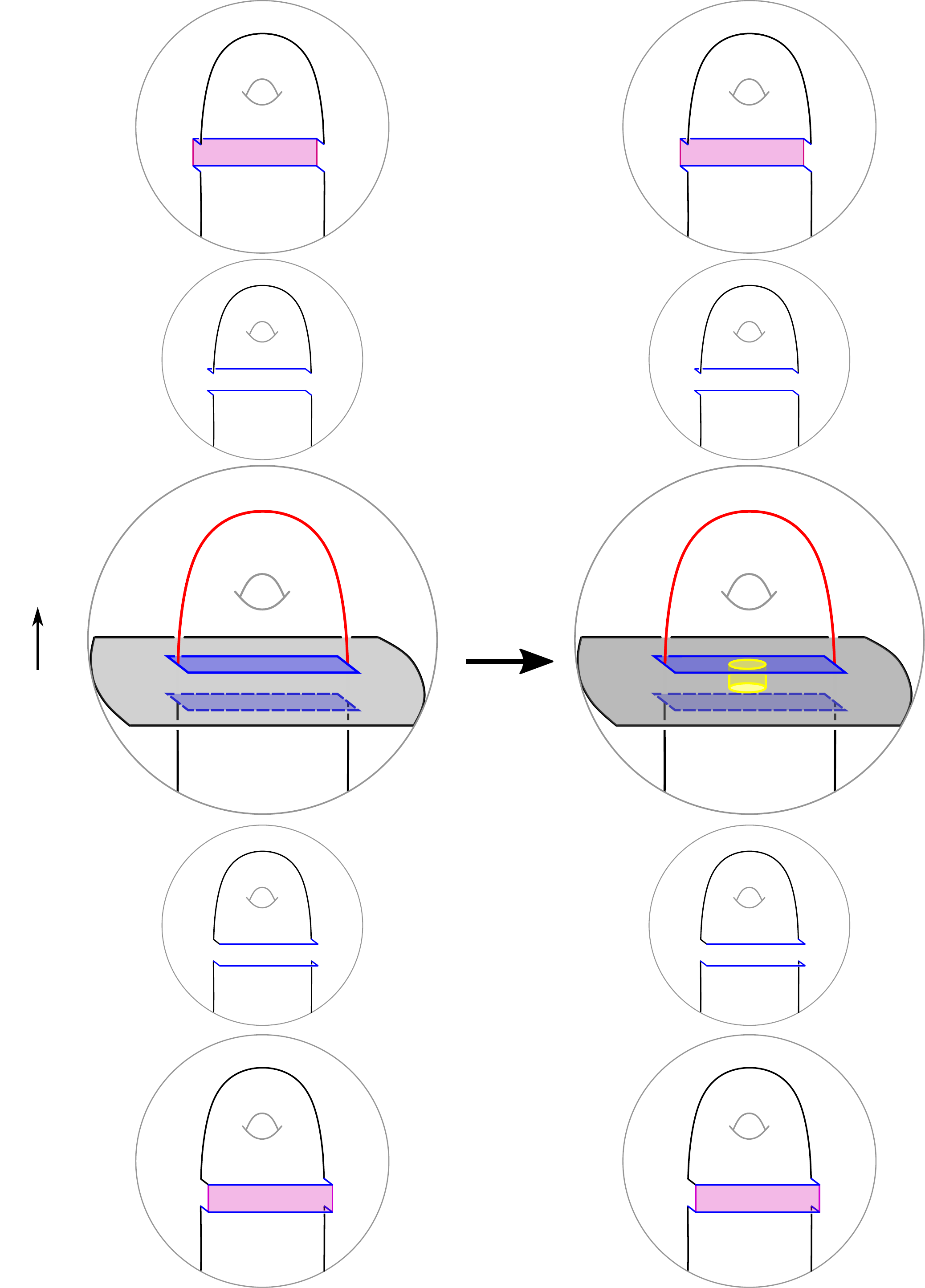}};
  \begin{scope}[x={(image.south east)},y={(image.north west)}]
    \node at (0.02,0.503) { $t$};
    \node at (0.283,-0.013) { A};
    \node at (0.814,-0.0128) { B};
    \end{scope}
\end{tikzpicture}
      \caption{\label{fig:tubeswap1}In A we depict the associated tubed surface $\widetilde{S}$ in a neighbourhood of~$X$ given by $\nu ( \Gamma\cup f(\beta))\cong S^1\times B^3$. To obtain B from A we perform a stabilisation from the tube to the sheet of the surface surface~$f(\nu ( \gamma_i))$.}
\end{figure}
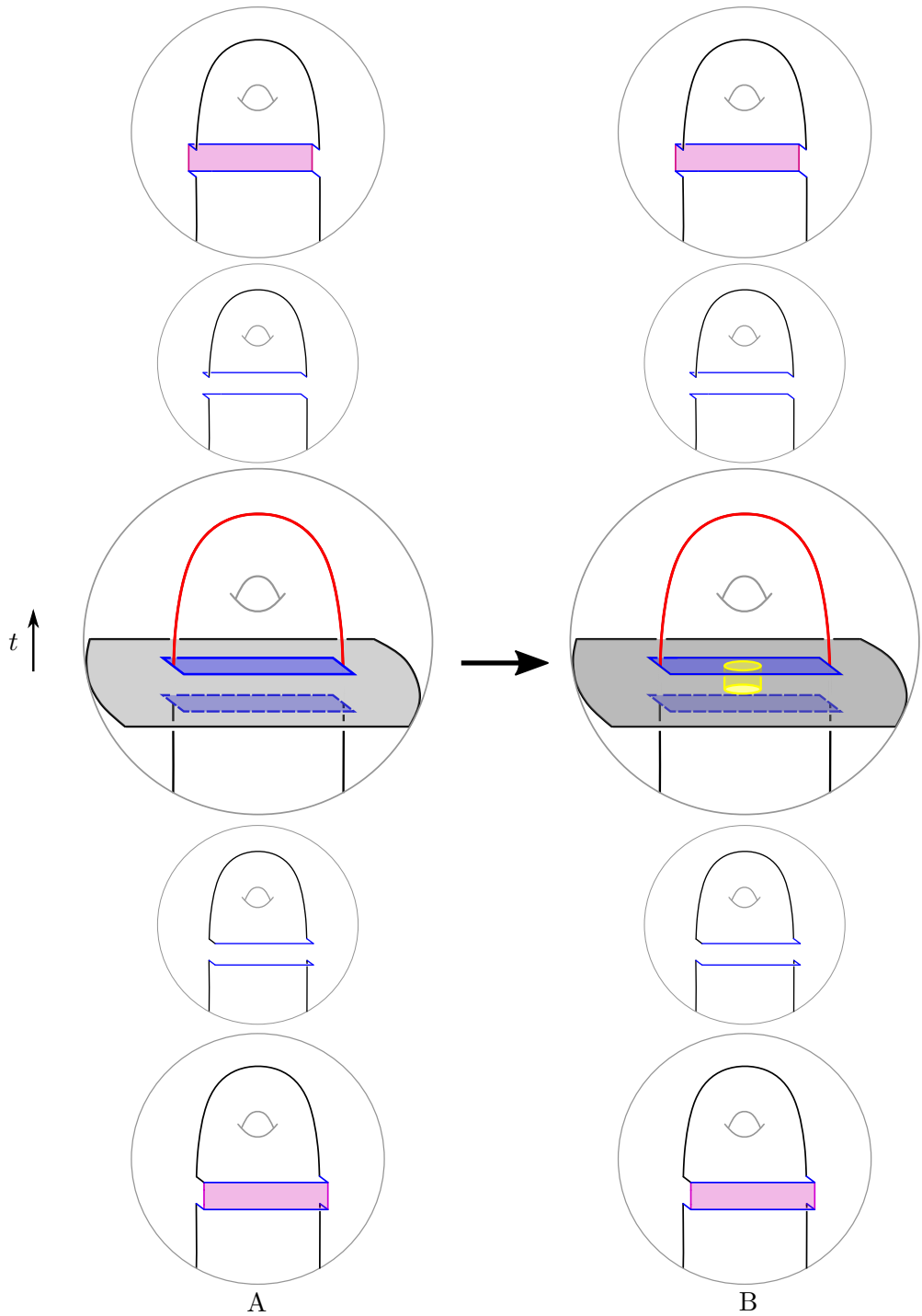
\begin{figure}[p]
   \centering 
   \begin{tikzpicture}
--\node[anchor=south west,inner sep=0] (image) at (0,0) {\includegraphics[width=0.9\textwidth]{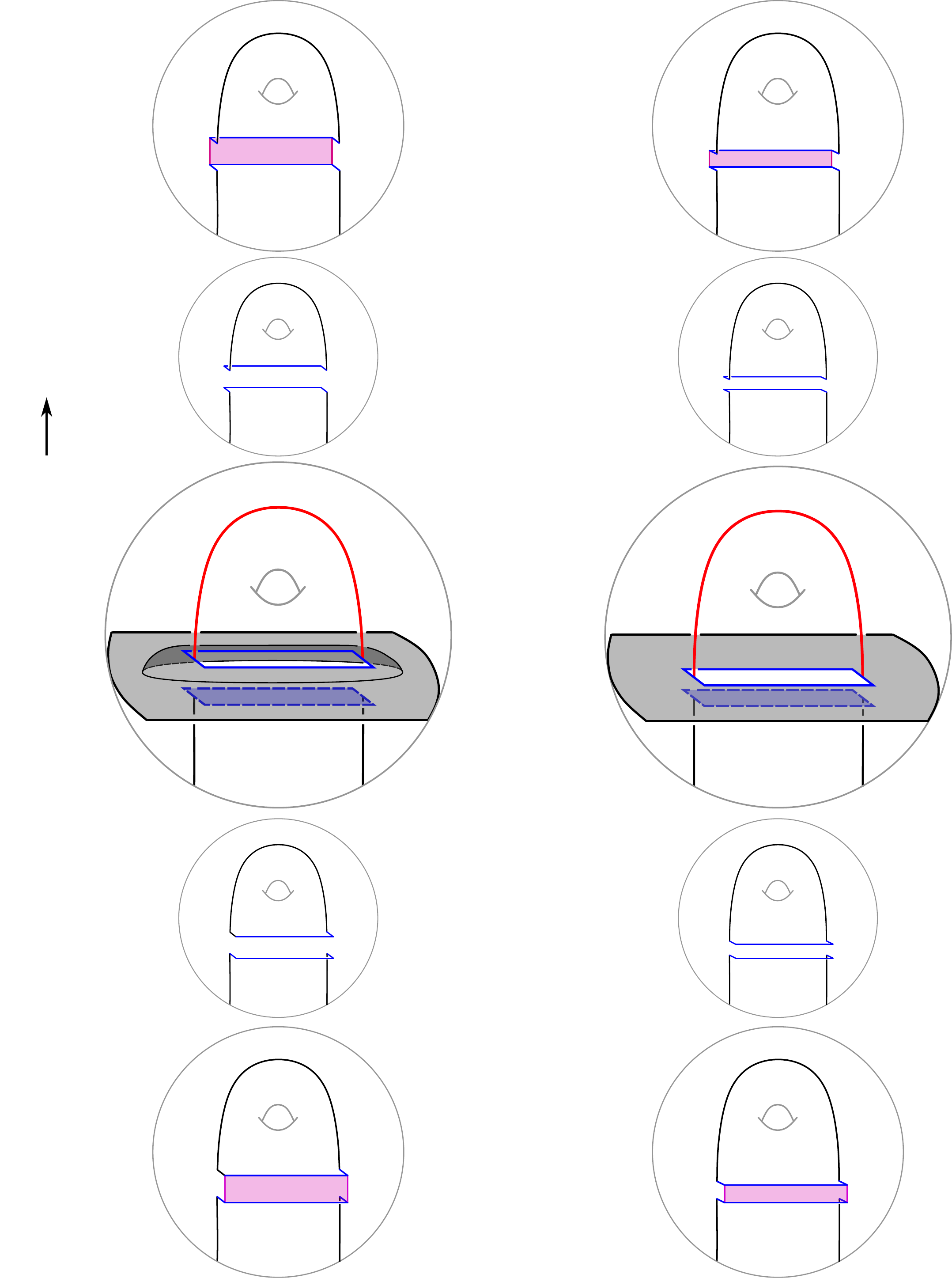}};
  \begin{scope}[x={(image.south east)},y={(image.north west)}]
    \node at (0.03,0.66) { $t$};
    \node at (0.04,0.5) {\Large $\sim$};
    \node at (0.56,0.5) {\Large $\sim$};
    \node at (0.283,-0.013) { C};
    \node at (0.814,-0.0128) { D};
    \end{scope}
\end{tikzpicture}
      \caption{\label{fig:tubeswap2}To obtain C from  Figure~\ref{fig:tubeswap1}.B we perform an isotopy which is supported in the~$t=0$ frame (`inflating' the tube from the stabilisation). To obtain D from C we then flatten the picture out, which slightly deforms the surface in other time frames.}
\end{figure}
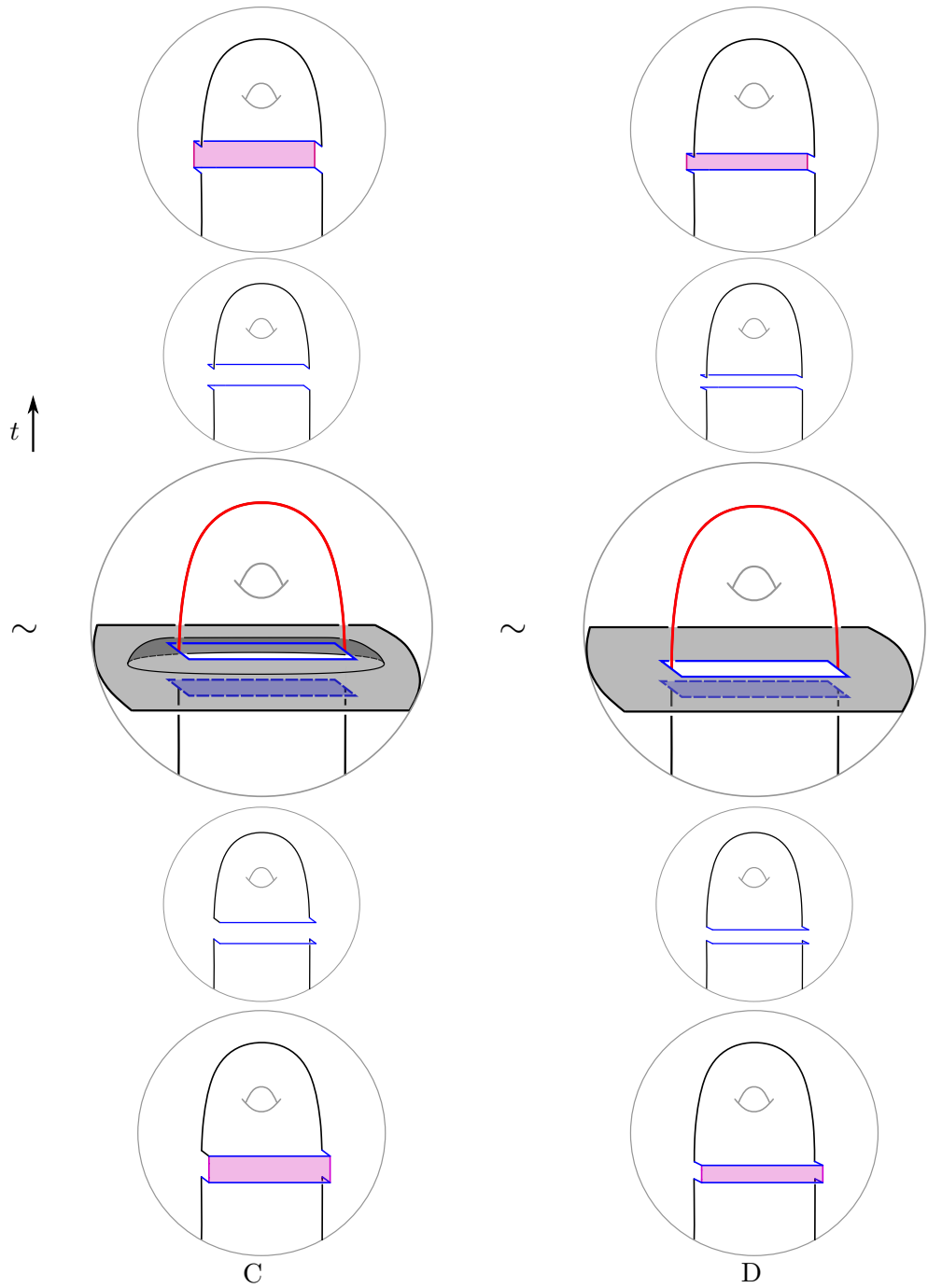
%\begin{figure}[!ht]
%   \centering 
%    \begin{tikzpicture}
%\node[anchor=south west,inner sep=0] (image) at (0,0) {\includegraphics[width=\textwidth]{finaltubeswap3.pdf}};
%  \begin{scope}[x={(image.south east)},y={(image.north west)}]
%    \node at (0.05,0.68) { $t$};
%    \end{scope}
%\end{tikzpicture}
%      \caption{\label{fig:tubeswap3}We push the sides of the tube into the $t=0$ frame, then flatten the resulting surface.}
%\end{figure}

%\begin{figure}[!ht]
%   \centering
%   \begin{tikzpicture}
%\node[anchor=south west,inner sep=0] (image) at (0,0) {\includegraphics[width=\textwidth]{finaltubeswap4.pdf}};
%  \begin{scope}[x={(image.south east)},y={(image.north west)}]
%    \node at (0.037,0.66) { $t$};
%    \end{scope}
%\end{tikzpicture}
%      \caption{\label{fig:tubeswap4}We `thicken' the red arc by pushing some of the surface from the future into the $t=0$ frame. We then perform an isotopy of this band, to form a tube with a band removed.}
%\end{figure}
\begin{figure}[p]
   \centering
   \begin{tikzpicture}
\node[anchor=south west,inner sep=0] (image) at (0,0) {\includegraphics[width=0.795\textwidth]{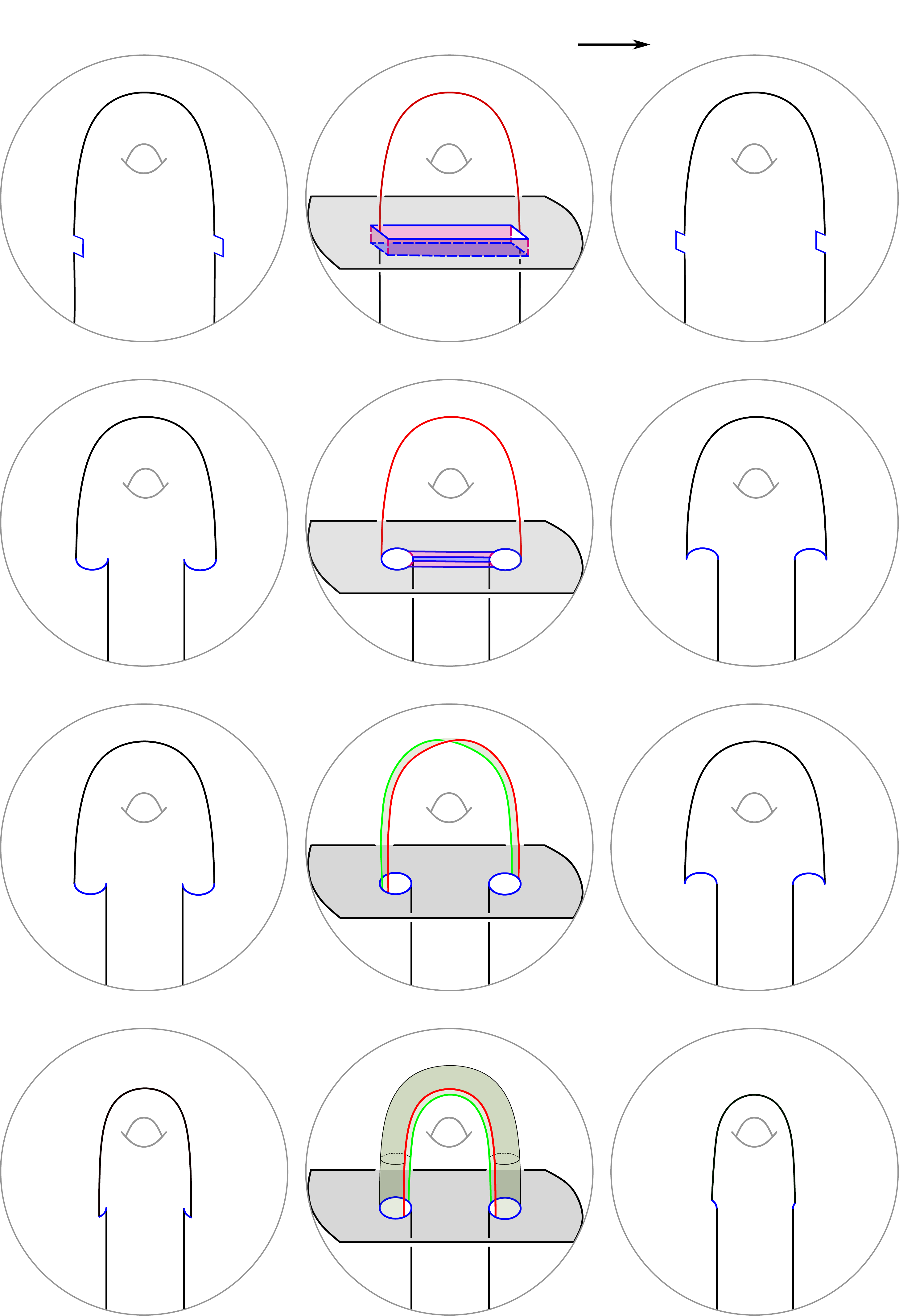}};
  \begin{scope}[x={(image.south east)},y={(image.north west)}]
%  \draw[step=0.1,gray,very thin] (0,0) grid (1,1);
    \node at (0.683,0.98) { $t$};
    \node[rotate=90] at (0.503,0.97) {\Large $\sim$};
    \node[rotate=90] at (0.503,0.48) {\Large $\sim$};
    \node[rotate=90] at (0.503,0.724) {\Large $\sim$};
    \node[rotate=90] at (0.503,0.232) {\Large $\sim$};
    \node[anchor=west] at (-0.06,0.848){ E:};
    \node[anchor=west] at (-0.06,0.602){ F:};   
    \node[anchor=west] at (-0.06,0.352){ G:};   
    \node[anchor=west] at (-0.06,0.102){ H:};
    \end{scope}
\end{tikzpicture}
      \caption{\label{fig:tubeswapA}Here the time direction is to the right, and we depict 4 stages of the isotopy. To obtain E from Figure~\ref{fig:tubeswap2}.D we push the sides of the tube into the $t=0$ frame. To obtain F we perform an isotopy to E which flattens the surface. To obtain G from E we then `thicken' the red arc by pushing some of the surface from the future into the $t=0$ frame. To obtain H from G we then perform an isotopy of this band, to form a tube with a band removed.}
\end{figure}
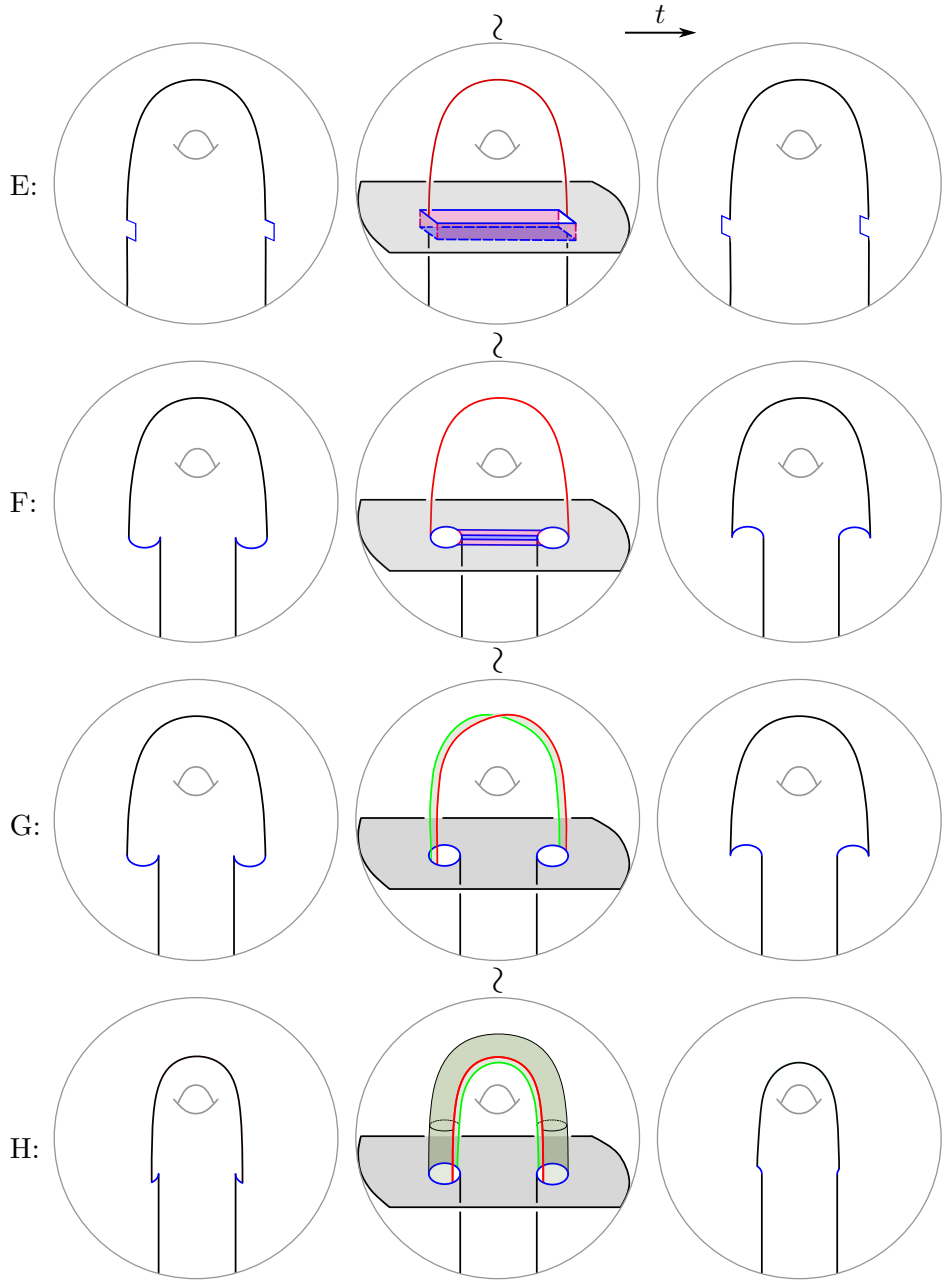
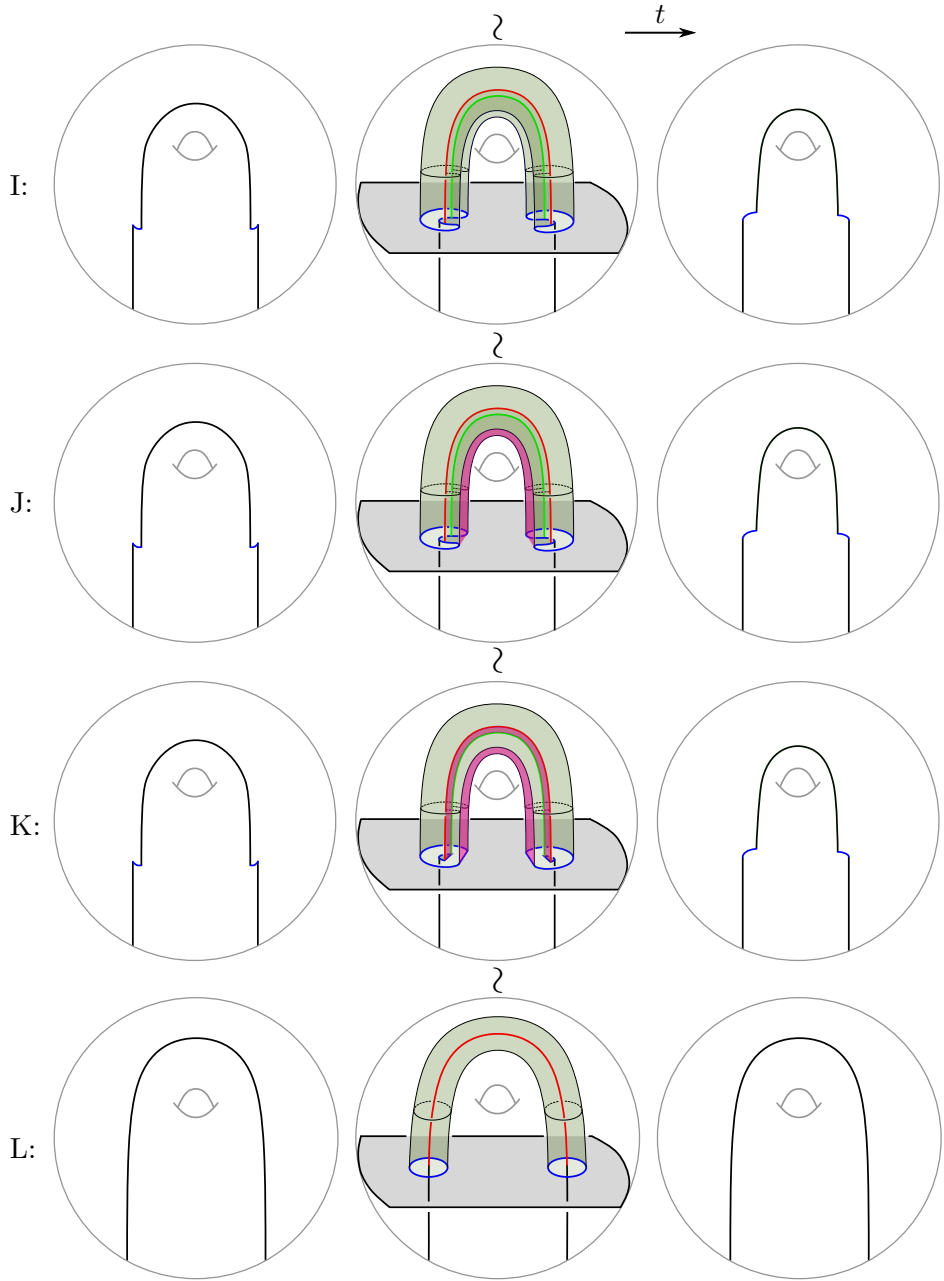
\begin{figure}[p]
   \centering
   \begin{tikzpicture}
\node[anchor=south west,inner sep=0] (image) at (0,0) {\includegraphics[width=0.795\textwidth]{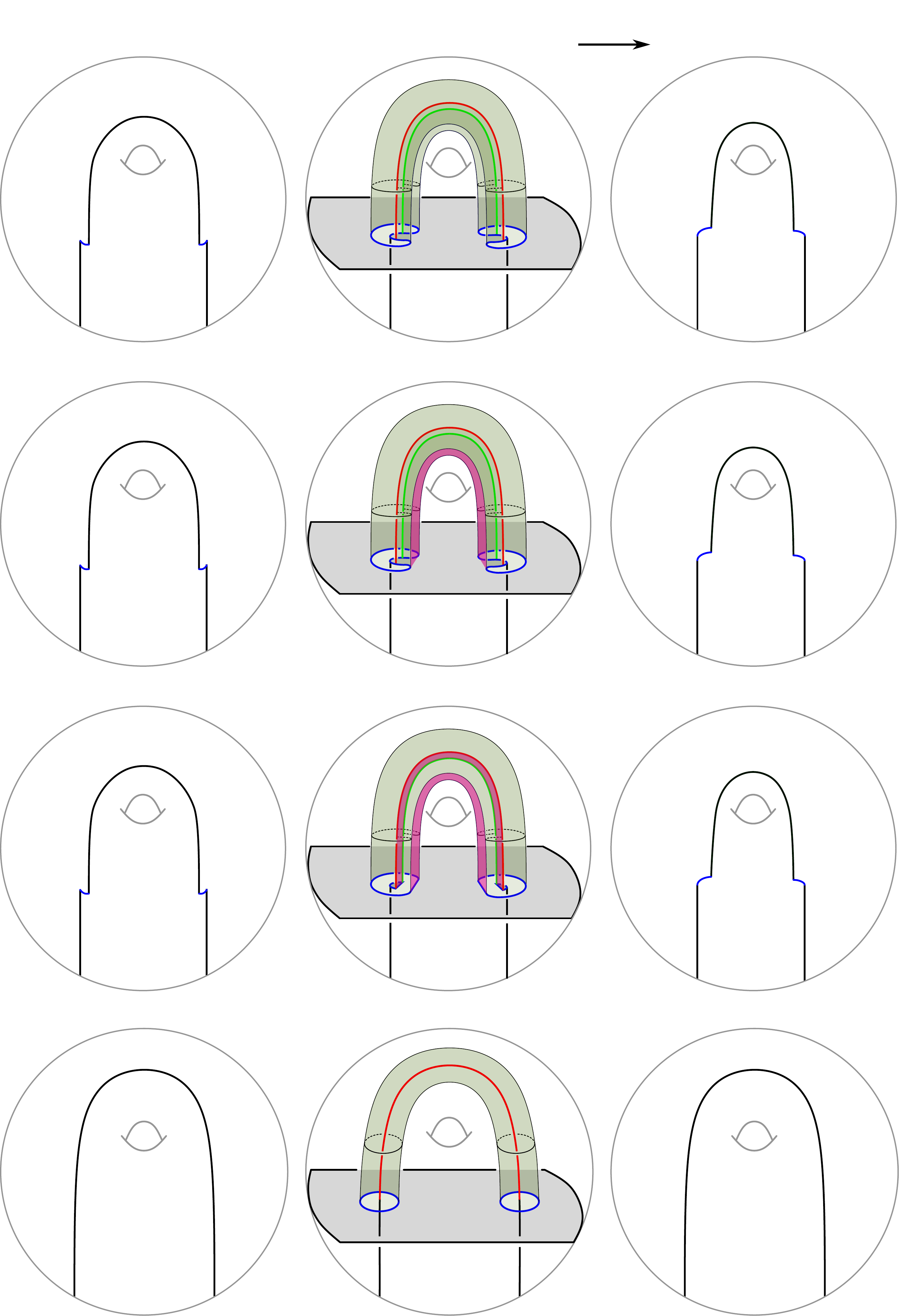}};
  \begin{scope}[x={(image.south east)},y={(image.north west)}]
%  \draw[step=0.1,gray,very thin] (0,0) grid (1,1);
    \node at (0.683,0.98) { $t$};
    \node[rotate=90] at (0.503,0.97) {\Large $\sim$};
    \node[rotate=90] at (0.503,0.48) {\Large $\sim$};
    \node[rotate=90] at (0.503,0.724) {\Large $\sim$};
    \node[rotate=90] at (0.503,0.232) {\Large $\sim$};
        \node[anchor=west] at (-0.06,0.848){ I:};
    \node[anchor=west] at (-0.06,0.602){ J:};   
    \node[anchor=west] at (-0.06,0.352){ K:};   
    \node[anchor=west] at (-0.06,0.102){ L:};
    \end{scope}
\end{tikzpicture}
      \caption{\label{fig:tubeswapB}Here the time direction is to the right, and we depict 4 stages of the isotopy. To obtain I from Figure~\ref{fig:tubeswapA}.H we perform a small isotopy of the surface. In J we depict a disc intersecting the surface on the boundary of the disc. To obtain K from J we perform a destabilisation by removing a neighbourhood of the boundary of the disc and adding two parallel copies of the disc. To obtain L from K we perform an isotopy (pushing some of the middle band into the future) to obtain the associated tubed surface $\widetilde{S'}$.}
\end{figure}
%\begin{figure}[!ht]
%   \centering 
%    \begin{tikzpicture}
%\node[anchor=south west,inner sep=0] (image) at (0,0) {\includegraphics[width=\textwidth]{finaltubeswap5.pdf}};
%  \begin{scope}[x={(image.south east)},y={(image.north west)}]
%    \node at (0.015,0.66) { $t$};
%    \end{scope}
%\end{tikzpicture}
%      \caption{\label{fig:tubeswap5}We perform a small isotopy (left) so we may see a disc intersecting the surface on its boundary (right).}
%\end{figure}
%\begin{figure}[!ht]
%   \centering 
%   \begin{tikzpicture}
%\node[anchor=south west,inner sep=0] (image) at (0,0) {\includegraphics[width=\textwidth]{finaltubeswap6.pdf}};
%  \begin{scope}[x={(image.south east)},y={(image.north west)}]
%    \node at (0.015,0.674) { $t$};
%    \end{scope}
%\end{tikzpicture}
%      \caption{\label{fig:tubeswap6}We perform a destabilisation by removing a neighbourhood of the boundary of the disc. We then perform an isotopy (pushing some of the middle band into the future) to obtain the associated tubed surface $\widetilde{S'}$}
%\end{figure}

\begin{figure}[p]
   \centering
   \begin{tikzpicture}
\node[anchor=south west,inner sep=0] (image) at (0,0) {\includegraphics[width=0.92\textwidth]{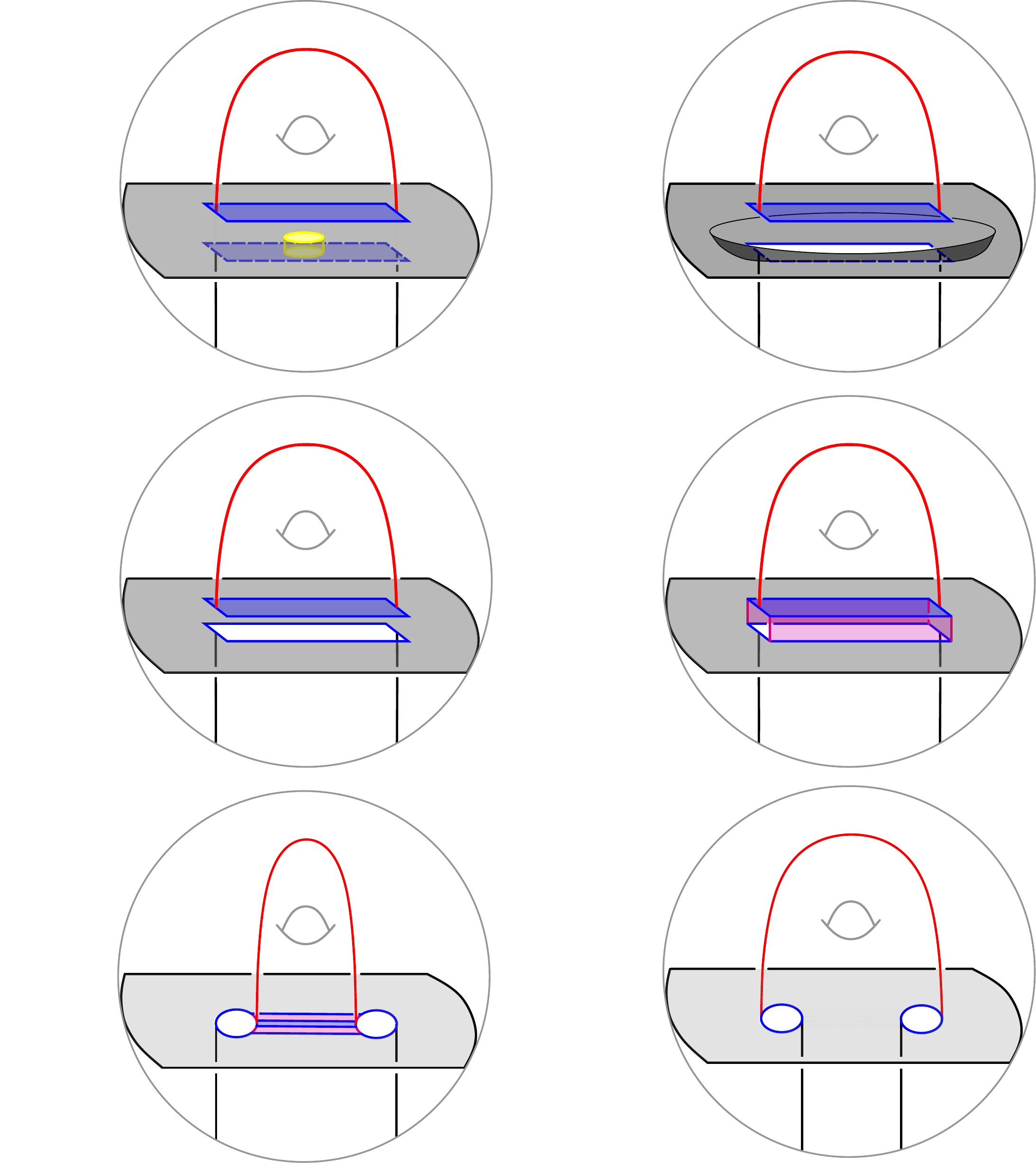}};
  \begin{scope}[x={(image.south east)},y={(image.north west)}]
%  \draw[step=0.1,gray,very thin] (0,0) grid (1,1)
    \node[anchor=west] at (0.00,0.84) {\Large $\hphantom{\sim}\;\;$ \normalsize A: };
    \node[anchor=west] at (0.00,0.5) {\Large $\sim\;\;$ \normalsize C:};
    \node[anchor=west] at (0.00,0.16) {\Large $\sim\;\;$ \normalsize E:};
    \node[anchor=west] at (0.522,0.16) {\Large $\sim\;\;$ \normalsize F:};
    \node[anchor=west] at (0.522,0.5) {\Large $\sim\;\;$ \normalsize D:};
    \node[anchor=west] at (0.522,0.84) {\Large $\sim\;\;$ \normalsize B:};
    \end{scope}
\end{tikzpicture}
\caption{\label{fig:alternativeorientationswap}A stabilisation and isotopy respecting the alternative possible orientations. We only depict the $t=0$ slice, the past and future frames are as in the previous case, until the final picture where the past and future pictures are swapped. In A we see the associated surface, to which we already performed a stabilisation respecting the alternative orientations. To obtain B from A we perform an isotopy `inflating' the tube from the stabilisation. To obtain B from C we flatten the surface (which deforms the past and future images slightly). To obtain D from C we push the sides of the tube into the present. To obtain E from D we perform an isotopy to flatten the surface. To obtain F from E we perform an isotopy. F is identical to Figure~\ref{fig:tubeswapA}.F, though we note the past and future images are swapped. From this point we proceed as in the previous case, with the past and future images swapped. }      
\end{figure}

The associated tubed surface $\widetilde{S}$ in this parametrisation is depicted in Figure~\ref{fig:tubeswap1}{.A}. We perform a stabilisation between the top of the tube, and the part of surface $f(\nu ( \gamma_i))$  to obtain Figure~\ref{fig:tubeswap1}{.B} (provided the stabilisation pictured respects orientations; we deal below with the case in which it does not). 

To obtain Figure~\ref{fig:tubeswap2}.C from Figure~\ref{fig:tubeswap1}.B we perform an isotopy supported in the $t=0$ frame by `inflating' the tube from the stabilisation. To obtain Figure~\ref{fig:tubeswap2}.D we flatten the resulting bulge, which slightly deforms the surface in other time frames. 

Next, we push the sides of the tube into the $t=0$ frame to obtain Figure~\ref{fig:tubeswapA}.E. Then we perform an isotopy, flattening the resulting dip to obtain Figure~\ref{fig:tubeswapA}.F. To obtain Figure~\ref{fig:tubeswapA}.G from Figure~\ref{fig:tubeswapA}.F, we thicken the red arc $f(\beta)$ by pushing some of the surface from the future into the present. We then perform an isotopy to create a tube with a band removed, depicted in Figure~\ref{fig:tubeswapA}.H. 

To obtain Figure~\ref{fig:tubeswapB}.I from Figure~\ref{fig:tubeswapA}.H  we perform an isotopy. In the same picture, we depict an embedded disc intersecting the surface on its boundary in Figure~\ref{fig:tubeswapB}.J. We then remove an annulus given by a neighbourhood of the boundary of this disc, and replace it with two parallel copies of the disc in the~$t=0$ frame, thus performing a destabilisation, to obtain Figure~\ref{fig:tubeswapB}.K. We then push part of the band containing the arc $f(\beta)$ into the future and perform a small further isotopy to obtain Figure~\ref{fig:tubeswapB}.L. Note that Figure~\ref{fig:tubeswapB}.L is precisely $\widetilde{S'}$.

In the case that the above stabilisation was not compatible with orientations, we instead perform a stabilisation on the underside of the surface to obtain Figure~\ref{fig:alternativeorientationswap}.A. We then perform the sequence of isotopies in Figure~\ref{fig:alternativeorientationswap}, which are analogous to those in the previous case. Note that in Figure~\ref{fig:alternativeorientationswap} only the $t=0$ slice is depicted, the past and future pictures are as in the previous case, until the final picture where the past and future images are swapped (i.e.\ the past images become the future images and \emph{vice versa}). Figure~\ref{fig:alternativeorientationswap}.F is identical to Figure~\ref{fig:tubeswapA}.F, except that the past and future images are swapped. We now proceed as above (with the past and future images swapped). Again we obtain $\widetilde{S'}$, completing the proof of the tube swap lemma. \end{proof}
\section{The Tube Move Lemma}\label{sec:tubemove}
We prove that if we take a single arc $\gamma_i$ in our surface tubing diagram, remove it and replace it with another arc, as in Figure~\ref{fig:tubemovediag0}, there is a single stabilisation and single destabilisation taking one associated surface to the other.
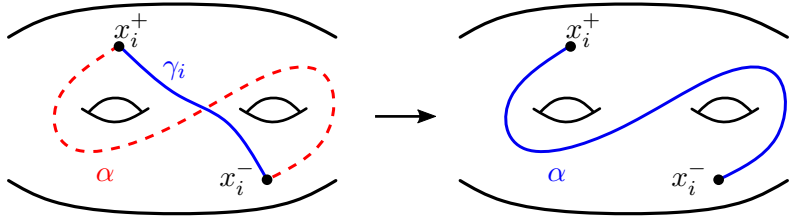
\begin{figure}[h!]
   \centering
   \begin{tikzpicture}
\node[anchor=south west,inner sep=0] (image) at (0,0) {\includegraphics[width=0.7\textwidth]{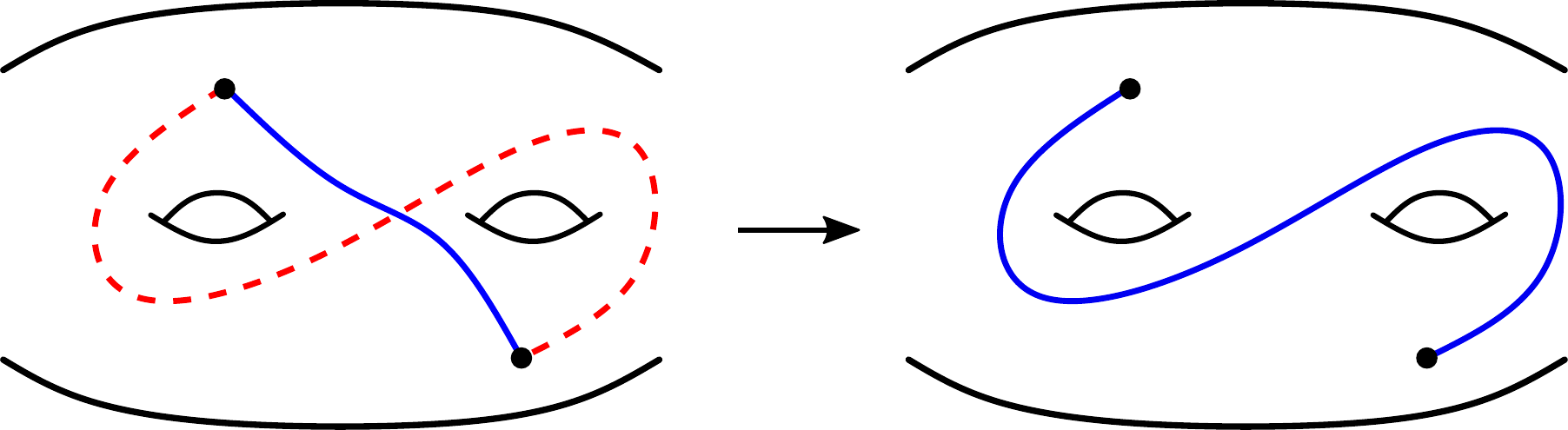}};
  \begin{scope}[x={(image.south east)},y={(image.north west)}]
%  \draw[step=0.1,gray,very thin] (0,0) grid (1,1);
    \node[text=blue] at (0.69+0.577-0.564,0.19) {$\alpha$};
  \node[text=red] at (0.69-0.564,0.19) {$\alpha$};
  \node[text=blue] at (0.215,0.68) {$\gamma_i$};
  \node at (0.16,0.87) {$x_i^+$};
  \node at (0.295,0.178) {$x_i^-$};

  \node at (0.16+0.577,0.87) {$x_i^+$};
  \node at (0.295+0.577,0.178) {$x_i^-$};

    \end{scope}
\end{tikzpicture}
      \caption{\label{fig:tubemovediag0}A diagram for a tube move. We replace $\gamma_i$ with $\alpha$, which may intersect $\gamma_i$, but not other arcs or points.}
\end{figure}
\begin{lemma}[Tube Move Lemma]\label{tubemove} Given a surface tubing diagram $\mathcal{S}$, let $\alpha$ be an arc in $S$ from $x_i^+$ to $x_i^-$ which is disjoint from the curves $\{\gamma_k\}_{k\neq i}$ $($note that $\alpha$ may intersect~$\gamma_i)$, and is also disjoint from all marked points on the surface other than $x_i^+$ and $x_i^-$. Form $\mathcal{S}'$ by removing the arc~$\gamma_i$ and replacing it with $\alpha $.

Then the associated tubed surface $\widetilde{S}$ can be transformed into $\widetilde{S'}$ by performing a single stabilisation, followed by a destabilisation $($and ambient isotopy$)$.
\end{lemma}

\begin{proof}
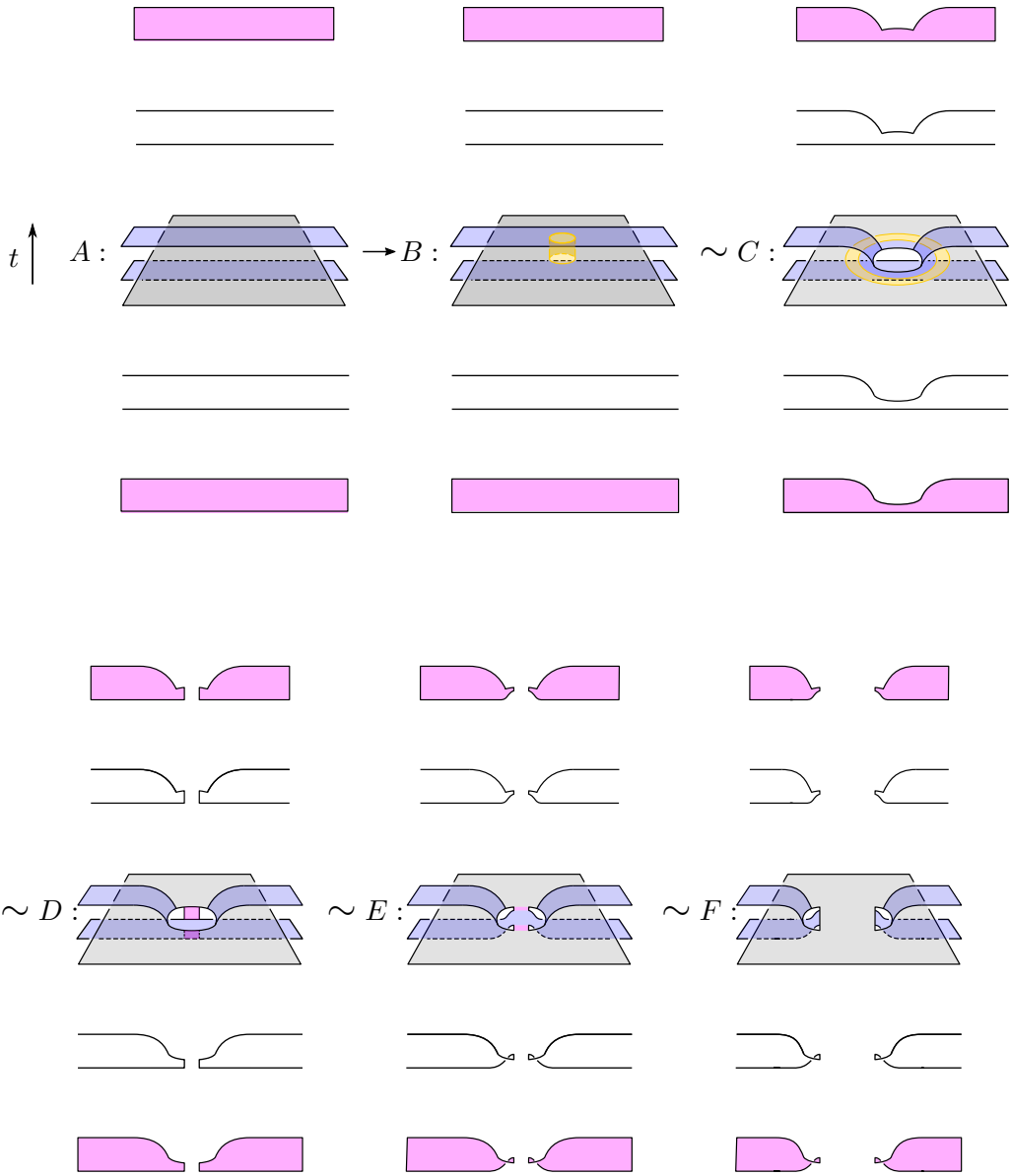
\begin{figure}[p]
   \centering
\begin{tikzpicture}
\node[anchor=south west,inner sep=0] (image) at (0,0) {\includegraphics[width=0.92\textwidth]{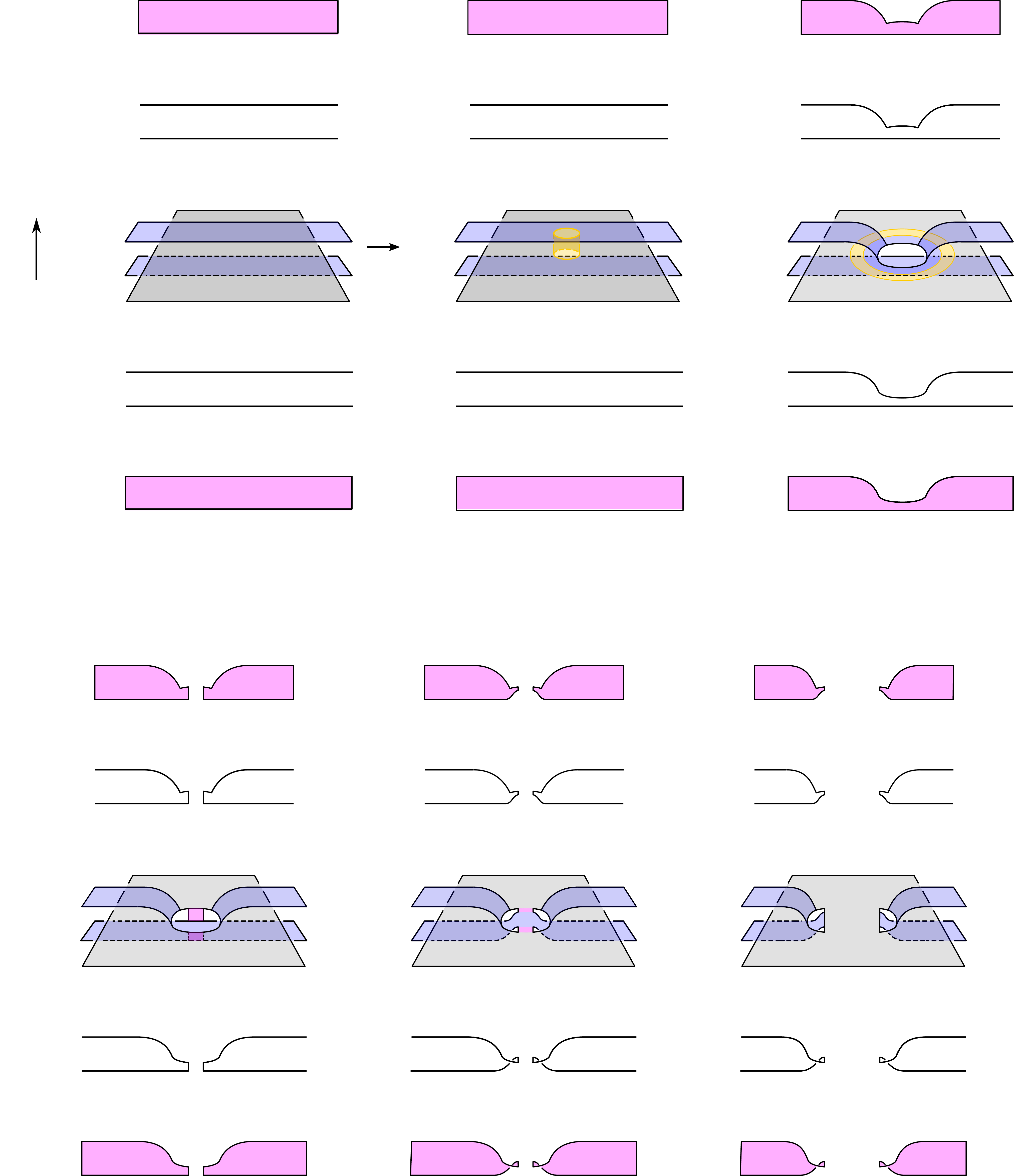}};
  \begin{scope}[x={(image.south east)},y={(image.north west)}]
%  \draw[step=0.1,gray,very thin] (0,0) grid (1,1);
  \node at (0.091,0.79) {$A:$};
    \node at (0.418,0.79) {$B:$};
    \node at (0.732,0.79) {\Large $\sim\:$\normalsize $C:$};
    \node[anchor=west] at (0,0.225) {\hspace{-0.07cm}\Large $\sim\:$\normalsize $D:$};
        \node[anchor=west] at (0.317,0.225) {\Large $\sim\:$\normalsize $E:$};
         \node[anchor=west] at (0.647,0.225) {\Large $\sim\:$\normalsize $F:$};
    \node at (0.018,0.784) { $t$};
    \end{scope}
\end{tikzpicture}
      \caption{\label{fig:tubemove}A stabilisation and sequence of isotopies allowing us to cut open tubes in order to move their ends about freely. In $A$ we depict a neighbourhood $\nu(p)\subset X$ for some~$p\in\Gamma_i$. Note that $\nu(p)\cong B^4$. The time direction is depicted upwards and in each time frame, we see a copy of $B^3$. To obtain B from A we perform a stabilisation. To obtain C from B we perform an isotopy in the $t=0$ frame, which also slightly deforms the surface in other time frames. To obtain D from C we push part of the sides of the tube into the $t=0$ frame. To obtain E from D we perform an isotopy flattening the surface. Finally, in F we see two disjoint tubes which join the surface at their ends as pictured. These ends may now be pushed about the surface. When we later rejoin the tubes we read the pictures in reverse order.}
\end{figure}

We direct the reader to Figure~\ref{fig:tubemove}. In Figure~\ref{fig:tubemove}.A, we see the associated tubed surface $\widetilde{S}$ in a neighbourhood of a point $p\in \Gamma_i$ which intersects the linking annulus of~$\Gamma_i$ on a smaller annulus. First, we perform a stabilisation between the tube and the surface to obtain Figure~\ref{fig:tubemove}.B (note that if this is not compatible with orientations, we instead perform the stabilisation on the underside and proceed in the same way, taking the mirror image of each figure). We then perform an isotopy to obtain Figure~\ref{fig:tubemove}.C, then push part of the sides of the tube into the present to obtain Figure~\ref{fig:tubemove}.D. We then perform a further isotopy to obtain Figure~\ref{fig:tubemove}.E. Here we see two disconnected tubes, whose ends join the surface as depicted in Figure~\ref{fig:tubemove}.F.

We can now move these ends around freely on the surface, provided during the isotopy they are disjoint from the other tubes and double points. We depict how we drag the ends diagrammatically in Figure~\ref{fig:tubemovediag}. First, we retract the two tubes along $\Gamma_i$ dragging the ends towards $z_i^+$ and $z_i^-$. We then perform an isotopy that pushes these ends along~$\alpha$ so that the two tubes now run along $\alpha$ and the ends of the tubes lie in a neighbourhood of some point $q\in f(\alpha)$. In this neighbourhood, we see the final picture of Figure~\ref{fig:tubemove}. We then rejoin the tubes, by performing an isotopy and destabilisation which can be seen by reading the pictures in Figure~\ref{fig:tubemove} in reverse order. After rejoining the tubes they form the linking annulus of $\alpha$. The resulting surface is $\widetilde{S'}$ as required. This completes the proof of the tube move lemma.
\end{proof}
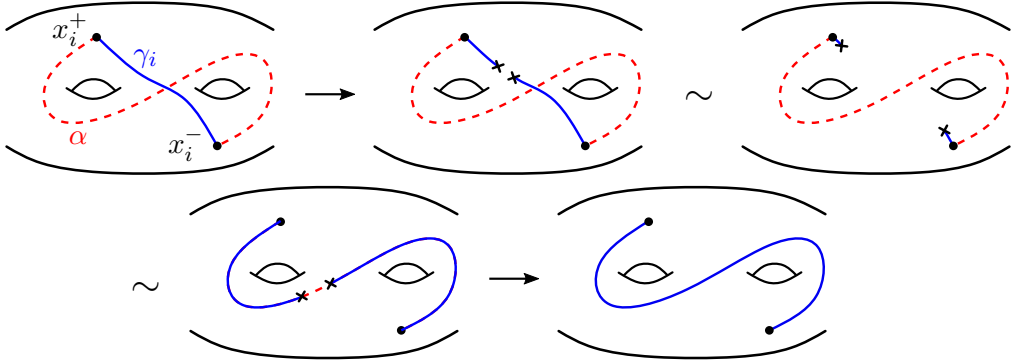
\begin{figure}[h!]
   \centering
   \begin{tikzpicture}
\node[anchor=south west,inner sep=0] (image) at (0,0) {\includegraphics[width=0.9\textwidth]{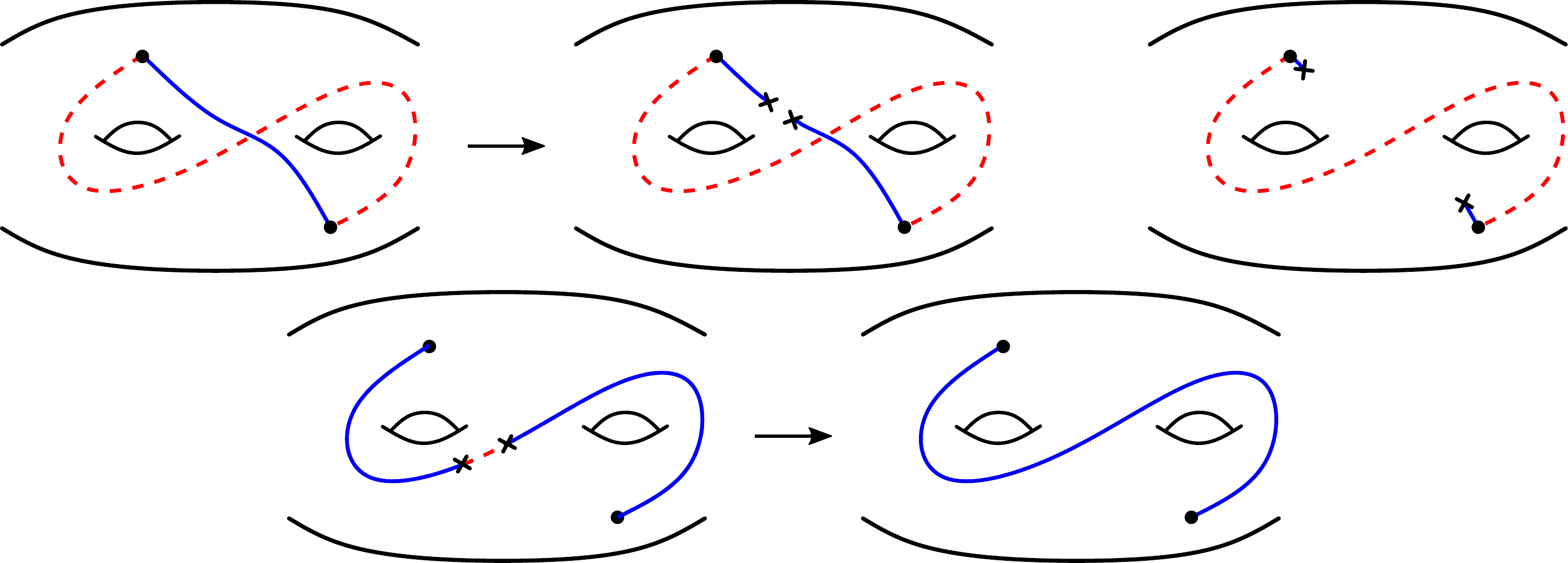}};
  \begin{scope}[x={(image.south east)},y={(image.north west)}]
%  \draw[step=0.1,gray,very thin] (0,0) grid (1,1);
    \node at (0.69,0.73) {\Large $\sim$};
    \node at (0.14,0.197) {\Large $\sim$};
    \node at (0.065,0.92) {$x_i^+$};
    \node at (0.18,0.6) {$x_i^-$};
    \node[text=red] at (0.073,0.62) {$\alpha$};
    \node[text=blue] at (0.14,0.84) {$\gamma_i$};
    \end{scope}
\end{tikzpicture}
      \caption{\label{fig:tubemovediag}A schematic of the proof of the tube move lemma. The second image depicts the cut-open tube. The crosses depict where the ends of the cut-open tube join the surface. We then show how to move these ends in the next two equivalences. To obtain to the final picture we rejoin the cut open ends.}
\end{figure}
\begin{remark}
Note that we cannot prove Lemma~\ref{tubemove} by using Lemma~\ref{tubeswap} twice, due to the condition in Lemma~\ref{tubeswap} that the new arc is disjoint from the old one. To do so we would need to find an intermediate arc $\beta$ from $y_i^+$ to $y_i^-$ disjoint from both $\gamma_i$ and $\alpha$, but if $\gamma_i$ and $\alpha$ intersect, then $S\setminus \{\gamma_i,\alpha\}$ may be disconnected so this may not be possible.
\end{remark}

\section{Proof of Theorem~\ref{ineq}}\label{sec:mainproof}
With these tools in place, we proceed with the proof of Theorem~\ref{ineq}, namely that $d_{\st}(\Sigma,\Sigma')\leq d_{\sing}(\Sigma,\Sigma')+1$. We do so by shadowing a regular homotopy by a sequence of embedded surfaces differing by stabilisation, destabilisation, and ambient isotopy.
\begin{proof}[Proof of Theorem~\ref{ineq}]
Recall that $\Sigma$ and $\Sigma'$ are immersed surfaces in $X$ of genus $g$. In the case the surfaces are not regularly homotopic $d_{\sing}(\Sigma,\Sigma')=\infty$ and we are done. Hence assume that $\Sigma$ and $\Sigma'$ are regularly homotopic and suppose $d_{\sing}(\Sigma,\Sigma')=n$. Given distance minimising regular homotopy from~$\Sigma$ to $\Sigma'$, let $P_1,\ldots P_k$ be the sequence of immersed surfaces describing this homotopy, each differing from the previous either a finger move, a Whitney move, or an ambient isotopy.

We shall describe a sequence of surface tubing diagrams with associated tubed surfaces that differ by stabilisations, destabilisations, and ambient isotopy, such that the genus of any intermediate surface never exceeds $g+n+1$.

Fix the abstract surface $S$, and immersions $f_i\colon  S\rightarrow X$ with image $P_i$ for each $i$. The immersion $f_1\colon S\rightarrow X$ gives a surface diagram $\mathcal{S}_1$ (the empty diagram in $S$) with associated tubed surface $\widetilde{S}_1=S_1^{\im}=P_1$.

We now suppose for induction that we have a surface tubing diagram $\mathcal{S}_i$ with immersion data $f_i\colon S\rightarrow X$, and associated tubed surface $\widetilde{S}_i$. We will construct a surface tubing diagram $\mathcal{S}_{i+1}$ with immersion data $f_{i+1}\colon S\rightarrow X$, such that $\widetilde{S}_{i+1}$ differs from $\widetilde{S}_i$ by a series of stabilisations, destabilisations, and ambient isotopy, such that the genus of any intermediate surface does not exceed $g+n+1$. There are three cases; $P_{i+1}$ is obtained from $P_i$ by ambient isotopy, a finger move, or a Whitney move.

\textit{Ambient Isotopy:} If $P_{i+1}$ just differs from~$P_i$ by ambient isotopy, by Remark~\ref{isoremark}(1) we have a new surface tubing diagram $\mathcal{S}_{i+1}$ with immersion data $f_{i+1}\colon S\rightarrow X$. Furthermore, $\widetilde{S_i}$ is ambiently isotopic to $\widetilde{S}_{i+1}$.

\textit{Finger Move:} If $P_{i+1}$ differs from $P_i$ by a finger move, let $\alpha$ and $\beta$ be the Whitney arcs of the Whitney disc that undoes the finger move. Note these arcs are in $S$, the abstract surface. Before we perform the finger move, we perform an isotopy of the arcs, to push any arcs off $\alpha$ and $\beta$ to form $\mathcal{S}_i'=(f_i,\{\gamma_j'\})$. We then add the new double points and $\beta$ to the diagram to obtain diagram $\mathcal{S}_{i+1}=(f_{i+1},\{\gamma_j'\}\cup\{\beta\})$ which we note uses the new map $f_{i+1}$, which differs from $f_i$ by the finger move; see Figure~\ref{fig:fingermovediag}. The associated tubed surfaces $\widetilde{S_i}$ and $\widetilde{S'_i}$ are ambiently isotopic by Remark~\ref{isoremark}(2), and $\widetilde{S'_i}$ and $\widetilde{S}_{i+1}$ differ by isotopy and a stabilisation, as in Figure~\ref{fig:fingermoveiso}.

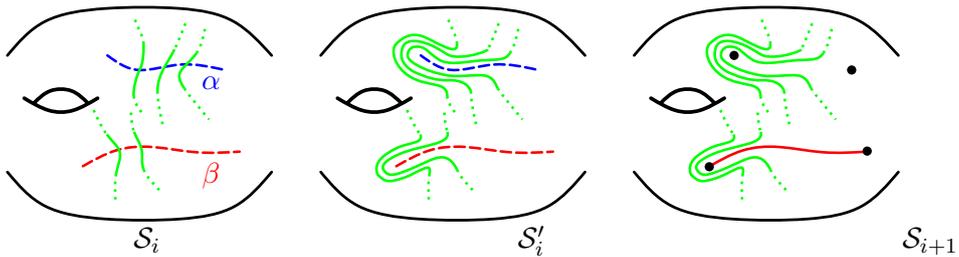
\begin{figure}[!h]
   \centering 
\begin{tikzpicture}
\node[anchor=south west,inner sep=0] (image) at (0,0) {\includegraphics[width=0.8\textwidth]{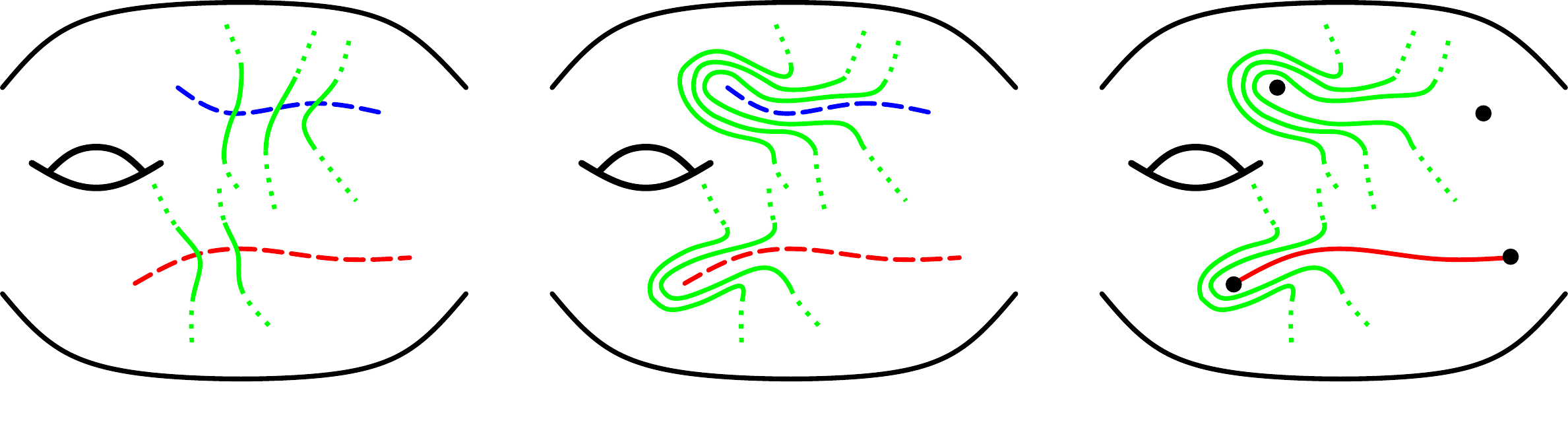}};
  \begin{scope}[x={(image.south east)},y={(image.north west)}]
  \node[text=red] at (0.23,0.32) {$\beta$};
  \node[text=blue] at (0.23,0.69) {$\alpha$};
    \node[anchor=south west] at (0.132,-0.04) {$\mathcal{S}_i\hspace{4.7cm}\mathcal{S}_i'\hspace{4.7cm}\mathcal{S}_{i+1}$};
    \end{scope}
\end{tikzpicture}
      \caption{\label{fig:fingermovediag}The sequence of surface tubing diagrams corresponding to performing a finger move.}
\end{figure}
\begin{figure}[!ht]
   \centering 
   \begin{tikzpicture}
\node[anchor=south west,inner sep=0] (image) at (0,0) {\includegraphics[width=0.97\textwidth]{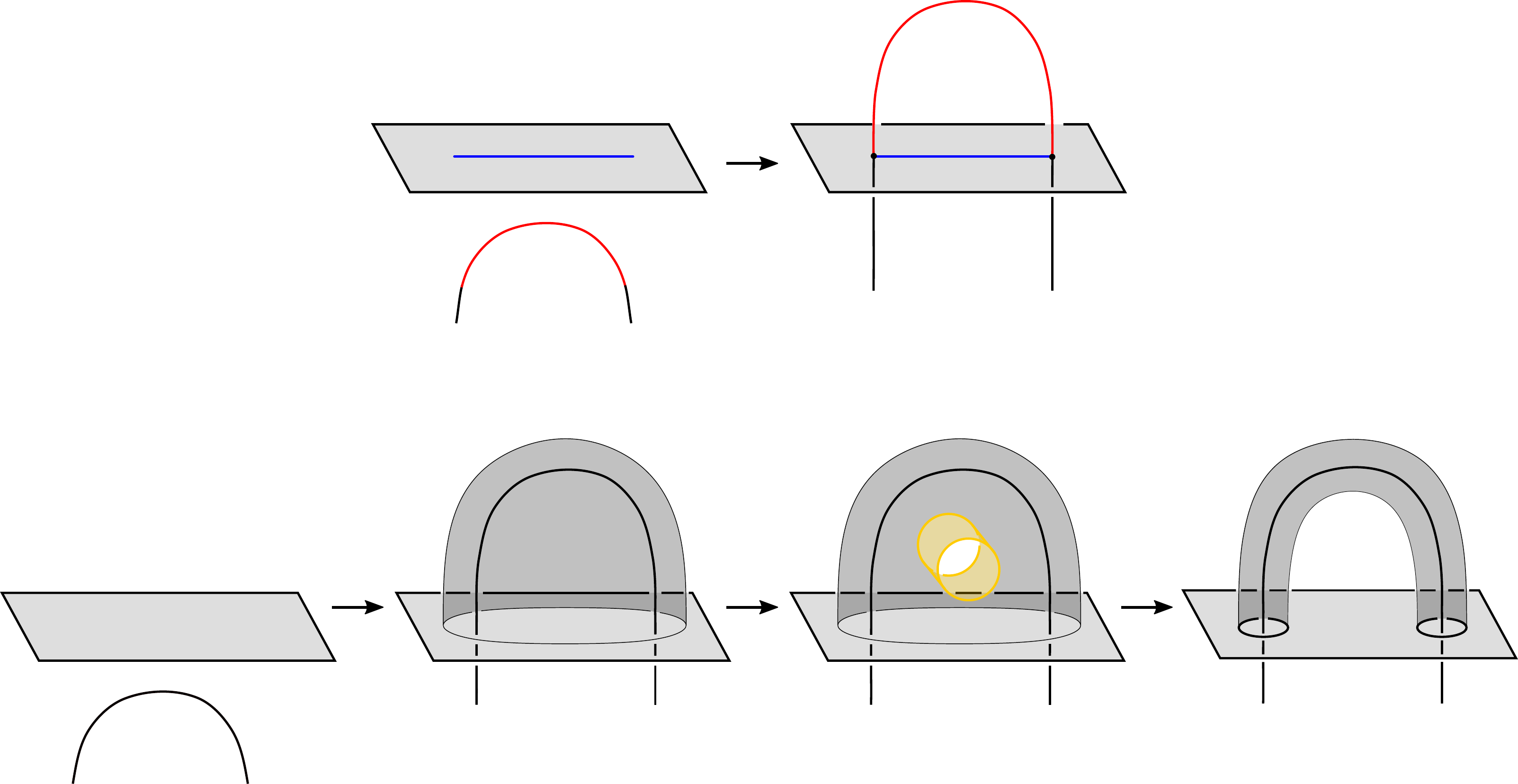}};
  \begin{scope}[x={(image.south east)},y={(image.north west)}]
  %  \draw[step=0.1,gray,very thin] (0,0) grid (1,1);
    \node[text=red] at (0.7,0.96) {$\Gamma_i$};
    \end{scope}
\end{tikzpicture}
      \caption{\label{fig:fingermoveiso}Above we see the immersed surface before and after a finger move. Below we perform a sequence of isotopies and stabilisations taking $\widetilde{S}_i'$ to $\widetilde{S}_{i+1}$. First, we isotope one sheet to create a Whitney bubble and push the other sheet into this bubble. We then stabilise the bubble to create the linking annulus of $\Gamma_i$.}
\end{figure}
\textit{Whitney Move:} The Whitney moves present the main difficulty. If $P_{i+1}$ differs from~$P_i$ by a Whitney move, there are several cases to consider.

The endpoints of the Whitney arcs form the set $\{x_i^+,y_i^+,x_j^-,y_j^-\}$ for some~$i$ and~$j$. In the case that $i=j$ then either~$\partial \alpha =\{x_i^+,x_i^-\}$ and~$\partial \beta =\{y_i^+,y_i^-\}$, or~$\partial \alpha =\{x_i^+,y_i^-\}$ and~$\partial \beta =\{y_i^+,x_i^-\}$ (up to relabeling of $\alpha$ and $\beta$). We call the former situation Case 1 and the latter situation Case 2. Case~2 is the crossed Whitney disc .

Otherwise, if $i\neq j$ then either $\partial \alpha =\{x_i^+,x_j^-\}$ and $\partial \beta =\{y_i^+,y_j^-\}$, or~$\partial \alpha =\{x_i^+,y_j^-\}$ and~$\partial \beta =\{y_j^+,x_i^-\}$ (again up to relabeling of $\alpha$ and $\beta$). We call the former situation Case 3 and the latter Case 4.

Note that $\alpha$ and $\beta$ may intersect other arcs, including $\gamma_i$ and $\gamma_j$. In all cases, we perform a small isotopy of all the arcs so that they intersect $\alpha$ and $\beta$ transversely.

\textbf{Case 1:} \textit{The Whitney arcs are $\alpha$ between $x_i^+$ and $x_i^-$, and $\beta$ between $y_i^+$ and $y_i^-$ for some $i$.} We wish to run in reverse what we did in the case of a finger move, however, we need $\gamma_i$ to run over either $\alpha$ or $\beta$, and we also need $\alpha$ and $\beta$ to be disjoint from other arcs, so that the surface tubing diagram looks like the diagram at the end of a finger move, as in Figure~\ref{fig:fingermovediag}. 

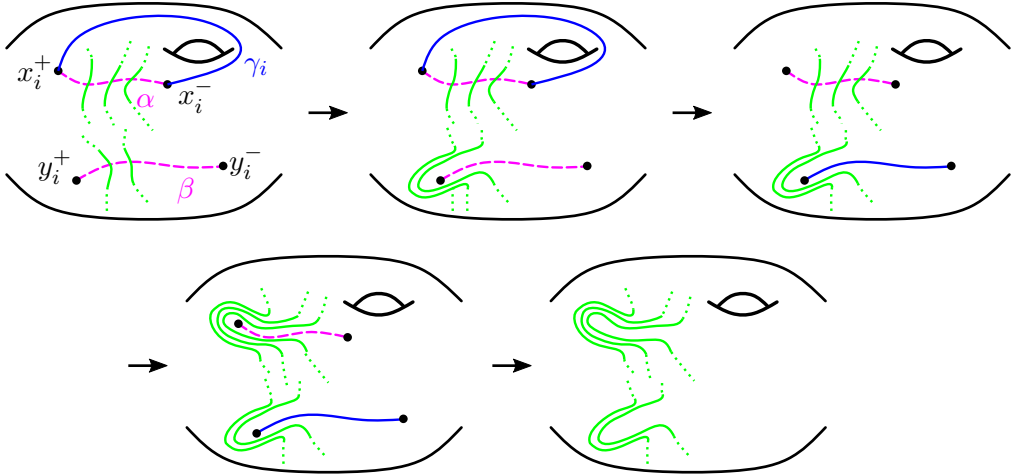
\begin{figure}[!h]
	\centering 
	\begin{tikzpicture}
	\node[anchor=south west,inner sep=0] (image) at (0,0) {\includegraphics[width=0.9\textwidth]{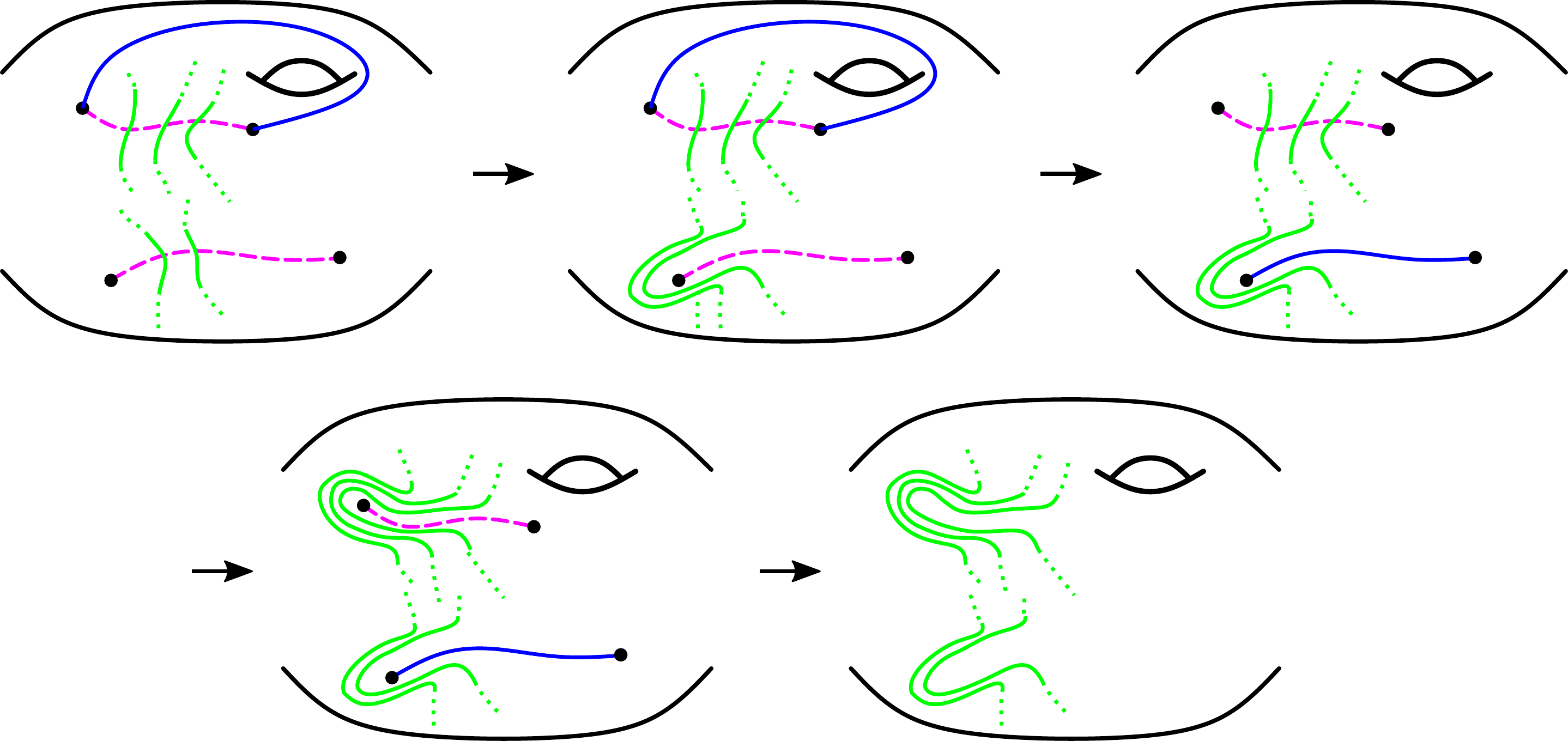}};
	\begin{scope}[x={(image.south east)},y={(image.north west)}]
	%\draw[step=0.1,gray,very thin] (0,0) grid (1,1);
	\node at (0.03,0.85) {$x_i^+$};
	\node at (0.19,0.8) {$x_i^-$};
	\node at (0.05,0.65) {$y_i^+$};
	\node at (0.24,0.66) {$y_i^-$};
	\node[text=mpurp] at (0.14,0.79) {$\alpha$};
	\node[text=mpurp] at (0.18,0.6) {$\beta$};
	\node[text=blue] at (0.25,0.86) {$\gamma_i$};
	\end{scope}
	\end{tikzpicture}
	\caption{\label{fig:case1diag}The sequence of tubing diagrams which correspond to performing a Case 1 Whitney move. First, we move all arcs off $\beta$, then swap $\gamma_i$ to $\beta$, then move all arcs off $\alpha$. The resulting diagram corresponds to that at the end of a finger move. We can now destabilise the associated surface to obtain the associated surface for the final diagram (which uses the new immersion data), by running Figure~\ref{fig:fingermoveiso} backwards.}
\end{figure}

We arrange this by first using the tube move lemma to move any arcs off $\beta$. To do so we remove intersection points with $\beta$ one by one. We consider the arc $\gamma_r$ which has the intersection with $\beta$ closest  to $y_i^+$ (in the sense of distance along $\beta$), at $p\in \beta$. We form the arc $\gamma_r'$, by removing an arc neighbourhood of $p$ from $\gamma_r$, and replacing it with an arc that runs along the boundary of a small neighbourhood of $\beta\subset S$; see Figure~\ref{fig:case1diag}. Provided the neighbourhood is taken to be sufficiently small, $\gamma_r'$ is disjoint from $\{\gamma_p\}_{p\neq r}$ and from marked points other than $x_r^\pm$, and so the corresponding associated surfaces differ by a stabilisation and destabilisation by the tube move lemma. We do this for every intersection point of arcs with $\beta$, until all arcs are disjoint from $\beta$. Note that we may move $\gamma_i$ during this process.

Next, we use the tube swap lemma to swap $\gamma_i$ to $\beta$. We then use the tube move lemma to move any arcs off $\alpha$ as we did for $\beta$; see the third to fourth picture of Figure~\ref{fig:case1diag}. Finally, we perform the destabilisation and isotopy coming from reading Figure~\ref{fig:fingermoveiso} in reverse. This has the effect of removing the points $x_i^\pm$ and $y_i^\pm$, and the arc $\gamma_i$ (which now runs along $\beta$) from the diagram.

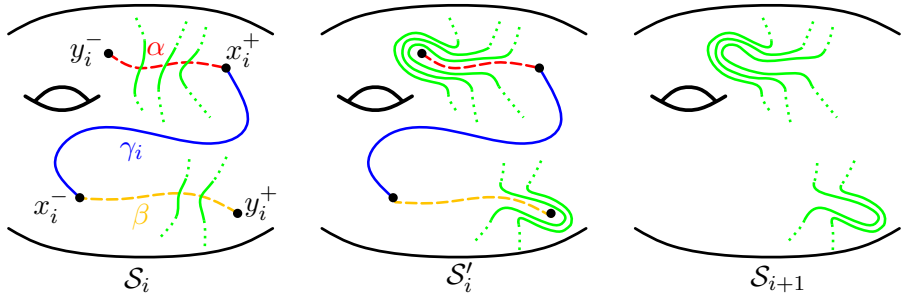
\begin{figure}[!h]
\centering
\begin{tikzpicture}
\node[anchor=south west,inner sep=0] (image) at (0,0) {\includegraphics[width=0.8\textwidth]{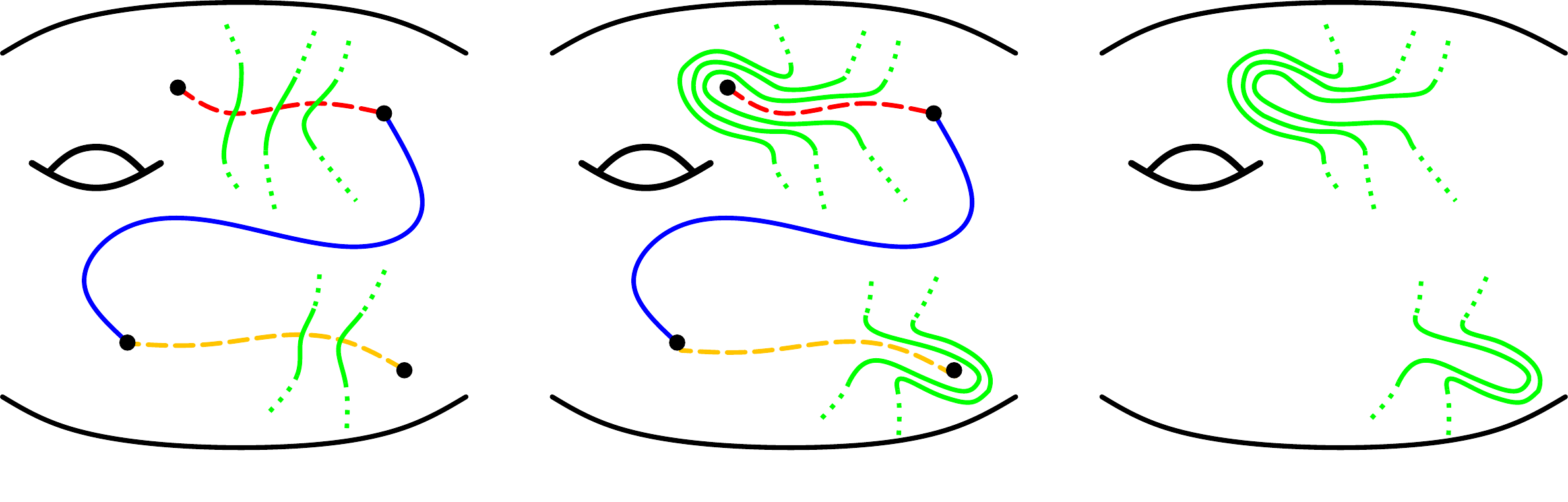}};
  \begin{scope}[x={(image.south east)},y={(image.north west)}]
  \node[text=myo] at (0.15,0.25) {$\beta$};
  \node[text=blue] at (0.14,0.48) {$\gamma_i$};
  \node[text=red] at (0.167,0.84) {$\alpha$};
  \node at (0.09,0.84) {$y_i^-$};
  \node at (0.265,0.835) {$x_i^+$};
  \node at (0.285,0.3) {$y_i^+$};
  \node at (0.05,0.289) {$x_i^-$};
  \node[anchor=south west] at (0.12,-0.05) {$\mathcal{S}_i$}; 
  \node[anchor=south west] at (0.48,-0.05) {$\mathcal{S}_i'$};
  \node[anchor=south west] at (1-0.18,-0.05) {$\mathcal{S}_{i+1}$};    \end{scope}
\end{tikzpicture}
\caption{\label{fig:case2diag}The sequence of diagrams corresponding to a Case 2 Whitney move. First, we remove all arcs running over $\beta$ and $\alpha$ using the tube move lemma, then claim we may remove the resulting tube and intersection points using a stabilisation and two destabilisations.}
\end{figure}

\textbf{Case 2:} \textit{The Whitney arcs are $\alpha$ between $x_i^+$ and $y_i^-$, and $\beta$ between $y_i^+$ and $x_i^-$ for some $i$.} This is a crossed Whitney disc. As in Case 1, we first remove any arcs running over $\alpha$ and $\beta$ using the tube move lemma, which may move $\gamma_i$; see Figure~\ref{fig:case2diag}. Let $\mathcal{S}_i'$ be the resulting surface diagram, and $\mathcal{S}_{i+1}$ be the diagram with the same points and arcs but with the points $x_i^\pm$ $y_i^\pm$, the arc $\gamma_i$ deleted, and immersion data $f_{i+1}$; See Figure~\ref{fig:case2diag}.

\begin{figure}[p]
   \centering 
   \begin{tikzpicture}
\node[anchor=south west,inner sep=0] (image) at (0,0) {\includegraphics[width=\textwidth]{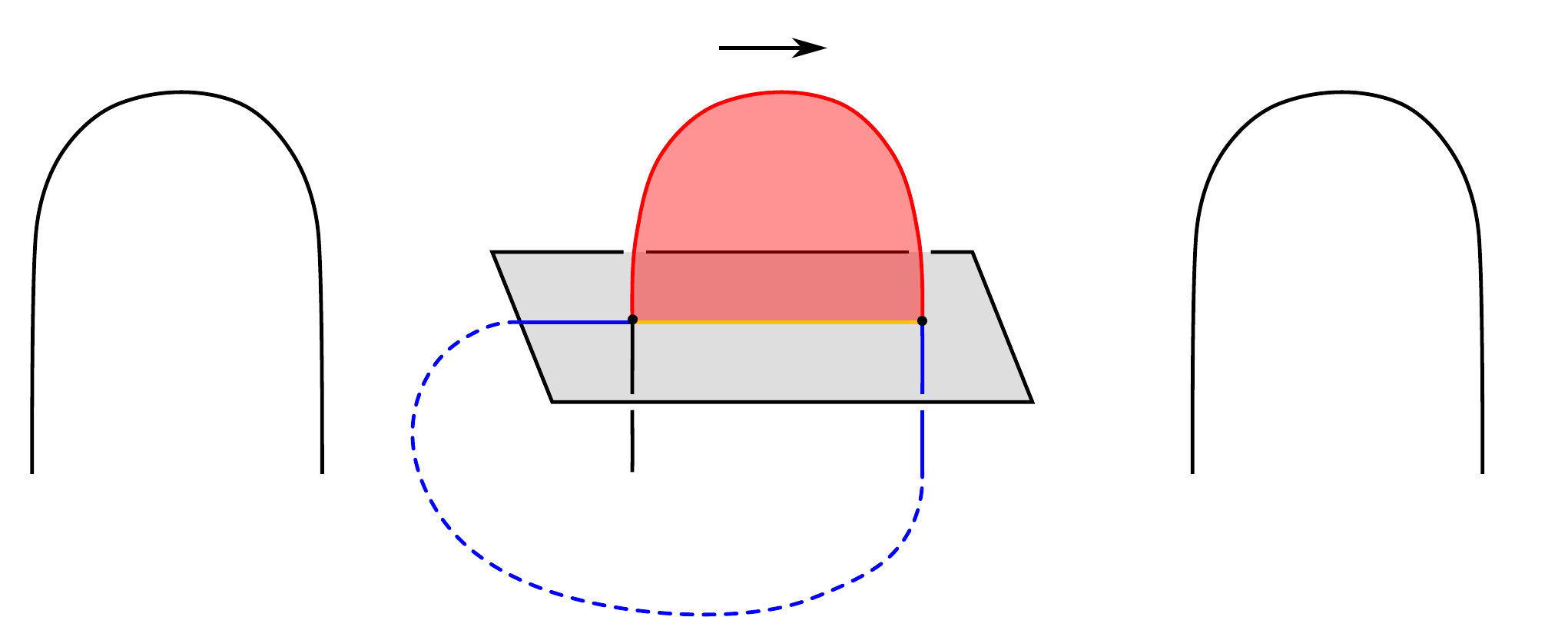}};
  \begin{scope}[x={(image.south east)},y={(image.north west)}]
    \node at (0.492,0.98) { $t$};
    \node[text=blue] at (0.34,0.132) { $\Gamma_i$};
    \end{scope}
\end{tikzpicture}
      \caption{\label{fig:crossedwhitney}The immersed surface in a neighbourhood of a crossed Whitney disc. The arc $\Gamma_i$ joins the two sheets, running over the surface outside of this 4-ball.\\}
   \centering 
\begin{tikzpicture}
\node[anchor=south west,inner sep=0] (image) at (0,0) {\includegraphics[width=\textwidth]{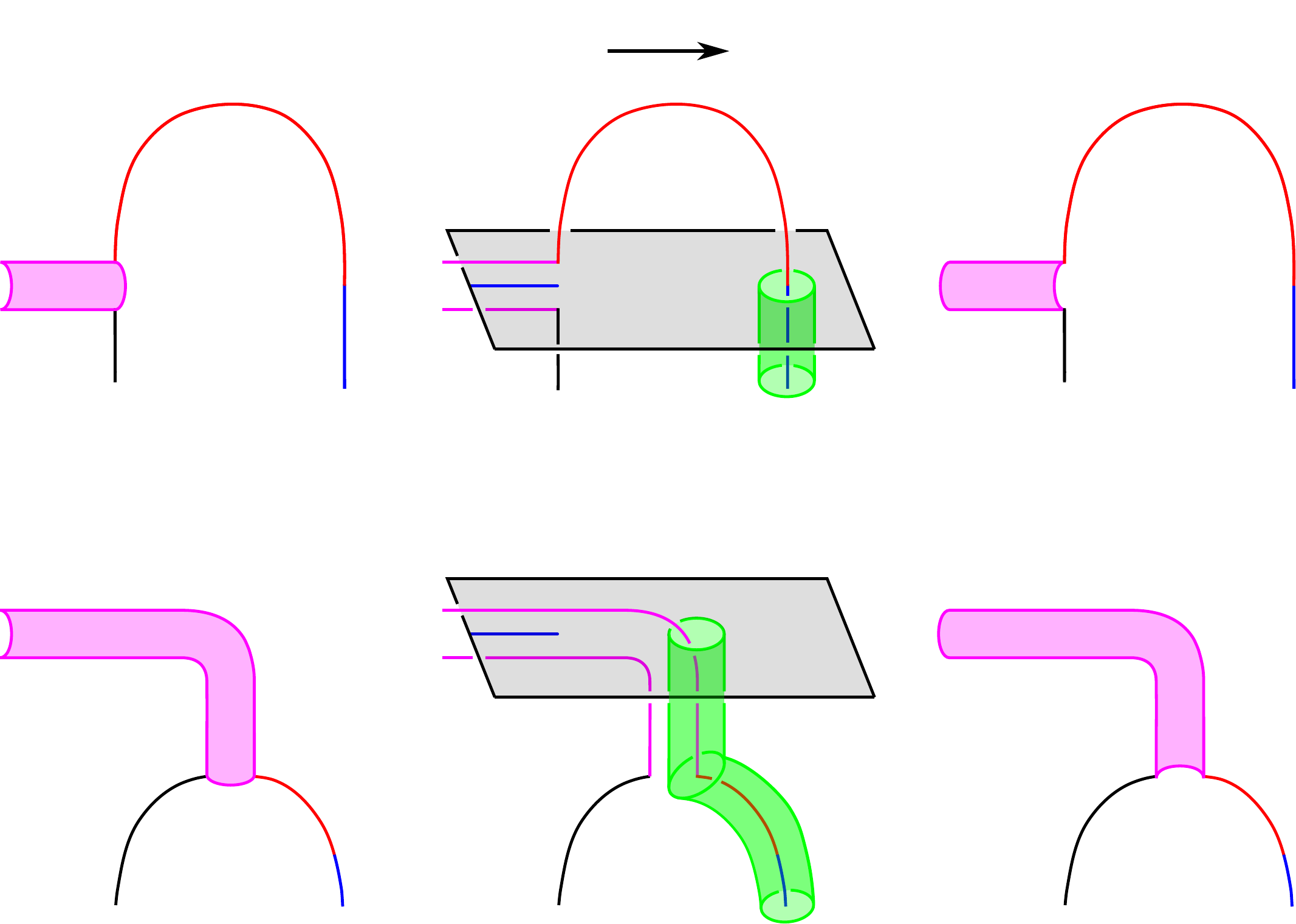}};
  \begin{scope}[x={(image.south east)},y={(image.north west)}]
  %\draw[step=0.1,gray,very thin] (0,0) grid (1,1);
    \node at (0.512,0.97) { $t$};
    \node[rotate=90] at (0.512,0.47) {\Large $\sim$};
    \node[anchor=west] at (0.56,0.195) {\Large $\leftarrow$ \normalsize Hopf Link$\times [0,1]$};
    \end{scope}
\end{tikzpicture}
\caption{\label{fig:creationofdoubletube}The associated tubed surface $\widetilde{S_i'}$. We pull the two sheets of the associated tubed surface apart, at the expense of creating a double tube. Note the pink and green tubes join up outside of this picture, and form the linking annulus for $\Gamma_i$. This is the same operation as that Gabai describes in \cite[Figure 5.9]{Gabai2017}.}
\end{figure}
\begin{figure}[p]
   \centering
\begin{tikzpicture}
\node[anchor=south west,inner sep=0] (image) at (0,0) {\includegraphics[width=0.85\textwidth]{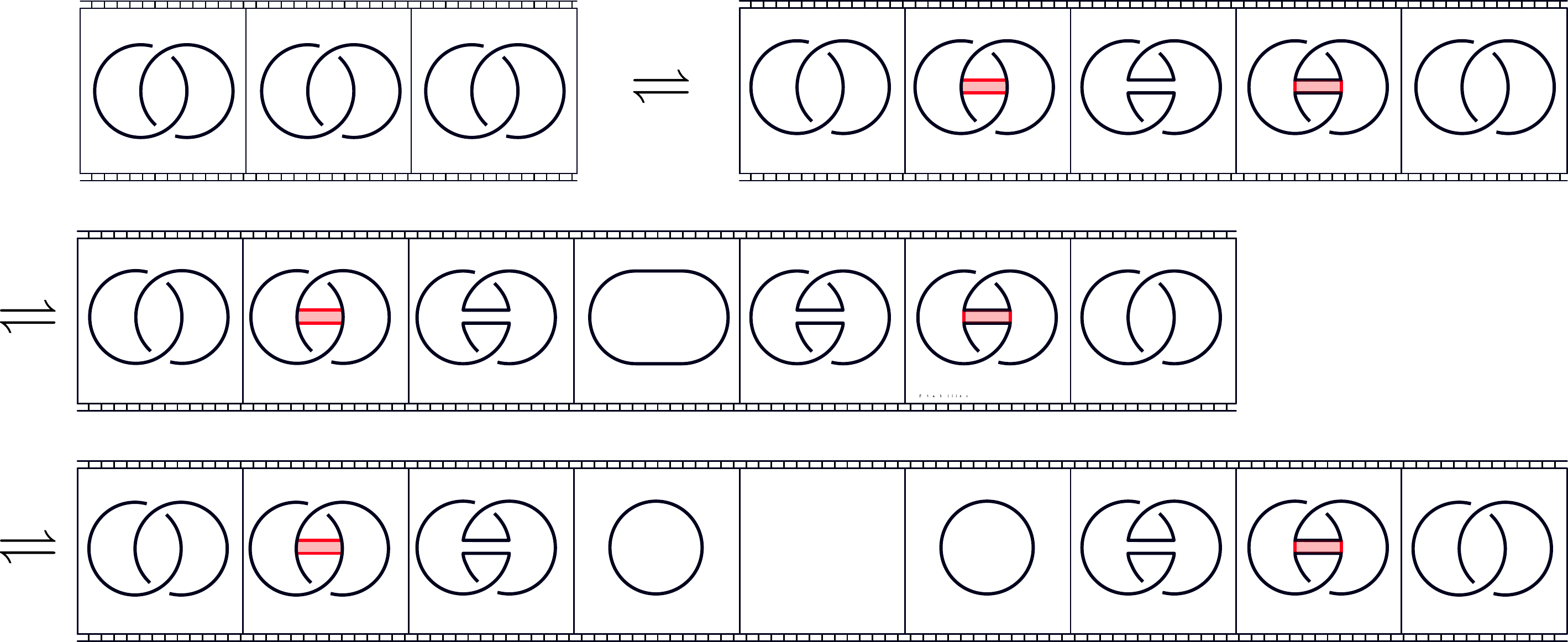}};
  \begin{scope}[x={(image.south east)},y={(image.north west)}]

    \end{scope}
\end{tikzpicture}   

\caption{\label{fig:cutthetube}The first image depicts part of the double tube. To obtain the second we perform a stabilisation. We then perform an isotopy to obtain the third, in which we see an unknot$\times [0,1]$ in the middle time frame. To obtain the fourth we perform a destabilisation along this unknot.\\}

   \centering 
   \begin{tikzpicture}
\node[anchor=south west,inner sep=0] (image) at (0,0) {\includegraphics[width=0.84\textwidth]{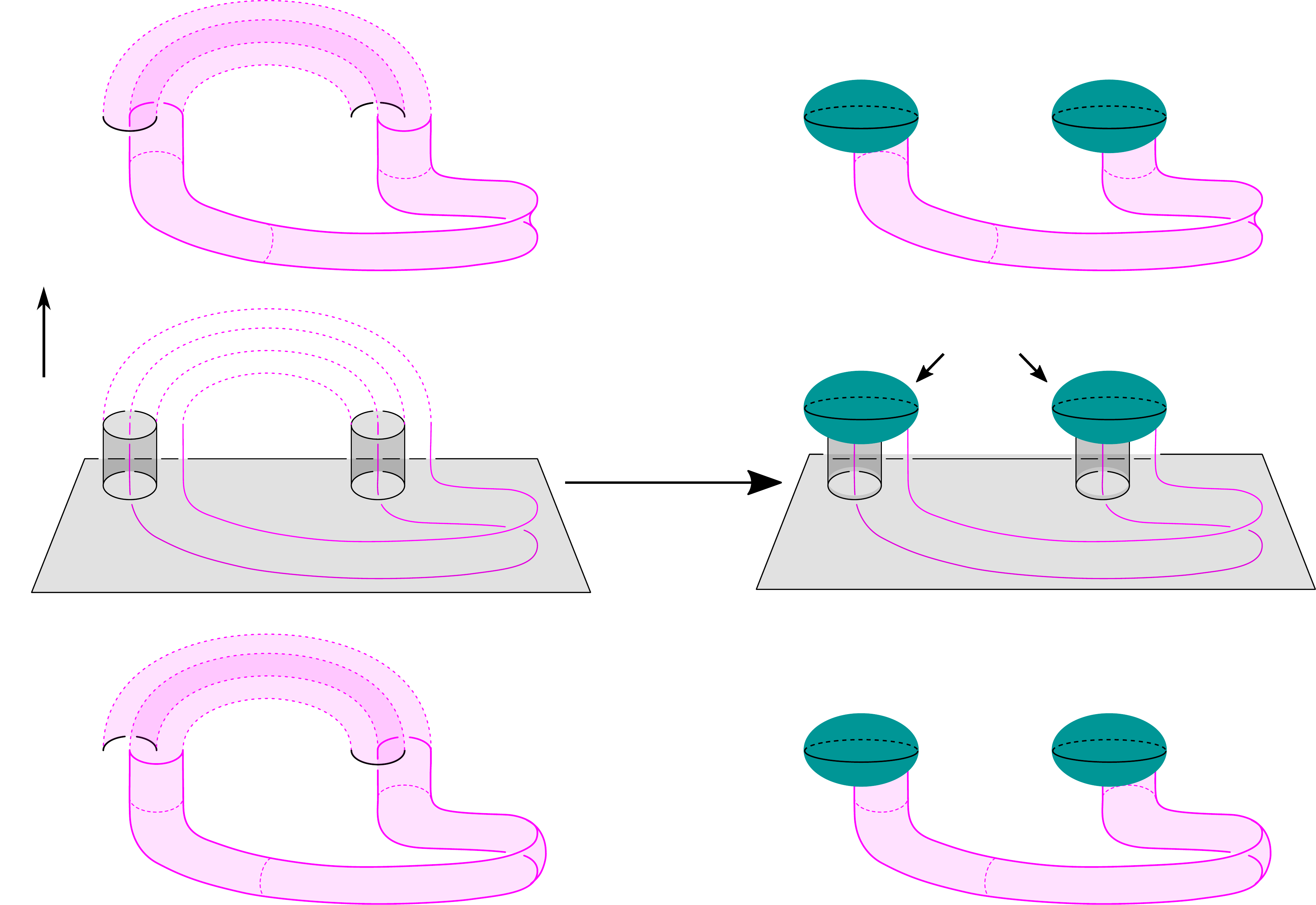}};
  \begin{scope}[x={(image.south east)},y={(image.north west)}]
    \node [anchor=west] at (0.41,0.55) [text width=2cm,align=left] {Cut the double tube};
    \node [anchor=west] at (0.5,0.64) [text width=7cm,align=left] {$4$-balls containing the capped off ends};
    \node [anchor=west] at (0,0.63) {$t$};
    \end{scope}
\end{tikzpicture}
      \caption{\label{fig:diagramfordisc}The surface in a neighbourhood of $\Gamma_i$. On the left we depict part of the double tube; we show how the double tube joins up using the dotted arcs (though this does not happen in this $4$-ball). On the right, we depict part of the slice surface $K$; we do not draw $K$ inside the two 4-balls indicated. Inside these two $4$-balls the surface is given by standard annulus bounded by the Hopf link, whose interior has been pushed into the $4$-ball. To obtain the surface on the right from that on the left, we cut open the double tube as in Figure \ref{fig:cutthetube} and retract the two capped off ends so that they lie near the surface, inside the neighbourhood of $\Gamma_i$.}\
\end{figure}

We show $\widetilde{S_i'}$ and $\widetilde{S}_{i+1}$ are related by a stabilisation and two destabilisations (and isotopy). We depict the immersed surface $S_i'^{\im}$ in a neighbourhood of the Whitney disc in Figure~\ref{fig:crossedwhitney}. The associated tubed surface in this neighbourhood is shown in Figure~\ref{fig:creationofdoubletube}. To obtain the second picture of Figure~\ref{fig:creationofdoubletube} we perform an isotopy which pulls apart the surface (taking the sheets of the surface to where they need to be after the Whitney move), at the expense of creating a \emph{double tube} again using the terminology of Gabai \cite{Gabai2017}, a Hopf link$\times [0,1]$ running through~$X$ between two disc neighbourhoods of the surface. The Hopf link$\times [0,1]$ has four boundary components, two of which are joined to the surface (the top of the green tube and bottom of the pink), while the other two join to the linking annulus of $\Gamma_i$ (the bottom of the green and top of the pink). 

We now perform a stabilisation inside the Hopf link$\times [0,1]$, which can be seen in the movie picture as attaching two bands (provided this is compatible with orientation, if not see below); see Figure~\ref{fig:cutthetube}. The middle picture of this movie is then an unknot, along which we perform a destabilisation to cut the double tube; again see Figure~\ref{fig:cutthetube}. 

The resulting `end' of the double tube can be made by taking an open double tube, adding a band between the two tubes, then capping off the resulting boundary circle with a disc. It can also be thought of as the standard annulus whose boundary is the Hopf link whose interior is pushed into the 4-ball. We now suck back these ends to be close to the surface. The resulting surface differs from $\widetilde{S}_{i+1}$ only in a 4-ball neighbourhood of the arc $\Gamma_i$. The resulting capped off double tube in this 4-ball neighbourhood is pictured on the right in Figure~\ref{fig:diagramfordisc}. In this neighbourhood, we see a genus 1 slice surface for the unknot. We call this slice surface $K$.

If the above stabilisation was not compatible with the orientation on the double tube, we instead perform a different stabilisation and destabilisation which are compatible with the alternative possible orientations; see Figure~\ref{fig:unorientabletubecut}. Again the result is a cut open double tube, and we again suck the ends back to lie in a 4-ball neighbourhood of $\Gamma_i$. We call the resulting surface in this neighbourhood $K'\subset B^4$ and note the boundary of $K'$ is the unknot in $S^3$ as before.
\begin{figure}[h]
   \centering \includegraphics[width=0.92\textwidth]{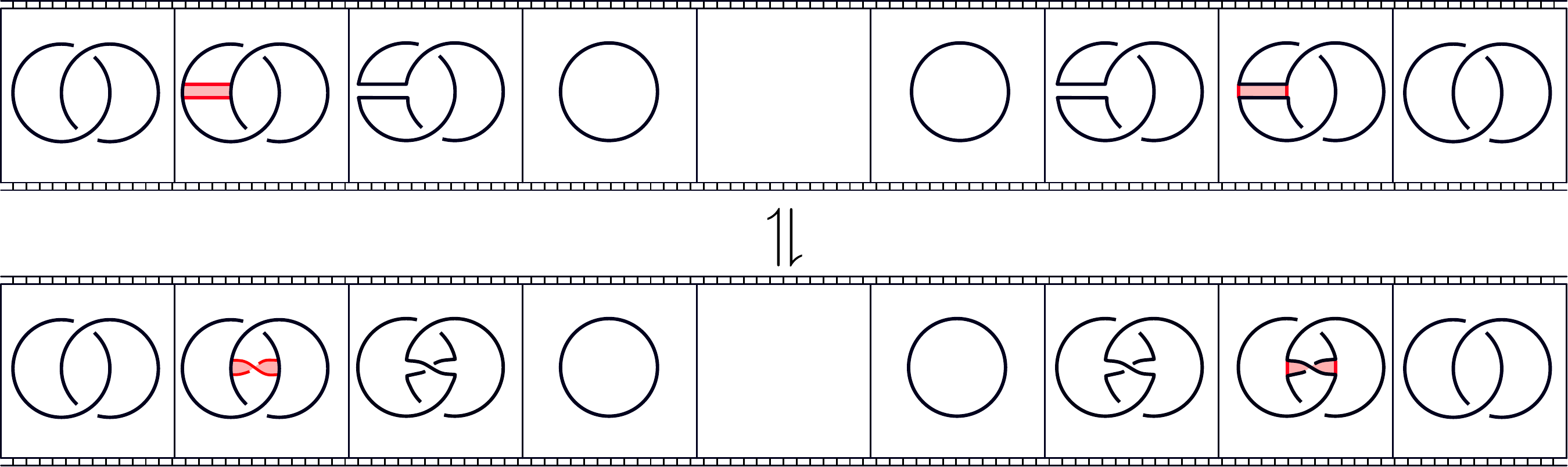}
      \caption{\label{fig:unorientabletubecut}Cutting the double tube using a stabilisation and a destabilisation which are compatible with the other possible orientation of the double tube. We perform an isotopy of the surface to obtain the second picture which shows the result differs by just a twist in the bands.}
   \centering
\end{figure}

We now wish to construct a sequence of stabilisations and destabilisations taking $K$ and $K'$ to the trivial disc. Such a sequence would remove the mess of tubes, and take our surface to $\widetilde{S}_{i+1}$ as required. We exhibit such a sequence in Lemma~\ref{killthe2knot}. Note that results in \cite{InancBaykur2015} imply that some sequence of stabilisations and destabilisations exists but we show that in fact, one destabilisation is sufficient, and so the genus of any intermediate surface does not exceed $g+n+1$.

\textbf{Case 3:}  \textit{The Whitney arcs consist of an arc $\alpha$ between $x_i^+$ and $x_j^-$, and an arc $\beta$ between $y_i^+$ and $y_j^-$ for $i\neq j$.} This case presents the most diagrammatic difficulty. We wish to remove the intersections of arcs with $\alpha$ and $\beta$, however this is made difficult by the fact that $\gamma_i$ and $\gamma_j$ may intersect $\alpha$ and $\beta$ in a complicated way; see Figure~\ref{fig:tubejoiningdiag}.

\begin{figure}[h]
   \begin{tikzpicture}
\node[anchor=south west,inner sep=0] (image) at (0,0) {\includegraphics[width=0.88\textwidth]{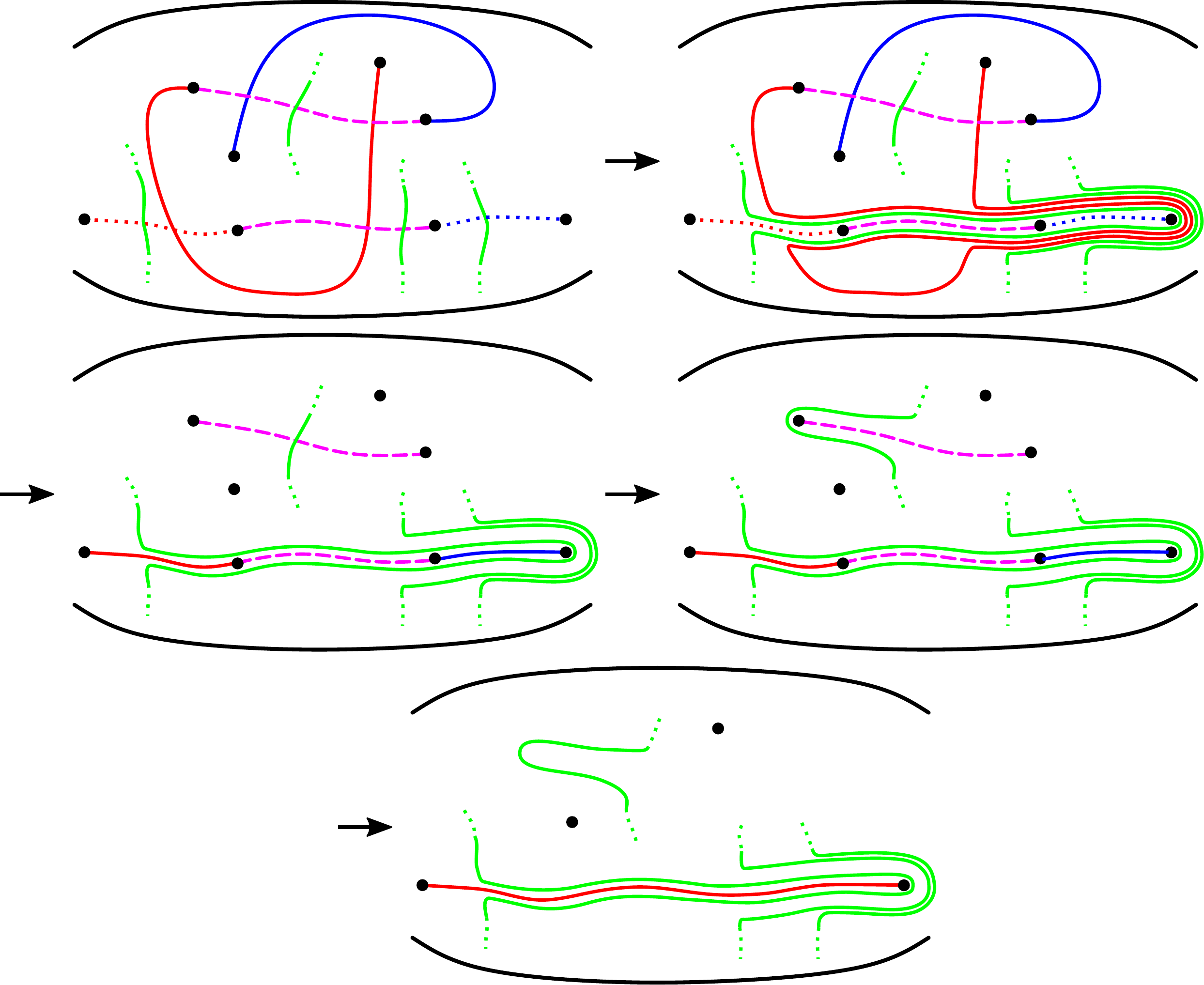}};
  \begin{scope}[x={(image.south east)},y={(image.north west)}]
%  \draw[step=0.1,gray,very thin] (0,0) grid (1,1);

    \node[mpurp] at (0.27,0.8*0.7+0.303) { $\alpha$};
     \node[mpurp] at (0.26,0.705*0.7+0.303) { $\beta$};
    \node[blue] at (0.43,0.875*0.7+0.303) { $\gamma_j$};
    \node[red] at (0.26,0.6*0.7+0.303) { $\gamma_i$};
     \node[blue] at (0.44,0.643*0.7+0.303) { $a_j$};
    \node[red] at (0.103,0.64*0.7+0.303) { $a_i$};
    \node at (0.17,0.903*0.7+0.303) { $x_i^+$};
    \node at (0.3,0.91*0.7+0.303) { $x_i^-$};
    \node at (0.367,0.86*0.7+0.303) { $x_j^-$};
    \node at (0.176,0.766*0.7+0.303) { $x_j^+$};
    
    \node at (0.203,0.625*0.7+0.303) { $y_i^+$};
    \node at (0.07,0.72*0.7+0.303) { $y_i^-$};
    \node at (0.37,0.63*0.7+0.303) { $y_j^-$};
    \node at (0.49,0.71*0.7+0.303) { $y_j^+$};
    \end{scope}
\end{tikzpicture}
      \caption{\label{fig:tubejoiningdiag}The sequence of diagrams corresponding to a Case 3 Whitney move. We first pick arcs $a_i$ and $a_j$ disjoint from $\alpha$ and $\beta$. To obtain the second diagram we use the tube move lemma to remove intersections of any arcs with the arc $a_i\cup\beta\cup a_j$. To obtain the third diagram we use the tube swap lemma twice. We then remove intersections with $\alpha$ using the tube move lemma. Finally, we remove the points removed the Whitney move and join the tubes; the corresponding destabilisation is pictured in Figure~\ref{fig:tubejoiningiso}. Note that the final diagram uses the new immersion data.}
\end{figure}
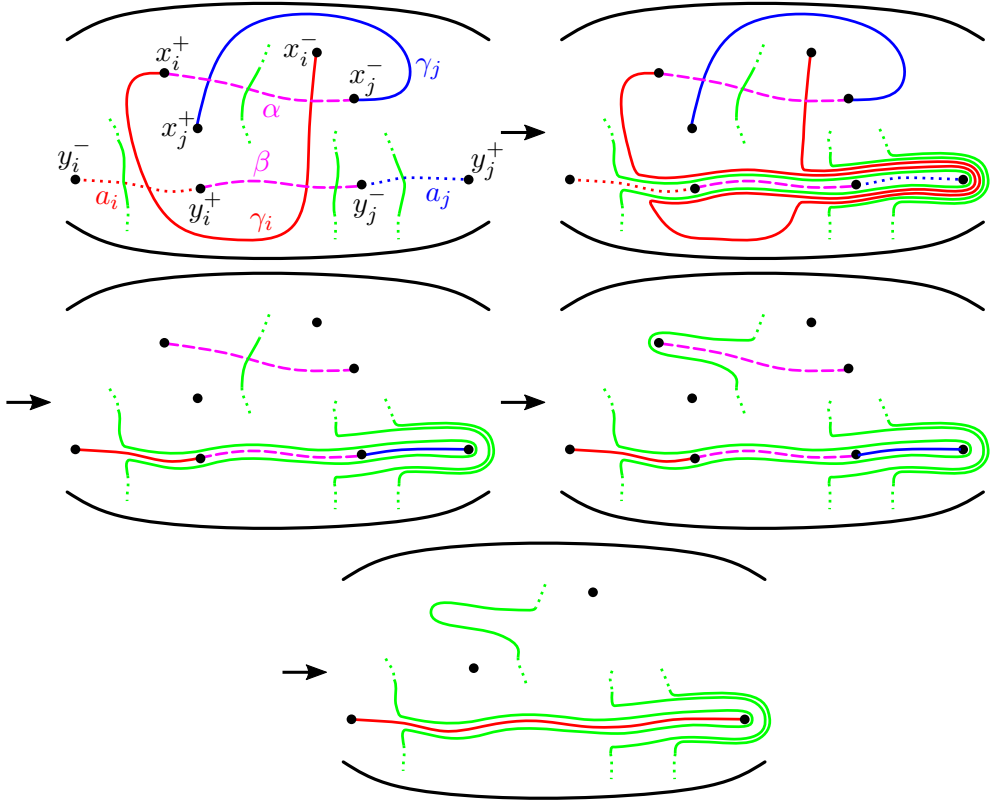

To overcome this, we first we pick an arc $a_i$ from $y_i^+$ to $y_i^-$ which is disjoint from $\alpha\cup\mathring{\beta}$ and all other marked points on the surface; clearly, we may do so since the complement of $\alpha\cup\mathring{\beta}$ and all points is connected. We then similarly pick an arc $a_j$ from $y_j^+$ to $y_j^-$ which is disjoint from~$\alpha\cup\mathring{\beta}\cup a_i$ and all marked points other than $y_j^+$ and $y_j^-$; again the complement of these arcs and points is connected since $a_i$ does not intersect $\alpha$ or $\beta$. We remove the intersections of all arcs with~$a_i \cup \beta\cup a_j$, which is one long embedded arc, one by one using the tube move lemma as in previous cases; see Figure~\ref{fig:tubejoiningdiag}. Note that this may move $\gamma_i$ and $\gamma_j$.

We now use the tube swap lemma to swap $\gamma_i$ to $a_i$ and $\gamma_j$ to $a_j$. We then move any arcs off $\alpha$ using the tube move lemma.

Finally we remove the intersection points $y_i^+$, $x_i^+$, $y_j^-$, $x_j^-$, join the arcs $\gamma_i$ and $\gamma_j$ using~$\beta$, and change the index of the points labelled $j$ to $i$; see Figure~\ref{fig:tubejoiningdiag}. The corresponding isotopy and destabilisation of the associated surface is given by Figure~\ref{fig:tubejoiningiso}.
\begin{figure}[!ht]
   \centering \includegraphics[width=\textwidth]{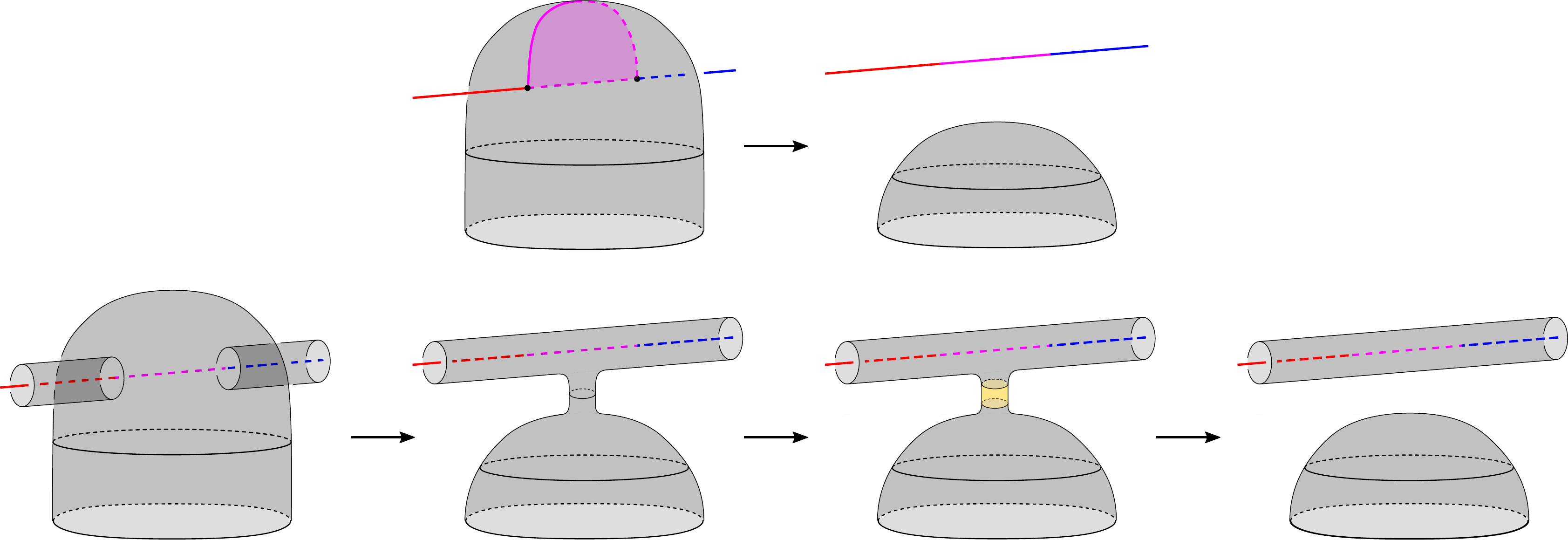}
      \caption{\label{fig:tubejoiningiso}Above we see the Whitney move. Below we see the corresponding isotopy and destabilisation of the associated surfaces. This operation is the operation described by Gabai in \cite[Figure 5.8]{Gabai2017} with an additional destabilisation to remove the `single tube' pictured in \cite[Figure 5.8]{Gabai2017}.}
\end{figure}

\textbf{Case 4:} \textit{The Whitney arcs are an arc $\alpha$ between $x_i^+$ and $y_j^-$, and an arc $\beta$ between~$y_i^+$ and $x_j^-$ for $i\neq j$.} In this case, we use the tube swap lemma to swap $\gamma_i$ to any arc between~$y_i^+$ and $y_i^-$ disjoint from other tube arcs, but which may intersect $\alpha$ and $\beta$; some such arc exists since the complement of $\cup_r\gamma_r\cup_k y_k^\pm$ is a punctured surface, so is connected. We are then in Case 3 and proceed as before.

At stage $P_k$, we have a surface diagram $\mathcal{S}_k$ with immersion data $f_k\rightarrow X$. Since $f_k$ is an embedding it has no intersection points so must be the empty diagram, hence the associated tubed surface $\widetilde{S}_k$ is $f_k(S)=P_k$ as required. This completes the proof of Theorem \ref{ineq} modulo Lemma~\ref{killthe2knot}.
\end{proof}

Examining the proof, we in fact prove a stronger, if more technical, fact.
\begin{proposition}
Let $\Sigma$ and $\Sigma'$ be immersed surfaces with $|\sing(\Sigma)|=|\sing(\Sigma')|$, both with the same number of positive and negative double points. Suppose they are regularly homotopic through surfaces with at most $2n$ double points. Then given a surface tubing diagram $\mathcal{S}$ for $\Sigma$, there exists a surface tubing diagram $\mathcal{S}'$ for $\Sigma'$ such that
$$d_{\st} (\widetilde{S},\widetilde{S}')\leq n-\tfrac{1}{2}|\sing(\Sigma)|+1.$$
\end{proposition}
\noindent Note that taking both $\Sigma$ and $\Sigma'$ to be embeddings yields Theorem~\ref{ineq}, since the only surface tubing diagrams for embeddings are the empty diagram.
%\item We can extend the definition of $d_{\sing} (\Sigma,\Sigma')$ to immersed surfaces $\Sigma$ and $\Sigma'$ with $|\sing(\Sigma)|=|\sing(\Sigma')|$, by defining it to be 
%$$d_{\sing}(\Sigma,\Sigma')=\tfrac{1}{2} \min_{H} \max_{t\in [0,1]} |\sing(H_t)|-\frac{1}{2}\sing(\Sigma),$$
%where the minimum is taken over all regular homotopies $H$ from $\Sigma$ to $\Sigma'$. Then we can rephrase our theorem as defining it as   we obtain $d_{\st} (\widetilde{S},\widetilde{S}')\leq d_{\sing} (\Sigma,\Sigma')+1$
\begin{remark}\leavevmode
\begin{enumerate}[(1)]
\item Schwartz \cite{Schwartz2018} constructs pairs of embedded spheres in a 4-manifold $X$ with 2-torsion in $\pi_1(X)$, which are regularly homotopic, and are such that any regular homotopy between them must contain a crossed Whitney move. 
\item For each finger move we made a choice of Whitney arc to tube along; we tubed along $\beta$, but we could equally have tubed along $\alpha$. In the absence of a crossed disc we may make this choice so that only Case 1 and Case 3 Whitney moves occur. Indeed considering a homotopy as a map $H\colon S\times [0,1]\to X\times [0,1]$, the set of double points is a union of circles. Each double point circle $C$ has two disjoint circles as its preimage, $H^{-1}(C)=C_x\cup C_y$. Labelling the double point preimages so that $x_i^\pm\in C_x$ and $y_i^\pm\in C_y$ gives such a choice of tubing.
\label{improvremark}
\end{enumerate}
\end{remark}
\section{Destabilising the Slice Surfaces $K$ and $K'$}\label{sec:destabilising}
We complete the proof of Theorem~\ref{ineq} by showing that both $K$ and $K'$ become the standard disc bounded by an unknot in $S^3$ after a single destabilisation.

\subsection{Banded link presentations of knotted surfaces}

First, we review the calculus of banded link presentations for slice discs and 2-knots set out by Jablonowski \cite{Jablonowski2015}.

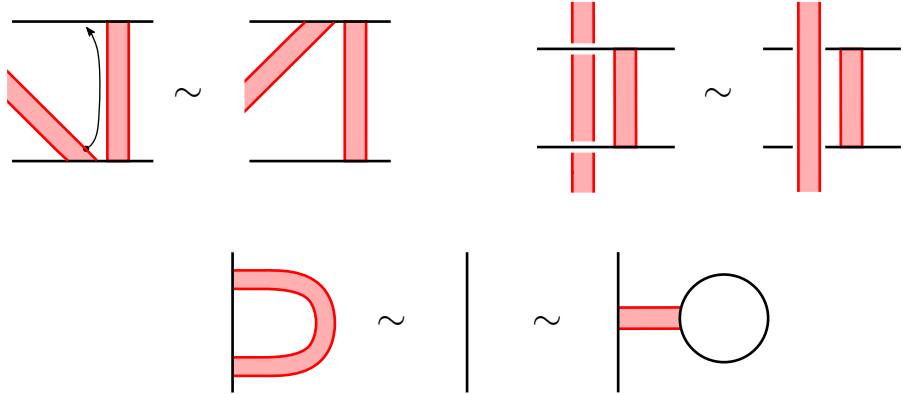
\begin{figure}[h!]
\centering
 
\begin{tikzpicture}
\node[anchor=south west,inner sep=0] (image) at (0,0) {\includegraphics[width=0.8\textwidth]{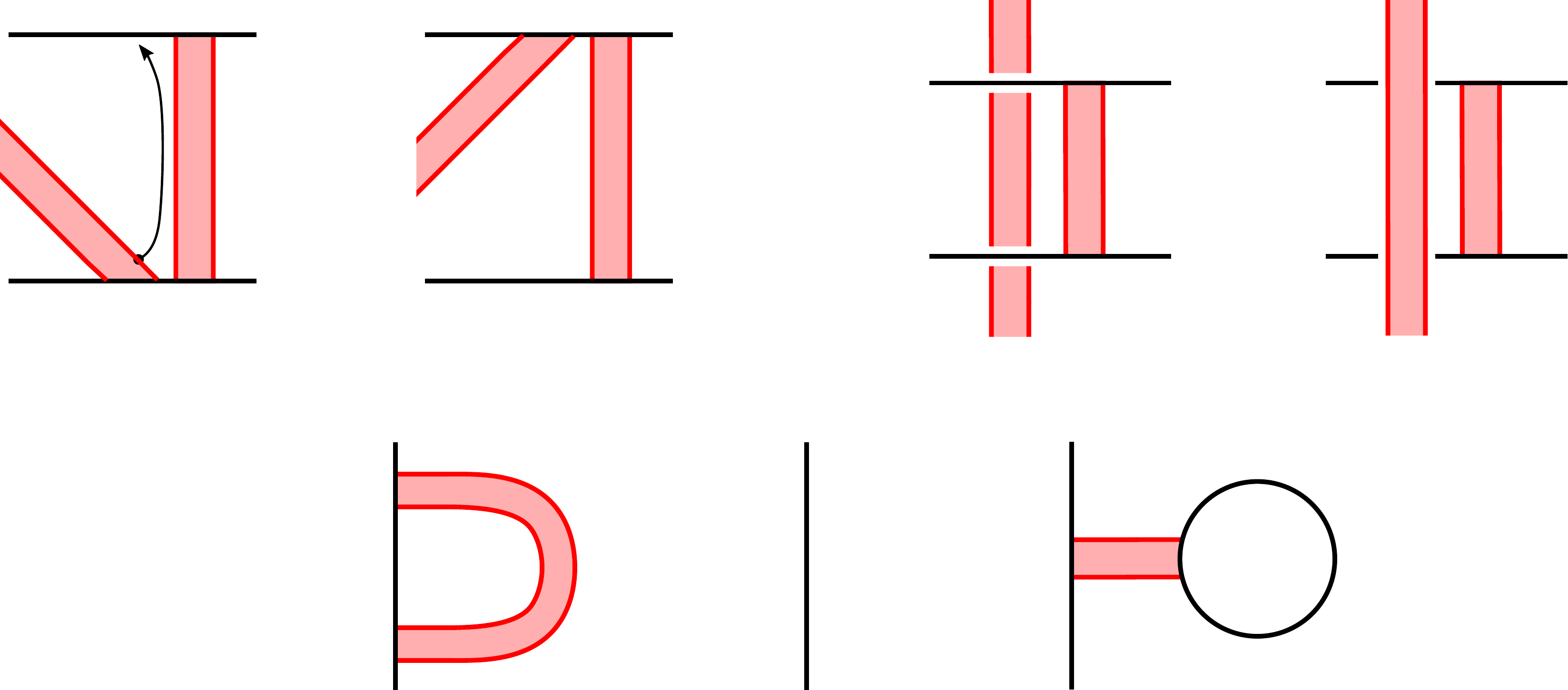}};
  \begin{scope}[x={(image.south east)},y={(image.north west)}]
%  \draw[step=0.1,gray,very thin] (0,0) grid (1,1);
    \node at (0.203,0.76) {\Large $\sim$};
    \node at (0.795,0.76) {\Large $\sim$};
    \node at (0.43,0.18) {\Large $\sim$};
    \node at (0.604,0.18) {\Large $\sim$};
  \end{scope}
\end{tikzpicture}
      \caption{\label{fig:bandmoves}Moves on banded link diagrams giving isotopic surfaces. The top left equivalence is a band slide, the top right equivalence a band swim. The bottom left equivalence is the cancellation of a maximum and a saddle, the bottom right equivalence is the cancellation of a minimum with a saddle.}
\end{figure}

\begin{definition}
A \emph{banded link presentation} of a smooth embedded surface $\Sigma\subset B^4$ bounded by a link $L$ in $S^3$, consists of the link $L'=L\cup U_n$, where $U_n$ denotes the unlink with $n$ components, along with a number of bands, embedded copies $[0,1]\times[0,1]$ disjoint from each other and~$L'$, except at the ends of the bands $\{0,1\}\times[0,1]$ which lie in~$L'$. Performing a \emph{band move} is the operation of removing the ends of the bands $\{0,1\}\times[0,1]$ from~$L'$ and adding in the sides of the bands $[0,1]\times\{0,1\}$. We require that the resulting link after performing all the band moves to $L'$ is the unlink.
\end{definition}

A banded link presentation describes a slice surface for $L$ via a movie which can be seen by considering $B^4$ as $B^3\times [0,1]$. The first slide in this movie is $L$. In the next, we add the unknotted components of $U_n$, these correspond to minima of the surface with respect to the projection onto $[0,1]$. In the next slide we a perform a band move on each band, which corresponds to adding saddles to the surface. After performing these band moves we obtain the unlink $U_{n'}$ for some $n'$. Each component of this unlink is then capped off with a disc, corresponding to maxima of the surface.
\begin{figure}[ht!]
\centering
 \begin{minipage}[c]{0.34\textwidth}
   \includegraphics[width=\textwidth]{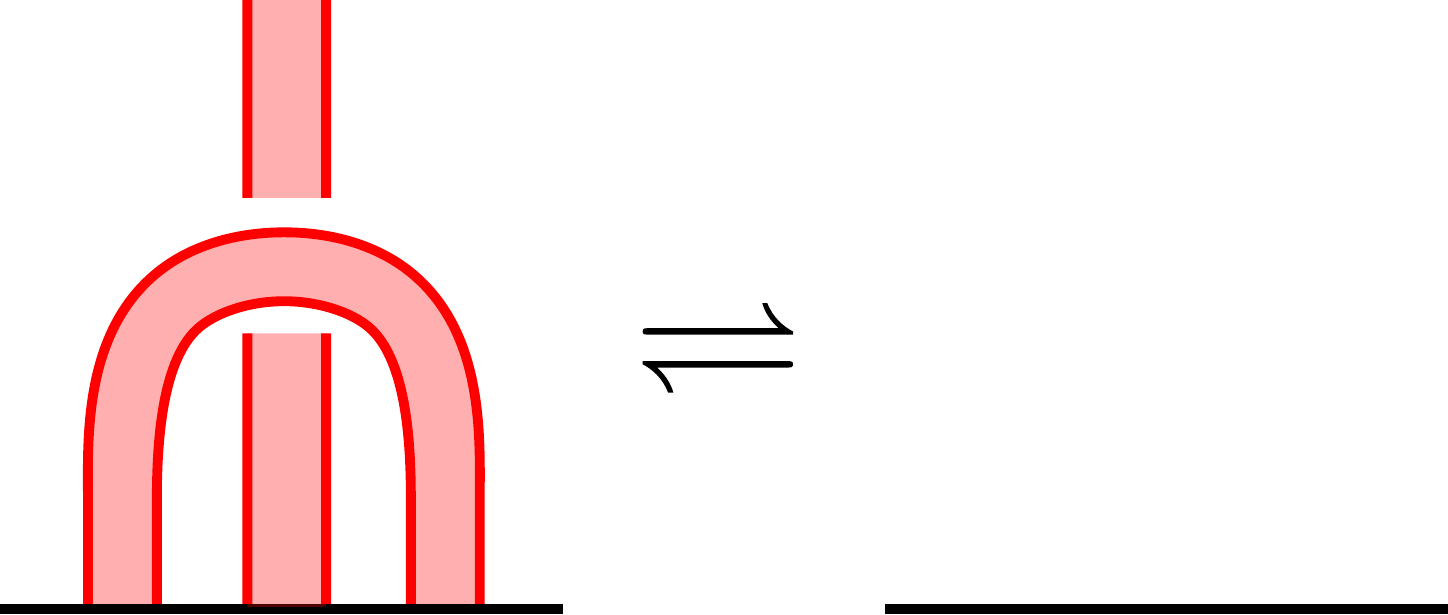}
\end{minipage}%
\begin{minipage}[c]{0.66\textwidth} \caption{Stabilisation and destabilisation in banded link presentations.\label{fig:bandstab}}
\end{minipage}
\end{figure}

There are several moves on banded link diagrams one may perform that give isotopic surfaces. These are isotopy of the diagram, band slides, band swims, cancelling a maximum with a saddle, and cancelling a minimum with a saddle; see Figure~\ref{fig:bandmoves}.

Stabilisation or destabilisation in banded link diagrams corresponds to adding or removing respectively two bands, as in Figure~\ref{fig:bandstab}. Note that it does not matter where the loose end of the band goes.

\subsection{Proof that $K$ and $K'$ destabilise to the standard disc}
\begin{lemma} The surfaces $K$ and $K'$, described in Case 2 in the proof of Theorem~\ref{ineq}, become the standard disc bounded by the unknot in $S^3$ after a single destabilisation.\label{killthe2knot}
\end{lemma}
\begin{proof}
\begin{figure}[!h]
   \centering 
\begin{minipage}[c]{0.5\linewidth}
\begin{tikzpicture}
\node[anchor=south west,inner sep=0] (image) at (0,0) {\includegraphics[width=\textwidth]{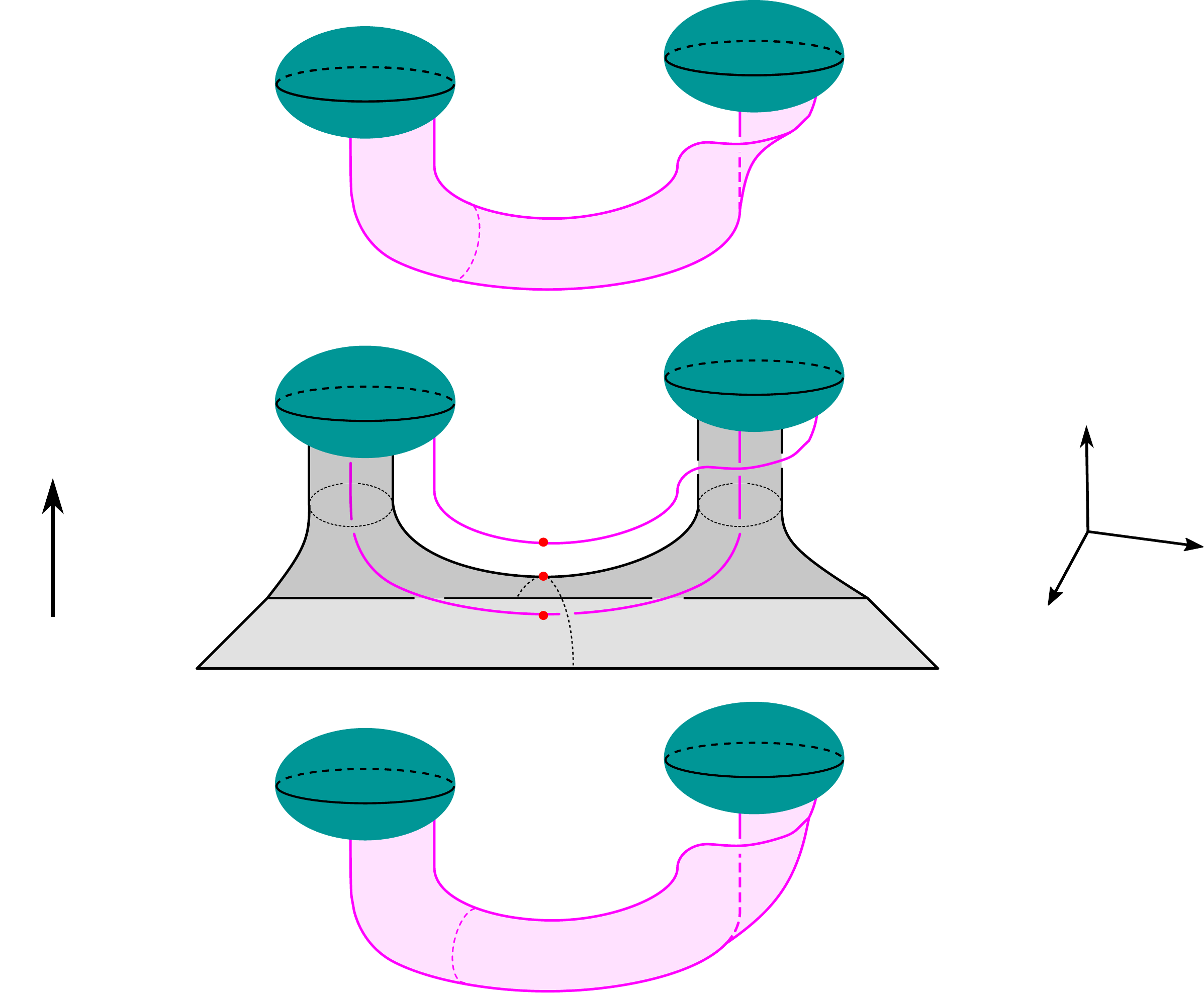}};
  \begin{scope}[x={(image.south east)},y={(image.north west)}]
    
    \node [anchor=west] at (0,0.44) {$t$};
    \node [anchor=west] at (0.875,0.39) {x};
    \node [anchor=west] at (0.955,0.41) {y};
    \node [anchor=west] at (0.9,0.53) {z};
    \end{scope}
\end{tikzpicture}\end{minipage}%
\begin{minipage}[c]{0.49\linewidth}
      \caption{\label{fig:morsedisc}Deformation of $K$, so that the $z$ direction, as depicted, restricts to a Morse function on the surface.}\end{minipage}%
\end{figure}
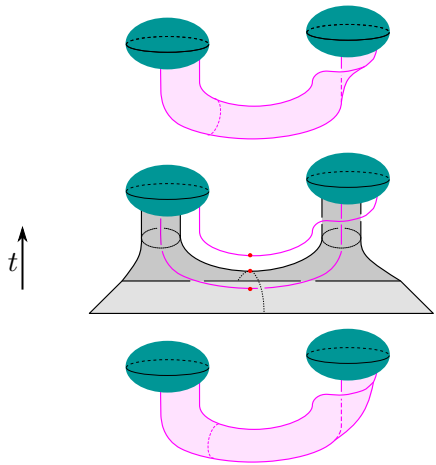

We first consider Figure~\ref{fig:diagramfordisc}. We recall that the `capped-off' ends are made by attaching a band between the tubes and capping off by a disc as in Figures~\ref{fig:cutthetube} and~\ref{fig:unorientabletubecut}. After an isotopy, the $z$-direction as depicted in Figure~\ref{fig:morsedisc} gives the standard Morse function for $B^4$, which restricts to a Morse function on the surface. With this Morse function the surface has one minimum and two saddles, then two further saddles and two maxima from the caps. 

\begin{figure}[p]
   \centering \includegraphics[width=0.94\textwidth]{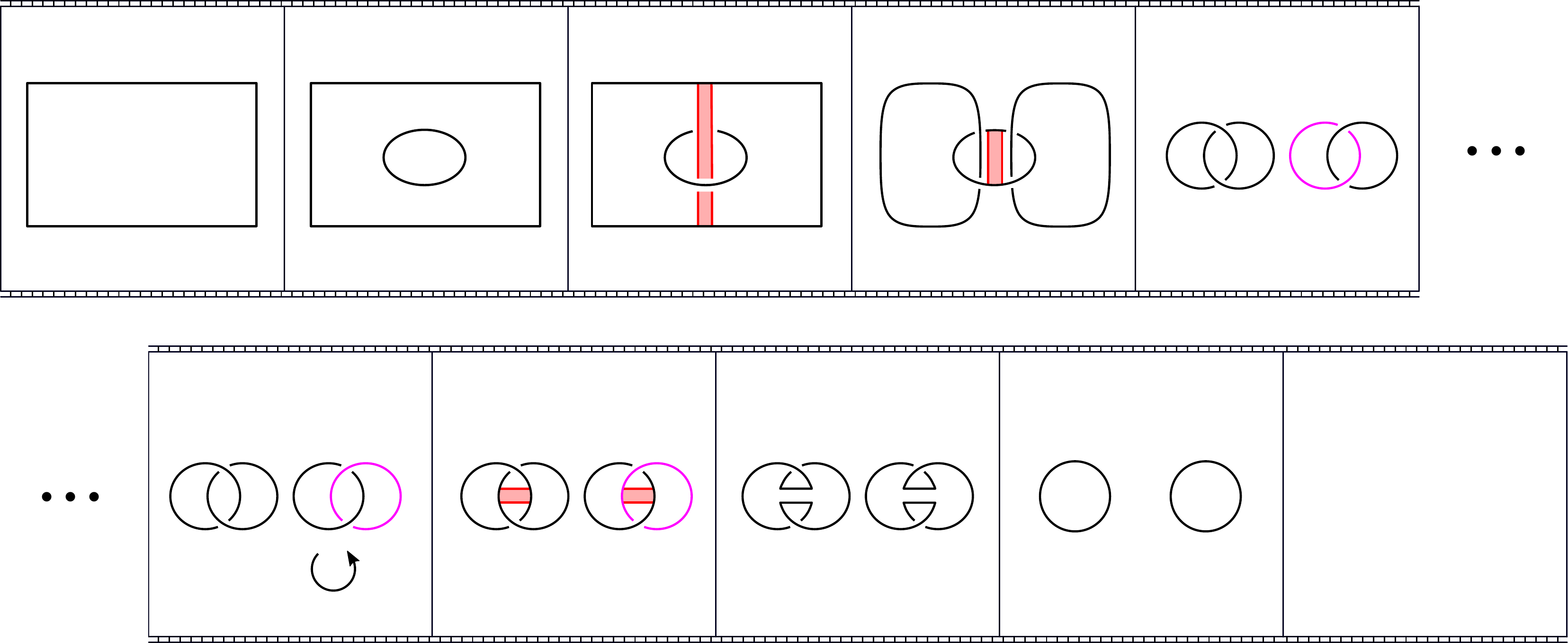}
      \caption{\label{fig:moviefordisc}A movie for the slice surface $K$. Note that the $z$ direction depicted in Figure~\ref{fig:morsedisc}, is now the time direction of our movie.\\}
\smallbreak
\begin{tikzpicture}
\node[anchor=south west,inner sep=0] (image) at (0,0) {\includegraphics[width=\textwidth]{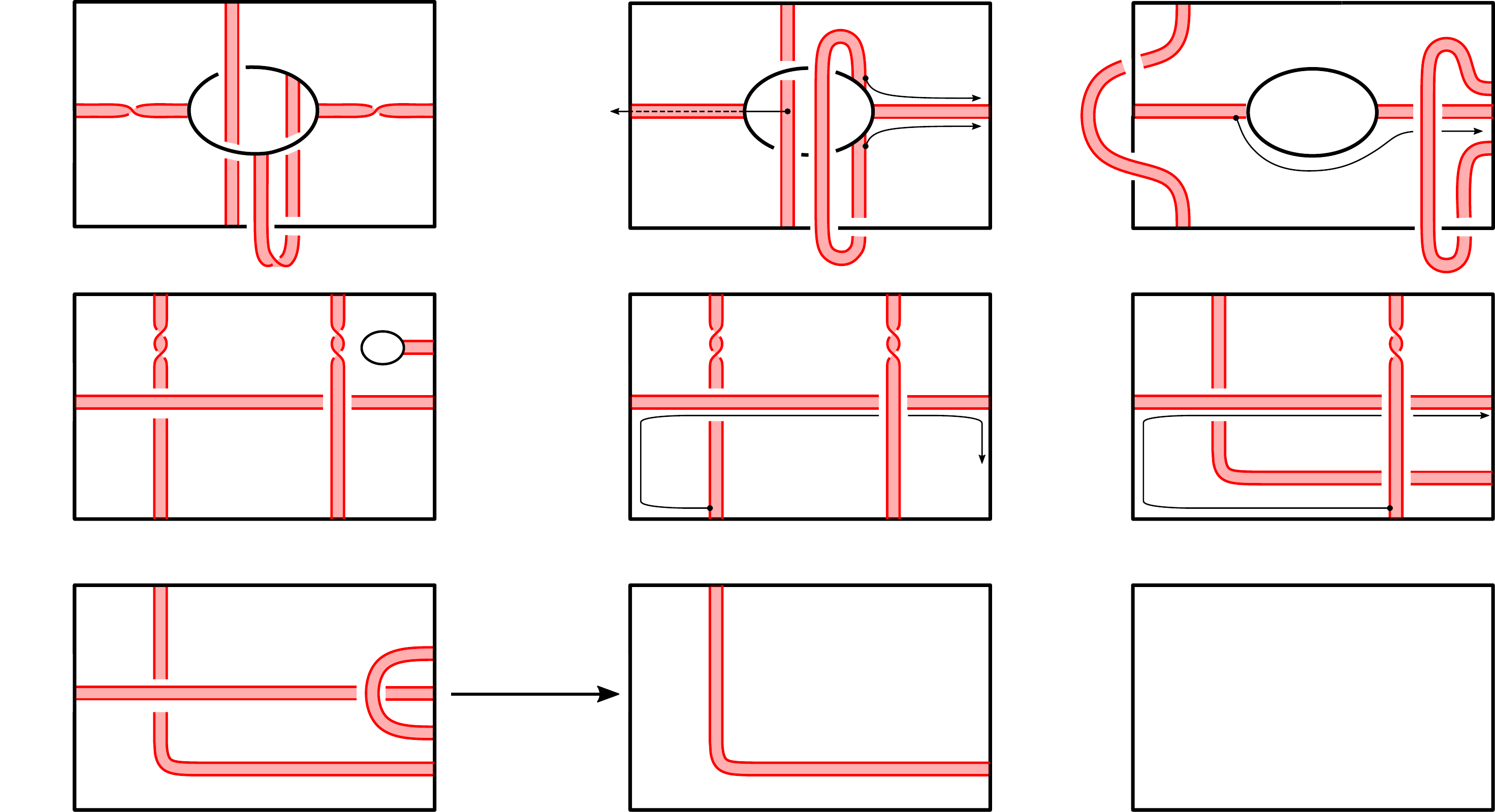}};
  \begin{scope}[x={(image.south east)},y={(image.north west)}]
    \node [anchor=west] at (0.246+0.04,0.186) {\normalfont Destabilise};
    \node [anchor=west] at (0.29+0.04,0.504) {\Large$\sim$};
    \node [anchor=west] at (0.29+0.04,0.84) {\Large$\sim$};
    \node [anchor=west] at (0.653+0.02,0.84) {\Large$\sim$};
    \node [anchor=west] at (0.653+0.02,0.504) {\Large$\sim$};
    \node [anchor=west] at (0.653+0.02,0.15) {\Large$\sim$};
    \node [anchor=west] at (0.0,0.15) {\Large$\sim$};
    \node [anchor=west] at (0.0,0.504) {\Large$\sim$};
    \end{scope}
\end{tikzpicture}
\caption{\label{fig:simplifiyingtheslicedisc}A band presentation for $K$, which we destabilise to obtain the standard disc. To obtain the second image from the first we perform an isotopy, untwisting the two bands at the sides. To obtain the third from the second we perform two band slides and one band swim, as indicated by the arrows. To obtain the fourth we perform a band slide as indicated, and an isotopy of bands. To obtain the fifth we cancel a minimum with a saddle. To obtain the sixth we perform the indicated band slide, we then perform another band slide to obtain the seventh. We then destabilise to obtain the eighth image. Finally we cancel a maximum and a saddle to obtain the banded link presentation which is just the unknot, which is a banded link presentation for the standard disc bounded by the unknot. }
\end{figure}

This Morse function gives a movie presentation for the surface in the 4-ball; see Figure~\ref{fig:moviefordisc}. We deform this into a band presentation, depicted in Figure~\ref{fig:simplifiyingtheslicedisc}, which we simplify using band swims, slides, and a destabilisation, to obtain the standard disc.

In the case of $K'$, recall that we stabilised the outside of one tube with the inside of the other, which has the effect of adding a half twist to the bands; see Figure~\ref{fig:unorientablebands}. After a band swim and isotopy, we obtain the mirror image of Figure~\ref{fig:simplifiyingtheslicedisc} and proceed as before (taking the mirror image of each picture), destabilising to obtain the standard disc.
\end{proof}

\begin{figure}[h]
   \centering 
\begin{minipage}[c]{0.6\linewidth}
\begin{tikzpicture}
\node[anchor=south west,inner sep=0] (image) at (0,0) {\includegraphics[width=\textwidth]{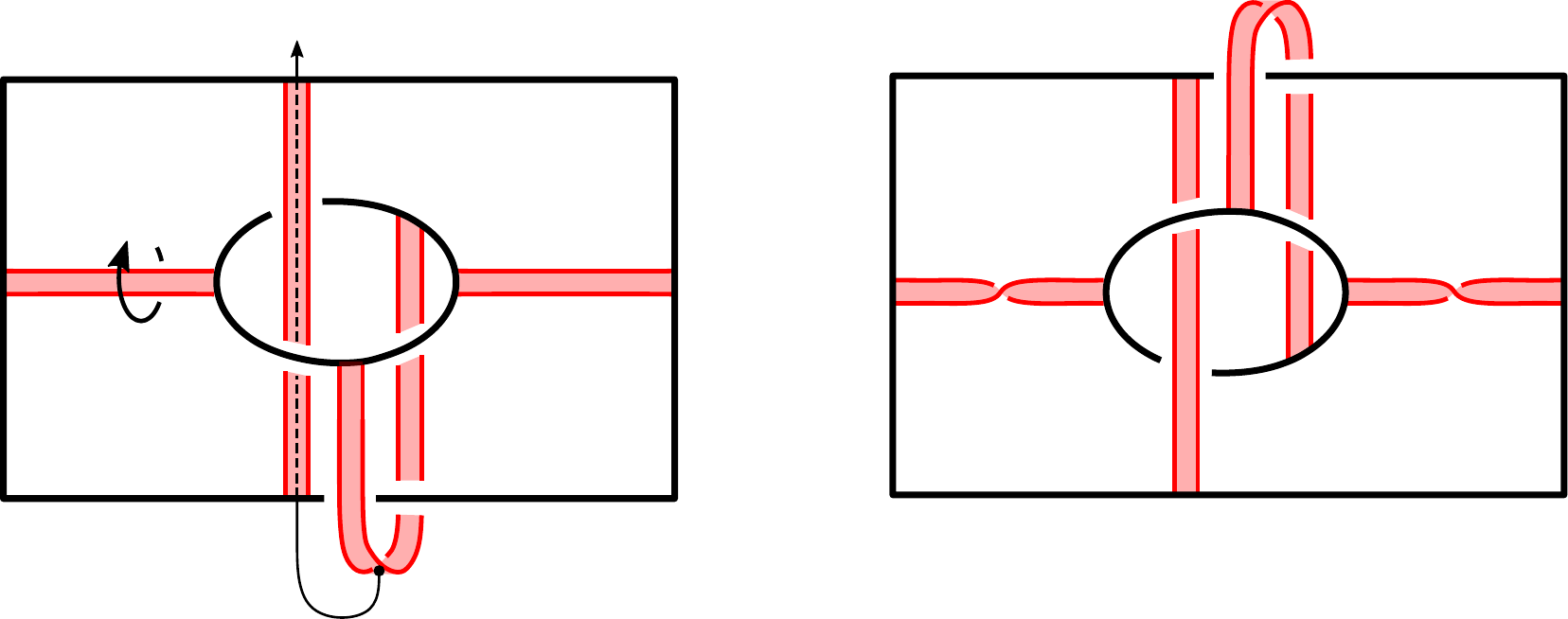}};
  \begin{scope}[x={(image.south east)},y={(image.north west)}]
  \node at (0.494,0.53) {\Large$\sim$};
    \end{scope}
\end{tikzpicture}\end{minipage}%
\begin{minipage}[c]{0.4\linewidth}
\caption{\label{fig:unorientablebands}A band presentation for $K'$, which we see is the mirror image of the band presentation for $K$.}\end{minipage}%
\end{figure}

\end{document}